%% file: main.tex
\documentclass[11pt,a4paper]{article}
\input{headers}

\title{Polynomial growth in degree-dependent first passage percolation on spatial random graphs}

\author{J{\'u}lia Komj{\'a}thy\thanks{Delft University of Technology, j.komjathy@tudelft.nl}, John Lapinskas\thanks{University of Bristol, john.lapinskas@bristol.ac.uk}, Johannes Lengler\thanks{ETH Z{\"u}rich, johannes.lengler@inf.ethz.ch}, Ulysse Schaller\thanks{ETH Z{\"u}rich, ulysse.schaller@inf.ethz.ch U.S. was supported by the Swiss National Science Foundation [grant number 200021\_192079].}}

\begin{document}

\maketitle

\begin{abstract}
In this paper we study a version of (non-Markovian) first passage percolation on graphs, where the transmission time between two connected vertices is non-iid, but increases by a \emph{penalty factor} polynomial in their expected degrees. Based on the exponent of the penalty-polynomial, this makes it increasingly harder to transmit to and from high-degree vertices. This choice is motivated by awareness or time-limitations. For the iid part of the transmission times we allow any nonnegative distribution with regularly varying behaviour at $0$.
For the underlying graph models we choose spatial random graphs that have power-law degree distributions, so that the effect of the penalisation becomes visible: (finite and infinite) Geometric Inhomogeneous Random Graphs, and Scale-Free Percolation. In these spatial models, the connection probability between two vertices depends on their spatial distance and on their expected degrees. 
We prove that upon increasing the penalty exponent, the transmission time between two far away vertices $x,y$ sweeps through four universal phases even for a single underlying graph: 
explosive (tight transmission times), polylogarithmic, \emph{polynomial but sublinear} ($|x-y|^{\eta_0+o(1)}$ for an explicit $\eta_0<1$), and \emph{linear} ($\Theta(|x-y|)$) in their Euclidean distance. Further, none of these phases are restricted to phase boundaries, and those are non-trivial in the main model parameters:  the tail of the degree-distribution, a long-range parameter, and the exponent of regular variation of the iid part of the transmission times.
In this paper we present proofs of lower bounds for the latter two phases and the upper bound for the linear phase. 
These complement the matching upper bounds for the polynomial regime in our companion paper. 
\end{abstract}

\input{intro-ejp}
\input{lower_bounds}
\input{linear_regime}
\input{appendix}

\bibliographystyle{abbrv}	
\bibliography{references}

\end{document}

%% file: headers.tex
\usepackage{a4wide}
\usepackage{amsmath, amsthm, amssymb, amsfonts, thm-restate}
\usepackage{hyperref}
\usepackage{enumerate}
\usepackage{mdwlist}
\usepackage{xparse}
\usepackage{xspace}
\usepackage{cite}
\usepackage{todonotes}
\usepackage{tikz}
\usepackage{bbm}
\usepackage{pgf} 
\usepackage{graphicx}
\usepackage{euscript}
\usetikzlibrary{calc}
\usetikzlibrary{decorations.markings}
\usetikzlibrary{arrows.meta, shapes.misc, positioning,patterns,snakes}
\usepackage{microtype}
\usepackage{multirow}
\usepackage{makecell}
 \definecolor{darkgrey}{gray}{0.35}
\definecolor{grey}{gray}{0.86}
\definecolor{lightgrey}{gray}{0.91}
\definecolor{ballblue}{rgb}{0, 0.5,0.5}
\definecolor{lightballblue}{rgb}{0, 0.8,0.8}
\definecolor{dbblue}{rgb}{0, 0.4,0.4}
\usepackage{caption}
\usepackage[normalem]{ulem}
\newcounter{constant} 
\newcommand{\newconstant}[1]{\refstepcounter{constant}\label{#1}} 
\newcommand{\useconstant}[1]{C_{\ref{#1}}}

\newcommand{\cost}[1]{\mathcal{C}(#1)}

\newtheorem{theorem}{Theorem}[section]
\newtheorem{lemma}[theorem]{Lemma}
\newtheorem{claim}[theorem]{Claim}
\newtheorem{proposition}[theorem]{Proposition}

\newtheorem{corollary}[theorem]{Corollary}

\newtheorem{definition}[theorem]{Definition}
\theoremstyle{definition}
\newtheorem{remark}[theorem]{Remark}
\newtheorem{assumption}[theorem]{Assumption}
\numberwithin{equation}{section}

\newcommand{\N}{\mathbb{N}}
\newcommand{\Z}{\mathbb{Z}}

\newcommand{\R}{\mathbb{R}}

\newcommand{\pr}{\mathbb{P}}
\newcommand{\E}{\mathbb{E}}

\newcommand{\vol}{\textnormal{Vol}}

\newcommand{\eps}{\varepsilon}

\newcommand{\mpar}{\textnormal{\texttt{par}}\xspace}
\newcommand{\lls}{{\,\ll_{\star}\,}}
\newcommand{\ggs}{{\,\gg_{\star}\,}}

\newcommand{\calA}{\mathcal{A}}
\newcommand{\calB}{\mathcal{B}}
\newcommand{\calC}{\mathcal{C}}

\newcommand{\calE}{\mathcal{E}}
\newcommand{\calF}{\mathcal{F}}

\newcommand{\calH}{\mathcal{H}}

\newcommand{\calS}{\mathcal{S}}

\newcommand{\calV}{\mathcal{V}}
\newcommand{\calW}{\mathcal{W}}
\newcommand{\calX}{\mathcal{X}}

\tikzset{cross/.style={cross out, draw=black, minimum size=2*(#1-\pgflinewidth), inner sep=0pt, outer sep=0pt},
cross/.default={2pt}}

%% file: intro-ejp.tex
\section{Introduction}\label{sec:intro}
First passage percolation (FPP) is a natural way to understand geodesics in random metric spaces.
In this paper, we investigate a natural extension of Markovian and non-Markovian first passage percolation, where the transmission time of the first passage percolation process across an edge depends on its direct surroundings in the underlying graph, in particular on the expected degrees of the sending and receiving vertex \cite{komjathy2020stopping}. We set the transmission time through each edge $uv$ as $L_{uv} (W_u W_v)^\mu$, where $L_{uv}$ is a non-negative iid factor with distribution that varies regularly near $0$ with exponent $\beta$, $W_u, W_v$ are constant multiples of the expected degrees of the vertices $u,v$, and $\mu$ is the penalty exponent. We then let the process run according to the standard rules of first passage percolation: transmission time between $x$ and $y$ is the the total sum of transmission time on edges on a path between $x,y$, minimised over all paths between $x$ and $y$.
We abbreviate this process as $1$-FPP for short. For the underlying graph models we choose spatial models with power-law degree distributions, i.e., with highly varying degrees, so that the dependence of the transmissions on the local surroundings causes non-negligible effects. The dependence caused via setting $\mu\ge 0$ is so that the transmission time from and towards high-degree vertices is slowed down by a polynomial of their expected degrees, which makes it harder, but not impossible, to transmit to and from ``superspreaders''.
The 1-FPP model is inspired by degree-dependent bond percolation~\cite{hooyberghs2010biased} and by topology-biased random walks~\cite{bonaventura2014characteristic, ding2018centrality, lee2009centrality,pu2015epidemic,zlatic2010topologically}, in which the transition probabilities from a vertex depend on the degrees of its neighbours.  Those works also assume a polynomial dependence on the degrees. 

This new one-dependent FPP process has several interesting features. It reveals topological features of the underlying graphs hidden from classical versions. Classical FPP tend to strongly depend on the highest degree vertices and shows the `explosion' phenomenon on spatial graphs with highly varying degrees for \emph{all} regularly varying transmission-time distributions, i.e., the process reaches infinitely many vertices in finite time \cite{komjathy2020explosion}. However, with $1$-FPP  explosion can be stopped efficiently by increasing the exponent $\mu$ of the penalty polynomial above a given value~\cite{komjathy2020stopping}, thus, only changing the dynamics of the process to/from outlier vertices but keeping the underlying graph fully intact. We study what happens in the sub-explosive regime. In a sequence of papers we prove that as the penalty exponent increases, the transmission time between two far away vertices $x,y$ in the infinite component sweeps through four universal phases:
\begin{enumerate}
\item [(i)] it converges to a limiting a.s. finite random variable, i.e., it grows \emph{explosively}. 
\item[(ii)] it grows \emph{polylogarithmically} with the Euclidean distance as $\sim (\log |x-y|)^{\Delta_0+o(1)}$ for some $\Delta_0\ge 1$);
\end{enumerate}
Phase (i) was our previous work \cite{komjathy2020stopping}. We treat phase (ii) and the upper bound for phase (iii) in the companion paper \cite{komjathy2022one1}. The focus of this paper is to prove the \emph{lower bounds} to the following phases (iii)-(iv); and the \emph{upper bound to phase (iv)}: 
\begin{enumerate}
\item[(iii)] it grows \emph{polynomially} as $|x-y|^{\eta_0+o(1)}$ for an explicitly given $\eta_0<1$;
\item[(iv)] it grows \emph{linearly} proportionally to $|x-y|$.
\end{enumerate}
See also Table \ref{table:summary}. Possibly the most interesting phase here is (iii): sublinear polynomial transmission times imply polynomial intrinsic ball-growth \emph{faster than the dimension} of the underlying space. This is rare in spatial graph models, and shows that 1-FPP cannot simply be mapped and studied as a simpler process (like graph distances) on a spatial graph with  different parameters. Indeed, in this paper we show that in the polynomial regime \emph{all} long edges around $x,y$ carry polynomial transmission times in the distance they bridge; thus the shortest path between $x,y$ must contain edges of polynomial cost, resulting in the lower bound $|x-y|^{\eta_0-o(1)}$ with an explicit $\eta_0<1$. This is hand-in-hand with the constructive proof we give as a matching upper bound in \cite{komjathy2022one1}, thus our result in phase (iii) is sharp. The $o(1)$ additive factor in the exponent is due to our current quite general assumptions on the iid factor $L$ and on the connection probability, but could be possibly sharpened by imposing stronger assumptions there. 
\begin{table}[t]
\begin{center}
    \begin{tabular}{|c|c|c|}
    \hline
Graph param. & 1-FPP parameters &  Behaviour of 1-FPP transmission times \\
        \hline
        \hline
        \multirowcell{2}{\textbf{Weak decay:}\\ $\tau\in(2,3)$\\$\alpha \in(1,2)$} & 
            $\mu < \frac{3-\tau}{2\beta}$
            &
            \makecell{\textbf{Explosive}:\\ $d_{\calC}(0,x) = \Theta(1)$}\\
            \cline{2-3}
            &
            $\mu > \frac{3-\tau}{2\beta}$
            &
            \makecell{\textbf{Polylogarithmic}:\\ $d_{\calC}(0,x) = (\log|x|)^{\Delta_0+o(1)}, \Delta_0>1$}\\
        \hline\hline
        \multirowcell{4}{\textbf{Strong decay:}\\ $\tau\in(2,3)$\\$\alpha > 2$ } &
            $\mu < \frac{3-\tau}{2\beta}$
            &
            \makecell{\textbf{Explosive}:\\ $d_{\calC}(0,x) = \Theta(1)$ }\\
            \cline{2-3}
        &
            $\mu\in\big(\frac{3-\tau}{2\beta}, \frac{3-\tau}{\beta}\big)$
            &
            \makecell{\textbf{Polylogarithmic}:\\ $d_{\calC}(0,x) = (\log|x|)^{\Delta_0+o(1)}, \Delta_0
            >1$} \\
            \cline{2-3}
        &
            $\mu\in\big(\frac{3-\tau}{\beta}, \frac{3-\tau}{\min\{\beta, d(\alpha-2)\}} + \frac{1}{d}\big) $
            &
            \makecell{\textbf{Strict Polynomial}:\\ $d_{\calC}(0,x) = |x|^{\eta_0\pm o(1)}, \eta_0<1$} \\
            \cline{2-3}
        &
            $\mu > \frac{3-\tau}{\min\{\beta, d(\alpha-2)\}} + \frac{1}{d} $
            &
            \makecell{\textbf{Linear:}\\ $d\ge 2\ : \ \ d_{\calC}(0,x) = \Theta(|x|)$ \\
           $d= 1: \kappa |x| \le  d_{\calC}(0,x) \le |x|^{1+o(1)}$  } \\
              \hline\hline
       \makecell{\textbf{Non-scale-free:}\\ $\tau>3, \alpha>2$} & 
            $\pr(L>0)=1, \ \mu\ge 0$
            &
            \makecell{\textbf{Linear}:\\ $d_{\calC}(0,x) \ge \kappa |x|$}\\
            \hline
    \end{tabular}
\end{center}
     \caption{\normalfont \small{Summary of our results. In 1-FPP, the transmission time on edge $xy$ is $L_{xy} (W_xW_y)^\mu$ where $W_x, W_y$ are constant multiples of the expected degrees of the vertices $x,y$, and $L_{xy}$ is an iid random variable with a regularly varying distribution function near $0$ with exponent $\beta\in(0,\infty]$. The parameter $\tau\in(2,3)$ is the exponent of the power-law degree distribution, the underlying graph has doubly-logarithmic graph distances. The transmission time $d_\calC(0,x)$ between $0$ and a far vertex $x$ passes through four different phases as $\mu$,  the penalty exponent,  increases. When the long-range parameter $\alpha\in(1,2)$, long edges between low-degree vertices cause polylogarithmic transmission times, so increasing $\mu$ stops explosion but it has no further effect. When $\alpha>2$, these edges are absent from the graph. Here there are three sub-explosive phases. In this paper we focus on the polynomial phases: we prove the lower bound for the strict polynomial phase, (with an explicit $\eta_0$ in~\eqref{eq:eta_0}). In particular, all long edges near $0, x$ have polynomial transmission times in the distance they bridge. We also prove both the upper and lower bound for the linear distance regime, and the lower bound much more generally when $\tau>3$. All regimes hold on proper intervals for $\mu$.}}
\label{table:summary}
\end{table}
\vskip0.5em

In phase (iv), we prove that the transmission time between two vertices $x$ and $y$ is between $\kappa_1 |x-y|$ and $\kappa_2 |x-y|$ for two constants $\kappa_1,\kappa_2>0$. The lower bound is valid in all dimensions and the upper bound valid in dimension at least $2$. In dimension $1$ we prove an upper bound of $|x-y|^{1+o(1)}$ in the accompanying paper \cite{komjathy2022one1}. Nevertheless, near-linear distances in dimension $1$ are somewhat surprising, since generally long-range spatial models either do not percolate in dimension $1$ or show shorter distances \cite{biskup2004scaling, trapman2010growth, deijfen2013scale, schulman1983long}.

An important message of our results is that the change in the phases can be obtained by \emph{changing only the dynamics} of the process, (i.e., increasing the penalty exponent $\mu$), while keeping the underlying graph intact. This rich behaviour arises despite the fact that the underlying graphs have doubly-logarithmic graph distances.
\vskip0.5em

\textbf{Growth phases of FPP and graph distances in other models.}\label{paragraph:FPP-other-graphs}
By contrast, in other models studied so far the behaviour of intrinsic distances is less rich, and the strict polynomial \emph{phase (iii) is absent} or restricted only to phase transition boundaries.
Sparse spatial graphs with finite-variance degrees (e.g. percolation, long-range percolation, random geometric graphs etc.) generally show linear graph distances/transmission times (phase (iv)) when long edges are absent \cite{antal1996chemical, auffinger201750, penrose2003random, cox1981some}, or polylogarithmic distances (phase (ii)) when long edges are present \cite{biskup2004scaling, biskup2019sharp, hao2021graph}. 
In spatial models with infinite-variance degrees, classical FPP explodes for all edge transmission-time distributions $L_{xy}$ with regularly varying behaviour near $0$. For the process to be sub-explosive, $L$ has to be at least doubly-exponentially flat near $0$, and then transmission times are at most the same order as the doubly-logarithmic graph distances \cite{van2017explosion, komjathy2020explosion}; in particular, there is no analogue of phases (ii)--(iv). 
The same happens on non-spatial models with infinite variance degrees \cite{adriaans2018weighted, jorritsma2020weighted}.
On sparse \emph{non-spatial} graph models with finite-variance degrees, classical FPP universally shows Malthusian behaviour, i.e., exponential growth \cite{bhamidi2017universality}, (for any distribution $L>0$ almost surely). Transmission times between two uniformly chosen vertices are then logarithmic in the graph size (phase (ii)).
There is some work on one-dependent FPP on non-spatial graphs: the process either explodes \cite{slangen19}, with the same criterion for explosion as for spatial graphs in \cite{komjathy2020stopping}, or becomes Malthusian \cite{Fransson1720143}. So only phases (i) and (ii) can occur. 

We generally see that the strict polynomial phase is missing from models. So far, the only known graph model that exhibits this behaviour is long-range percolation on its phase boundary $\alpha=2$, which is recently proven to show polynomial distances \cite{baumler2023distances}.
Even among degenerate models (i.e., with a complete graph as the underlying graph), long-range first passage percolation \cite{chatterjee2016multiple} is the only other model where a similarly rich set of phases is proven to occur. Long-range FPP uses exponential transmission times that depend on the Euclidean distance of the edge. The Markov property and the fact that there is no fixed underlying graph makes proof techniques in \cite{chatterjee2016multiple} unavailable in the context of 1-FPP on fixed graphs. We refer the reader to the discussion in~\cite{komjathy2022one1} for details, and also for a thorough discussion of other related literature. 

Summarising, one-dependent FPP is the first process that displays a full interpolation between the four phases on a \emph{single non-degenerate graph model}. Moreover, the phase boundaries for one-dependent FPP depend non-trivially on the main model parameters: the degree power-law exponent $\tau$, the parameter $\alpha$ controlling the prevalence of long-range edges, and the behaviour of $L_{xy}$ near $0$ characterised by $\beta$, see Table \ref{table:summary} for our results. We leave the investigation of any phase-boundary cases in the parameter space, i.e., where $\mu$ takes the value that separates two phases of growth, for future work. 

\textbf{Motivation.}
Arguably, $1$-FPP with its variety of regimes reflects spreading processes on real networks better than simpler models (e.g., iid FPP), since the speed of dissemination in social networks can range from extremely fast (e.g., when news or memes spread through social media) to rather slow, geometry-dominated patterns (e.g., as diseases spread through regions and countries), with change only in the dynamics of the process but not (much) in the underlying network. 
 The choice of degree-dependent transmission times is not just motivated by the related work mentioned above, but comes directly from applications. Actual contacts do not scale linearly with network connectivity due to  limited time or awareness~\cite{feldman2017high}, and this type of penalty has been used to model the sublinear impact of superspreaders as a function of contacts~\cite{giuraniuc2006criticality, karsai2006nonequilibrium, miritello2013time}.  
 
\textbf{Two papers, two techniques.}
In our companion paper \cite{komjathy2022one1} we develop methods of a very different flavour that allow us to construct paths for the upper bounds for all sub-explosive phases at once, at the cost of the $o(1)$ factor in the exponent. This construction works in the quenched setting, i.e., with revealed vertex set, then uses a multi-scale argument to construct the polylogarithmic/polynomial path. We think that both proof techniques (the ones there and the ones here) deserve their own exposition, hence we present them as two separate papers.
 
\subsection{Main methodological contributions}\label{sec:outline}

For the lower bounds, we develop an (iterative) renormalisation method that allows us to treat both the strict polynomial as well as the linear regime at once. We take the renormalisation method from Berger~\cite{berger2004lower} as our starting point, who proved a linear lower bound for graph distances in long-range percolation. We have to overcome several issues that occur in 1-FPP but do not occur for graph distances. Firstly, in 1-FPP the edges carry transmission times, (or costs, for short), and costs are neither bounded from below by $1$, nor are independent, because costs out of a vertex carry the same penalty factor coming from the expected degree of the vertex. Secondly, the underlying graphs we consider contain an excessive amount of long edges on all scales, which is due to the presence of very high degree vertices or `hubs', that make graph distances doubly-logarithmic. Thirdly, the proof in \cite{berger2004lower} uses Kingman's subadditive ergodic theorem, while for 1-FPP this is unavailable (especially in the polynomial regime). 

Nevertheless, we are able to set up a renormalisation argument that takes all the above issues into account, in particular, we quantify that long edges \emph{do} exist when $\tau\in(2,3)$ (which is implied by the doubly-logarithmic graph distances), but these edges are typically \emph{very expensive}. Although we follow broadly the proof in~\cite{berger2004lower}, our formulation gives the lower bound directly for all vertices, while~\cite{berger2004lower} only showed it along a sequence of norms for the vertices, and then used Kingman's subadditive ergodic theorem to extend the result to all vertices. We avoid this by improving on the conditions in \cite[Lemma 2]{berger2004lower} for the 1-FPP, which allows us to extend the proof to non-linear regimes, in which Kingman's theorem is not applicable (see Proposition \ref{prop:path_in_good_block} and the proof on page \pageref{proof:linear_polynomial_lower_bound}). See Section \ref{sec:organisation} for a more detailed outline.

\emph{Proof outline for the upper bounds.}\label{proof:idea-upper}
The upper bound in the linear regime in dimension at least two is also proved via renormalisation, here however a single renormalisation step suffices. Namely, for a sufficiently large constant $M$, we introduce an auxiliary graph:  let $G_M=(\calV_M, \calE_M)$ be formed by vertices with weights in the interval $I_M:=[M,2M]$ and edges with edge-cost at most $M^{3\mu}$ among them. The restriction $2M$ on the weights of vertices guarantees that 1-FPP restricted to vertices with weight in $I_M$ have penalty at most $4^{\mu}M^{2\mu}$, and hence a typical edge among vertices in $\calV_M$ is also part of $\calE_M$. The auxiliary graph $G_M$ can be stochastically dominated from below by a random geometric graph, and when $\tau\in(2,3)$, its average degree increases with $M$. 
We then use a boxing argument to renormalise to a site-bond percolation $\omega^\star$, and use a result by Antal and Pisztora~\cite{antal1996chemical} that distances in the infinite component $\calC_\infty^\star$ of $\omega^\star$ scale linearly with the Euclidean distance. We then generalise a local-density result of Deuschel and Pisztora in~\cite{DP-percolation}, saying that the infinite cluster $\calC_\infty^\star$ of $\omega^\star$ comes near every vertex $z\in \Z^d$, to hold also for the infinite component of quite general models of random geometric graph with high edge density (see Definition \ref{def:dense-geometric}), and thus also for $G_M$, see Corollary \ref{cor:dense-subgraph}. These results may be of independent interest. 
The local density result allows us to connect the starting vertices $0,x$ to a nearby vertex in $G_M$ at low cost, without having to deal with dependencies between the selection of the path and the edge-costs.

When $\tau >3$, the edge-density in $G_M$ \emph{decreases} with $M$, which is why infinite geometric inhomogeneous random graphs (IGIRG) on $\R^d$ and scale-free percolation (SFP) on $\Z^d$ with $\tau >3$ are both non-robust under percolation, while with $\tau\in(2,3)$ they are robust. The upper bound thus does not hold for the $\tau>3$ case. 
A proof of linear distances for $\tau>3, \alpha>2$ assuming high edge-density could go along a similar reasoning as above, namely renormalisation. To prove it for arbitrary supercritical edge-density is nontrivial, since the path has to avoid high-degree vertices (because of the penalties in 1-FPP), which are exactly the vertices keeping the infinite component together. We leave this for future work.

For finite models in a box, we additionally need to ensure that the constructed paths stay inside the box of the graph. For this we prove that there exist linear distance paths between any sites $u$ and $v$ in the infinite component which deviate ``little'' from the straight line segment $S_{u_n,v_n}$ (this is implicitly present in \cite{antal1996chemical}, and we prove it for general models of random geometric graphs with high-density, see Lemma \ref{lem:dense-geometric-linear-paths}). Hence, for random vertices $u_n,v_n$ in the giant component of a geometric inhomogeneous random graph (GIRG), there is a.a.s.\ a cheap path $\pi_{u_n,v_n}$ from $u_n$ to $v_n$ in the corresponding IGIRG with small deviation. Since $u_n$ and $v_n$ are random, they are unlikely to be close to the boundary of the box, and hence $\pi_{u_n,v_n}$ is completely contained in the GIRG.

\subsection{Graph Models}\label{sec:graph_model}

We will consider undirected, simple graphs with vertex set $\calV \subseteq \R^d$. We use standard graph notation, which we summarise along with other common terminology in Section~\ref{sec:notation}. 

We consider three random graph models: \emph{Scale-Free Percolation} (SFP), \emph{Infinite Geometric Inhomogeneous Random Graphs} (IGIRG)\footnote{They have also been called EGIRG, where E stands for extended~\cite{komjathy2020explosion}.}, and (finite) \emph{Geometric Inhomogeneous Random Graphs} (GIRG). Since the latter model contains \emph{Hyperbolic Random Graphs} (HypRG) as special case, our results also hold for HypRG. The main difference between the models is the vertex set $\mathcal V$. For SFP, we use $\mathcal V := \mathbb Z^d$, where $d \in \mathbb{N}$. For IGIRG, $\mathcal V$ is given by a Poisson point process on $\mathbb R^d$ of intensity one with respect to the Lebesgue measure. The formal definition is:

\begin{definition}[SFP, IGIRG, GIRG]\label{def:girg}
Let $d\in \N$, $\tau >2$, $\alpha\in(1,\infty)$, and $\overline{c}>\underline{c}>0$. Let $\ell:[1,\infty)\rightarrow(0,\infty)$ be a slowly varying function, and let
 $h:\R^d\times[1,\infty)\times[1,\infty)\rightarrow[0,1]$ be a function satisfying
\begin{align}\label{eq:connection_prob}
	\underline{c}\cdot\min\left\{1,
	\dfrac{w_1w_2}{|x|^d}\right\}^{\alpha}
	\le h(x,w_1,w_2)\le \overline{c}\cdot\min\left\{1,
	\dfrac{w_1w_2}{|x|^d}\right\}^{\alpha}.
\end{align}
We call $d$ the \emph{dimension}, $\tau$ the \emph{power-law exponent}, $\alpha$ the \emph{long-range parameter}, and $h$ the \emph{connection probability}.

For SFP, set $\mathcal V := \mathbb Z^d$, for IGIRG, let $\mathcal V$ be given by a Poisson point process on $\mathbb R^d$ of intensity one.\footnote{If we take an IGIRG and rescale the underlying space $\mathbb R^d$ by a factor $\lambda$, then we obtain a random graph which satisfies all conditions of IGIRGs except that the density of the Poisson point process is $\lambda^{-d}$ instead of one. Thus it is no restriction to assume density one.} For each $x\in\mathcal V$, we draw a \emph{weight} $W_x$ independently from a probability distribution on $[1, \infty)$ satisfying
\begin{equation}\label{eq:power_law}
  F_W(w)=\mathbb{P}( W\le w)= 1-\frac{\ell(w)}{w^{\tau-1}}. 
\end{equation} 
We denote the weight vector $(W_x)_{x\in \mathcal V}$ by $\calW$.  Conditioned on $\mathcal V$ and $\mathcal W$, every edge $xy$ is present independently with probability $h(x-y,W_x,W_y)$. 

A finite GIRG is obtained by restricting an IGIRG to a cube $Q_n$ of volume $n$ centred at $0$.
\end{definition}

For finite GIRG models we are interested in the behaviour as $n\to \infty$. Definition \ref{def:girg} leads to a slightly less general model than those e.g.\ in~\cite{bringmann2019geometric} and~\cite{komjathy2020stopping}. There, the original definition had a different scaling of the geometric space vs connection probabilities. However, the resulting graphs are identical in distribution after rescaling, see~\cite{komjathy2020stopping} for a comparison. Finally,~\cite{bringmann2019geometric} considered the torus topology on the cube, identifying ``left'' and ``right'' boundaries, but this does not make a difference for our results. Next we define $1$-dependent FPP on these graphs.

\begin{definition}[1-dependent first passage percolation (1-FPP)]\label{def:1-FPP}
Let $G(\mathcal V, \mathcal E)$ be a graph where to each vertex $v\in \mathcal V$ there is an associated  vertex-weight $W_v$.  For each edge $xy\in \mathcal E$, draw an iid random variable $L_{xy}$ from distribution $L$, and set the \emph{transmission cost} of an edge $xy$ as 
\begin{equation}\label{eq:cost}
\cost{xy}:=L_{xy}(W_xW_y)^{\mu}    
\end{equation} for a fixed parameter $\mu>0$ called the \emph{penalty strength}. Setting $\cost{xy}$ defines a cost-distance $d_{\mathcal C}$ between the vertices (see  \eqref{eq:cost_distance} below), that we call $1$-dependent first passage percolation. We define the \emph{cost-ball} $B_r^{\mathcal C}(x)$ of radius $r\ge 0$ around a vertex~$x$ to be the set of all vertices to which there is a path of cost at most $r$ from $x$.
\end{definition}
We emphasize that in the case of SFP, IGIRG, GIRG, we use the \emph{same} vertex weights $(W_v)_{v\in \calV}$ for generating the graph as well as for determining the cost of the edges, while in general the definition would allow for different choices as well. 
We usually assume that the cumulative distribution function (cdf) $F_L:[0,\infty)\rightarrow[0,1]$ of $L$ satisfies the following (exceptions of this assumption will be made explicit when applicable):
\begin{assumption}\label{assu:L}
There exist constants $t_0,\,c_1,\,c_2,\,\beta>0$ such that
\begin{align}\label{eq:F_L-condition}
	c_1t^{\beta}\le F_L(t)\le c_2t^{\beta}\mbox{ for all }t\in[0,t_0].
\end{align}
\end{assumption}
 Without much effort, Assumption~\ref{assu:L} could be relaxed to $\lim_{x\to0} \log F_L(x)/\log x=\beta$, but having the stronger form in \eqref{eq:F_L-condition} increases the readability of the paper. For the same reason, we also exclude extensions to $\alpha = \infty$ in Definition \ref{def:girg} and $\beta = \infty$ in \eqref{eq:F_L-condition} from the main body of the paper, and discuss those separately in Section~\ref{sec:threshold}.

We will call the set of parameters 
\begin{equation}\label{eq:parameters}
\mpar := \{d, \tau, \alpha, \mu, \beta, \underline{c}, \overline{c}, c_1, c_2, t_0\}
\end{equation}
the \emph{model parameters}.  We generally take the standpoint that we consider $\mu$ to be the easiest parameter to change, e.g.\ by adjusting behaviour of individuals corresponding to high-degree vertices, increasing $\mu$ means gradually slowing down the spreading process around high-degree vertices. We also phrase our results from this point of view.

\subsection{Results}\label{sec:results}
We now present two phases (polynomial but sublinear, and linear) of the transmission times between two far away nodes. Without loss of generality we fix one of the endpoints as~$0\in \R^d$. For IGIRG and GIRG, this means that we condition on the vertex set $\calV$ containing $0$, i.e., we consider the resulting Palm distribution. Recall the power-law exponent $\tau$ of  vertex-weights from \eqref{eq:power_law}, the long-range parameter $\alpha$ from the connection probability in \eqref{eq:connection_prob}, and $\beta$ of the transmission times from Assumption \ref{assu:L}.
We assume $\tau\in(2,3)$, this ensures that the degree distribution has finite mean but infinite variance, and that graph-distance balls grow doubly-logarithmically \cite{bringmann2016average}.
In the companion paper~\cite{komjathy2022one1} and in \cite{komjathy2020stopping} we show that cost-distances are at most polylogarithmic if $\alpha <2$ or if $\mu < (3-\tau)/\beta$. Thus here we only consider the complementary case (neglecting phase boundaries) that $\alpha > 2$ and $\mu > (3-\tau)/\beta$. We first define the two values that will separate the phases of growth as $\mu$ changes:
\begin{align}\label{eq:mu_pol_log}
    \mu_{\log}:=\frac{3-\tau}{\beta}, \quad \mu_{\mathrm{pol}}:=
    \frac1d+\frac{3-\tau}{\min\{\beta, d(\alpha-2)\}}=:
    \max\{\mu_{\mathrm{\log}} + \tfrac1d, \mu_{\mathrm{pol},\alpha} \}
\end{align}
with $\mu_{\mathrm{pol},\alpha}:=\tfrac{1}{d}+\tfrac{3-\tau}{d(\alpha-2)}=\tfrac{\alpha-(\tau-1)}{d(\alpha-2)}$. These values are thus all positive when $\tau\in(2,3)$ and $\beta>0$, and finite under Assumption \ref{assu:L} since $\beta>0$. When $\alpha>2$ then $\mu_{\mathrm{pol},\alpha}$ is above $1/d$ as well. The values are also well-defined for $\tau=3$, see Theorem \ref{thm:linear_polynomial_lower_bound} below about results with $\tau\ge 3$.
Then we define the polynomial growth exponent. For $\alpha>2, \tau\in(2,3)$, for all $\mu > \mu_{\log}$, let
\begin{align}\label{eq:eta_0}
    \eta_0 = \eta_0(d, \alpha,\tau,\beta,\mu) := \begin{cases}
	1 & \mbox{ if $\mu>\mu_{\mathrm{pol}}$,}\\
	\min\left\{d(\mu-\mu_{\log}), \mu/\mu_{\mathrm{pol},\alpha}\right\} & \mbox{ if $\mu\le\mu_{\mathrm{pol}}$,}
	\end{cases}
\end{align}
and note that $\eta_0>0$ for all $\mu>\mu_{\log}$, and $\eta_0<1$ exactly when $\mu< \mu_{\mathrm{pol}}$ by \eqref{eq:mu_pol_log}. 
We present special cases when $\alpha = \infty$ or $\beta=\infty$ and extensions to $\tau\ge3$ in Section \ref{sec:threshold} below. 
\begin{theorem}[Polynomial Lower Bound]
\label{thm:polynomial_regime}
	Consider $1$-FPP on IGIRG, GIRG, or SFP of Definition \ref{def:girg} satisfying the assumptions given in \eqref{eq:power_law}--\eqref{eq:F_L-condition}.
	Assume that $\alpha>2$, $\tau\in(2,3)$, and $\mu>\mu_{\log}$. Then for any $\eps>0$ almost surely there exists $r >0$ (independent of $n$ in case of finite GIRG) such that 
	\begin{align*}
        \mbox{for all } x\in \calV \mbox{ with } |x|\ge r:	d_{\mathcal{C}}(0,x) \ge |x|^{\eta_0-\eps}.
	\end{align*}
\end{theorem}

Note that $r$ is random, so it depends on the instance of IGIRG. For GIRG, ``independent of $n$'' means that for every $q >0$ there is $r_0 = r_0(q)$ independent of $n$ such that we can satisfy Theorem~\ref{thm:polynomial_regime} with an $r$ such that $\pr(r > r_0) < q$.\footnote{For fixed $q$, the $r_0$ in our proof grows like $\exp(c\cdot \log(1/q)\cdot \log\log(1/q))$ for a constant $c>0$, see Remark~\ref{rem:large_box_good}.} Note that the lower bound holds simultaneously for all vertices $x$ at distance at least $r$ from $0$. 

This means that we also obtain a geometric bound on the location of the cost-ball around $0$. However, the matching upper bounds in~\cite{komjathy2022one1} are not uniform over all $x$ with $\|x\|=r$, so we can only bound the cost-ball in one direction.
We rephrase Theorem \ref{thm:polynomial_regime} in terms of intrinsic ball growth.

\begin{corollary}\label{cor:ball-growth}
    Consider $1$-FPP on IGIRG, GIRG, or SFP of Definition \ref{def:girg} satisfying the assumptions given in \eqref{eq:power_law}--\eqref{eq:F_L-condition} and with $0\in \calV$.
	Assume that $\alpha>2$, $\tau\in(2,3)$, and $\mu>\mu_{\log}$. 
	Then for all $\eps>0$, almost surely there exists $r_0$ such that for all $r\ge r_0$,
	\begin{enumerate}[(a)]
	    \item $\calB_r^{\calC}(0) \subseteq \{x\in \R^d : |x| \le r^{1/\eta_0 +\eps}\}$, and	
	\item $|\calB_r^{\calC}(0)| \le r^{d/\eta_0 +\eps}$.
	\end{enumerate}
\end{corollary}
Part (a) of Corollary~\ref{cor:ball-growth} rephrases Theorem~\ref{thm:polynomial_regime}, and (b) follows from (a) when considering the fact that the number of vertices in a (Euclidean) ball of radius  $r^{1/\eta_0 +\eps}$ around $0$ concentrates around the volume of this ball (uniformly for all $r\ge r_0$), both for IGIRG and SFP.

To provide context, we cite the corresponding upper bound from~\cite{komjathy2022one1}. 
For $\tau \in (2,3)$, for all edge-densities, the infinite component exists, is unique, and has positive density, but not necessarily density one.\footnote{In SFP, there are choices of $h$ that enforce all nearest-neighbour edges to be present, so density one is possible, but not guaranteed.} Hence for an upper bound we need to condition on $0$ and $x$ both being in the infinite component. 
\begin{theorem}[Polynomial Upper Bound~\cite{komjathy2022one1}]
\label{thm:polynomial-upper}
		Consider $1$-FPP on IGIRG or SFP of Definition \ref{def:girg} satisfying the assumptions given in \eqref{eq:power_law}--\eqref{eq:F_L-condition} with $0\in\calV$. Assume that $\alpha>2$, $\tau\in(2,3)$, and $\mu>\mu_{\log}$. Let $\calC_\infty$ be the infinite component. Then for any $\eps>0$,
	\begin{align}\label{eq:polynomial-upper}
	\lim_{|x|\rightarrow\infty}
	\pr\left(d_{\mathcal{C}}(0,x)\le |x|^{\eta_0+\eps} \mid 0, x\in \calC_\infty\right)=1.
	\end{align}
\end{theorem}
In~\cite{komjathy2022one1} we also prove a corresponding theorem for finite GIRGs that we omit here. 
Observe that Theorem~\ref{thm:polynomial-upper} does not hold \emph{simultaneously} for all vertices at a certain norm, rather, the convergence holds in probability. The reason for this is that we did not make any assumptions on the \emph{tail-behaviour} of the distribution of $L$, and hence we cannot exclude (clusters of) vertices with all edges having very large $L$ values so that these are reached much later, violating  the upper bound in \eqref{eq:polynomial-upper}. Hence Theorem~\ref{thm:polynomial-upper} is not strong enough to provide an analogous lower bound to Corollary~\ref{cor:ball-growth}(a), since for that we would need a statement over \emph{all} vertices in distance at most~$r$.
Yet, even though some vertices at norm $r$ might have long cost-distance to $0$, linearly many of them do satisfy $d_{\mathcal{C}}(0,x)\le |x|^{\eta_0+\eps}$, so we still obtain the following corollary when combining Theorems \ref{thm:polynomial_regime} and  \ref{thm:polynomial-upper}:
\begin{corollary}\label{cor:polynomial-total}
    Consider $1$-FPP on IGIRG or SFP of Definition \ref{def:girg} satisfying the assumptions given in \eqref{eq:power_law}--\eqref{eq:F_L-condition} with $0\in\calV$. Assume that $\alpha>2$, $\tau\in(2,3)$, and $\mu>\mu_{\log}$. Let $\calC_\infty$ be the infinite component. Then for any $\eps>0$, 
\begin{align}
\lim_{|x|\rightarrow\infty}
	\pr\left( \left|\frac{\log d_{\mathcal{C}}(0,x)}{\log |x|} - \eta_0\right| \le \eps \mid 0,x \in \calC_\infty \right)&=1, \label{eq:part-a-earlier}\\
 \lim_{r\rightarrow\infty}
	\pr\left(\left|\frac{\log |\calB_r^{\calC}(0)|}{\log r}- \frac{d}{\eta_0}\right| \le \eps \mid 0\in \calC_\infty \right)&=1. \label{eq:part-b-earlier}
\end{align}
\end{corollary}
Note that \eqref{eq:part-b-earlier} follows immediately from \eqref{eq:part-a-earlier} because for the lower bound on $|\calB_r^{\calC}(0)|$ it suffices that a constant fraction of vertices at distance at most $r$ have cost-distance at most $r^{1/\eta_0+\eps}$, which is implied by (the upper bound in) \eqref{eq:part-a-earlier}. Hence we do obtain an absolute value in \eqref{eq:part-b-earlier} within the $\mathbb P$ sign.

Our next result refines Theorem \ref{thm:polynomial_regime} and \eqref{eq:part-a-earlier} in Corollary~\ref{cor:polynomial-total} when $\eta_0=1$ in \eqref{eq:eta_0}. I.e., when $
\mu > \mu_{\mathrm{pol}}$, we can sharpen both upper and lower bounds if the dimension is at least $2$:
\begin{restatable}{theorem}{LinearRegime}
\label{thm:linear_regime}
Consider $1$-FPP on IGIRG or SFP of Definition \ref{def:girg} satisfying the assumptions given in \eqref{eq:power_law}--\eqref{eq:F_L-condition} with $0\in\calV$. Assume that $\alpha>2$, $\tau\in(2,3)$, $\mu>\mu_{\mathrm{pol}}$, and additionally $d\ge 2$. Let $\calC_\infty$ be the infinite component.
	Then there exist constants $\kappa_1,\kappa_2>0$ depending only on the model parameters such that
	\begin{align}\label{eq:linear_regime}
		\lim_{|x|\rightarrow\infty} \pr\left(\kappa_1 |x| < d_{\mathcal{C}}(0,x)<\kappa_2 |x| \ \mid \  \textnormal{$0,x \in \calC_{\infty}$}\right)=1.
	\end{align}
	The lower bound also holds in dimension $1$.
\end{restatable}
The lower bound is actually valid in a more general setting, see Theorem~\ref{thm:linear_polynomial_lower_bound} below. Our proof of the lower bound is a generalisation of  ideas from Berger~\cite{berger2004lower} to the edge-weighted one-dependent setting, and indeed we recover his result on graph-distances in Long-Range Percolation and the extension to Scale-Free Percolation in~\cite{deprez2015inhomogeneous} when we set $\alpha>2,\tau>3,\mu=0, \beta=\infty$. However, we give a proof that avoids Kingman's subadditive ergodic theorem that both papers \cite{berger2004lower, deprez2015inhomogeneous} use.
As a corollary to our proof we obtain the following result.

\begin{corollary}\label{cor:linear-dev} Consider the setting of Theorem \ref{thm:linear_regime}. 
Let $\pi_{0,x}^\star$ be any path between $0$ and $x$ realising the cost-distance $d_{\calC}(0,x)$.
Then there is a constant $\kappa_3>0$ (depending only on the model parameters) such that $\pi_{0,x}^\star$ is contained in $B_{\kappa_3|x|}(0)$ asymptotically almost surely:
\begin{equation}
\lim_{|x|\to \infty}\pr\Big(\pi_{0,x}^\star\subset B_{\kappa_3|x|}(0)\Big)=1.
\end{equation}
\end{corollary}

The statements of Theorem \ref{thm:polynomial_regime} and Corollary \ref{cor:ball-growth} remain valid in finite-sized GIRGs, since finite GIRGs are obtained as subgraphs of IGIRG, and hence distances can only increase in GIRG versus the surrounding infinite model. For the upper bound in Theorem~\ref{thm:linear_regime}, the extension to finite GIRGs requires a proof:
\begin{restatable}{theorem}{FiniteGraph}
\label{thm:finite_graph}
Consider GIRG of Definition \ref{def:girg} satisfying the assumptions given in \eqref{eq:power_law}--\eqref{eq:F_L-condition}. Assume that $\alpha>2$, $\tau\in(2,3)$, $\mu>\mu_{\mathrm{pol}}$, and additionally $d\ge 2$. Let $\calC_{\max}^{(n)}$ be the largest component in $Q_n$. Let $u_n,v_n$ be two vertices chosen uniformly at random from $\calC_{\max}^{(n)}$. 
Then there exist constants $\kappa_1,\kappa_2>0$ depending only on  the model parameters such that 
	\begin{align}\label{eq:finite-linear-regime}
		 \lim_{n\to \infty}\pr\left(\kappa_1 |u_n-v_n| < d_{\mathcal{C}}(u_n,v_n)<\kappa_2 |u_n-v_n| 
   \right)=1.
	\end{align}
\end{restatable}
Next we present extensions, in particular for the $\tau>3, \alpha>2$ case where we have very mild conditions on the distribution of $L$, see Corollary \ref{cor:classical} below.

\subsubsection{Limit Cases and Extensions}\label{sec:threshold}
The results above naturally extend to cases/models that may informally be described as $\alpha = \infty$ or $\beta = \infty$, and to some extent to $\tau\ge3$ as well. We start with $\alpha = \infty$. This means that  we replace the condition~\eqref{eq:connection_prob} on the connection probability $h(\cdot)$ by
\begin{align}\label{eq:alpha_infty}
 h(x,w_1,w_2)
 \begin{cases} 
 \ = 0,\quad  & \text{if } \tfrac{w_1w_2}{|x|^d} < c', \\ 
 \ \ge \underline{c}\quad & \text{if }\tfrac{w_1w_2}{|x|^d} \geq 1,
 \end{cases}
 \end{align}
 for some constants $\underline c, c' \in(0,1] $, and where $h(x,w_1, w_2)$ can take arbitrary values in $[0,1]$ when $w_1w_2/|x|^d\in [c',1)$.
 We use the bound $\tfrac{w_1w_2}{|x|^d} \geq 1$ in the second row for convenience, as it allows us to write the proofs for different cases in a consistent way, but it could easily be replaced by $\tfrac{w_1w_2}{|x|^d} \geq c''$ for any other constant $c''\ge c'$.  The connectivity function $h$ in \eqref{eq:alpha_infty} covers the so-called threshold regime for \emph{hyperbolic random graphs} when we also set $d=1$, see~\cite[Theorem 2.3]{bringmann2019geometric}. In this case, when $\tau\in(2,3)$, we extend the definitions~\eqref{eq:mu_pol_log} and \eqref{eq:eta_0} in the natural way, since $\lim_{\alpha\to \infty}\mu_{\mathrm{pol},\alpha}=1/d$:
\begin{align}\label{eq:alpha-infty-definitions}
    \mu_{\log}:=\frac{3-\tau}{\beta}, \quad \mu_{\mathrm{pol}}:= \frac{1}{d}+\frac{3-\tau}{\beta},\quad 
    \eta_0 := \begin{cases}
	1 & \mbox{ if $\mu>\mu_{\mathrm{pol}}$,}\\
	  d\cdot(\mu-\mu_{\log})  & \mbox{ if $\mu\le\mu_{\mathrm{pol}}$.}
	\end{cases}
\end{align}
To describe the case $\beta = \infty $, we replace~\eqref{eq:F_L-condition} by the condition
\begin{align}\label{eq:beta_infty}
	\lim_{t\to 0} F_L(t)/t^{\beta} = 0 \mbox{ for all }0<\beta <\infty.
\end{align}
This means that the cdf of $L> 0$ is flatter than any polynomial near $0$. In particular, this condition is satisfied if $F_L$ has no probability mass around zero, for example in the case $L \equiv 1$. 
 
In this case, we again replace~\eqref{eq:mu_pol_log}-\eqref{eq:eta_0} naturally by 
\begin{align}\label{eq:beta-infty-definitions}
    \mu_{\log}:=0, \quad \mu_{\mathrm{pol}}:= \max\{\tfrac{1}{d}, \mu_{\mathrm{pol},\alpha}\}, \quad
    \eta_0 := \begin{cases}
	1 & \mbox{ if $\mu>\mu_{\mathrm{pol}}$,}\\
	\min\{ d\mu, \mu/\mu_{\mathrm{pol},\alpha}\} & \mbox{ if $\mu\le\mu_{\mathrm{pol}}$.}
	\end{cases}
\end{align}
We mention that these definitions stay valid also for $\tau=3$. 
Finally, in the case $\alpha=\beta=\infty$ we replace~\eqref{eq:mu_pol_log}-\eqref{eq:eta_0} by
\begin{align}\label{eq:alpha-beta-infty-definitions}
    \mu_{\log}:=0, \quad \mu_{\mathrm{pol}}:= \tfrac{1}{d}, \quad \eta_0 := \min\{1,d\mu\}.
\end{align}

Our main results still hold for these limit regimes. We remark that the corresponding upper bounds also hold, see~\cite[Theorem~1.8]{komjathy2022one1}.
\begin{theorem}
[Extension to threshold IGIRGs/GIRGs, and $\beta=\infty$]~
\label{thm:threshold_regimes}
\begin{enumerate}[(a)]
    \item Theorems~\ref{thm:polynomial_regime}, \ref{thm:linear_regime}, and \ref{thm:finite_graph} still hold for $\alpha=\infty$ if we replace the definitions of $\mu_{\mathrm{pol}}, \mu_{\log}, \eta_0$ in~\eqref{eq:mu_pol_log}-\eqref{eq:eta_0} by their values in~\eqref{eq:alpha-infty-definitions}. 
    \item Theorems~\ref{thm:polynomial_regime}, \ref{thm:linear_regime}, and \ref{thm:finite_graph} still hold for $\beta=\infty$ if we replace the definitions of $\mu_{\mathrm{pol}}, \mu_{\log}, \eta_0$ in~\eqref{eq:mu_pol_log}-\eqref{eq:eta_0} by their values in \eqref{eq:beta-infty-definitions}.
    \item Theorems~\ref{thm:polynomial_regime}, \ref{thm:linear_regime}, and \ref{thm:finite_graph} still hold for $\alpha=\beta=\infty$ if we replace the definitions of $\mu_{\mathrm{pol}}, \mu_{\log}, \eta_0$ in~\eqref{eq:mu_pol_log}-\eqref{eq:eta_0} by their values in~\eqref{eq:alpha-beta-infty-definitions}.
\end{enumerate}
\end{theorem}
In Theorem \ref{thm:threshold_regimes}(b), the requirement that  $\mu>\mu_{\mathrm{\log}}=0$ implies that we exclude the case $\mu=0$. Indeed, $\mu=0$ means the setting of classical iid first-passage percolation, where whenever $\tau\in(2,3)$, explosion occurs under mild conditions on $L$, see \cite[Theorems 2.11, 2.12]{komjathy2020explosion}.
Theorem~\ref{thm:threshold_regimes}(a) implies the analogous result for hyperbolic random graphs (HypRG)  by setting $d=1$ in \eqref{eq:alpha-infty-definitions}, except for some technical details. By our Definition \ref{def:girg}, the number of vertices in a finite GIRG is random, and has  Poisson distribution with parameter $n$, while in the usual definition of HypRG~\cite{krioukov2010hyperbolic,gugelmann2012random} and GIRG~\cite{bringmann2019geometric} the number of vertices is exactly $n$. Further, in HypRG the distribution of the vertex-weights is $n$-dependent $W^{(n)}$, so that $W^{(n)}$ converges to a limiting distribution $W$ \cite{komjathy2020explosion}. These issues can be overcome by coupling techniques e.g. presented in \cite{komjathy2020explosion}: a model with exactly $n$ vertices can be squeezed between two GIRGs with Poisson intensity $1-\sqrt{4\log n/n}$ and  $1+\sqrt{4\log n/n}$, and one can couple $n$-dependent and limiting vertex-weights to each other, respectively, but we avoid spelling out the details and refer the reader to \cite[Claims 3.2, 3.3]{komjathy2020explosion}.

Finally, the lower bounds in Theorems~\ref{thm:polynomial_regime} and~\ref{thm:linear_regime} also hold under much weaker, but more technical conditions. In particular, for $\tau > 3$ we can strongly relax Assumption~\ref{assu:L} on~$L$. 
\begin{restatable}{theorem}{LowerBounds}
\label{thm:linear_polynomial_lower_bound}
Consider $1$-FPP on IGIRG or SFP of Definition \ref{def:girg} satisfying the assumptions given in \eqref{eq:power_law}--\eqref{eq:cost} with $0\in\calV$ (but not necessarily Assumption~\ref{assu:L} on $L$). Assume that $\alpha>2$ and $\tau\in(2,\infty)$.  

(1) \emph{Conditions for polynomial distance lower bound.} 
If  $\boldsymbol{\tau\in (2,3]}$ and $L$ satisfies Assumption \ref{assu:L} with some $\beta\in(0,\infty]$ so that $\boldsymbol{\mu > \mu_{\log}}$, then
for all $\eps >0$, almost surely there is $r>0$ such that 
\begin{align}\label{eq:lin-poly-lower-bound}
	\mbox{for all $x\in \calV$ with $|x|\ge r$}: d_{\mathcal{C}}(0,x) \ge |x|^{\eta_0-\eps}.  
\end{align}

(2) \emph{Conditions for strictly linear distance lower bound.} 
If \emph{either} $\boldsymbol{\tau>3}$ and $\mu\ge 0$, and $L$ in \eqref{eq:cost} satisfies $\mathbb P(L>0)=1$, \emph{or}  $\boldsymbol{\tau\in(2,3]}$ and $L$ satisfies Assumption \ref{assu:L} with some $\beta>0$ so that $\boldsymbol{\mu > \mu_{\mathrm{pol}}}$, then there exists $\kappa>0$ such that almost surely there is $r>0$ such that
    \begin{align}\label{eq:linear-lower}
	\mbox{for all $x\in \calV$ with $|x|\ge r$}: d_{\mathcal{C}}(0,x) \ge \kappa |x|.
	\end{align}
\end{restatable}
Compared to Theorems~\ref{thm:polynomial_regime} and~\ref{thm:linear_regime}, Theorem~\ref{thm:linear_polynomial_lower_bound} covers the boundary case $\tau=3$.  More importantly, it states that the linear lower bound in \eqref{eq:linear-lower} holds   for $\tau> 3$ and $\alpha>2$ under very mild conditions on $L$ and $\mu$, i.e.,  we allow arbitrary edge weight distributions $L$ that have no probability mass at zero, and arbitrary penalty strength $\mu\ge 0$. This includes \emph{classical first passage percolation} by setting $\mu=0$ and $L$ arbitrary, a.s.\ positive, but also the case $L \equiv 1$ and $\mu=0$, giving \emph{graph distances}. In this latter case we recover the result of Berger~\cite{berger2004lower} for long-range percolation (LRP) and its extension by Deprez, Hazra, and W\"uthrich~\cite{deprez2015inhomogeneous} for SFP. We state the case of  \emph{classical} (iid) first passage percolation on finite variance degree models ($\tau>3$) with long-range parameter $\alpha>2$ as an immediate corollary.
\begin{corollary}\label{cor:classical}
    Consider classical first passage percolation with iid transmission times from distribution $L$ satisfying $\pr(L>0)=1$ on IGIRG, GIRG or SFP of Definition \ref{def:girg} with $\tau>3, \alpha>2$. 
Then there exists $\kappa>0$ such that almost surely there is $r>0$ such that
\begin{align*}
	\mbox{for all $x\in \calV$ with $|x|\ge r$}:	d_{\mathcal{C}}(0,x) \ge \kappa |x|.
	\end{align*}
\end{corollary}

In Theorem \ref{thm:linear_polynomial_lower_bound}, (unlike in Theorem~\ref{thm:linear_regime}), since we state lower bounds, we do not condition on $0,x \in \calC_\infty$. In the case $\tau > 3$, an infinite component does not need to exist (it depends on the constants and slowly varying functions), so conditioning on $0,x \in \calC_\infty$ would not make sense. However, for parameters that ensure an infinite component of positive density, the event $0,x \in \calC_\infty$ has positive probability, and~\eqref{eq:lin-poly-lower-bound} and ~\eqref{eq:linear-lower} remain true after conditioning. 

{\bf Discontinuity at $\tau=3$.} Remarkably, for $\tau > 3$ there is no analogue to the strictly polynomial regime in Theorem~\ref{thm:polynomial_regime}, even though algebraically the limits of $\lim_{\tau\nearrow 3}\mu_{\textrm{pol}}=1/d$ and $\lim_{\tau\nearrow 3}\eta_0=\mu d$ exist and are in $(0,1)$. So if we fix some $\mu < 1/d$ and let $\tau \nearrow 3$, the cost-distances grow polynomially with exponents bounded away from one (e.g., they approach $1/2$ from below for $\mu  = 1/(2d)$). But as soon as $\tau > 3$ is reached, the exponent ``jumps'' to one. In other words, even though $\mu_{\textrm{pol}}$ converges to a positive limit $1/d$ as $\tau \nearrow 3$, it drops to $0$ as soon as $\tau >3$ is reached. In this sense, the parameter space is discontinuous in $\eta_0$ and $\mu_{\textrm{pol}}$.

\subsection{Proof outline and organisation of the paper}\label{sec:organisation}

\emph{Proof outline of the lower bounds.} For lower bounds, we need to show that every path from a vertex $x$ to $0$ is expensive. For this, we quantify the property that ``most long edges are expensive'' in Lemma~\ref{lem:no_long_cheap_edge}. This allows us to generalise a renormalisation method from Berger~\cite{berger2004lower}. For long-range percolation, Berger considers a growing system of boxes around the origin and defines a box $Q$ as \emph{good} if it does not contain edges of linear length in the box size and the same property holds recursively for most of its child-boxes, which are a certain set of non-disjoint sub-boxes that cover $Q$ multiple times. In the setting of 1-FPP, we modify this definition since long edges \emph{do} exist when $\tau\in(2,3)$ (which is implied by the doubly-logarithmic graph distances), but these edges are typically very expensive. We then show inductively the deterministic statement that once a box is good, the cost-distance is ``large''  between any pair of vertices inside the box with Euclidean distance linear in the box size.   
By ``large'', we mean either linear, or polynomial with an exponent less than one, in the Euclidean distance, depending on the model parameters,  which is 
unlike the linear graph-distances in~\cite{berger2004lower}. Polynomial cost-distances occur when $\mu\in(\mu_{\mathrm{log}},\mu_{\mathrm{pol}})$. For the inductive step, for any two sufficiently distant vertices $u$, $v$ in the same good box $Q$, we use that any path $\pi_{u,v}$ between them must either \emph{(i)} contain a long and thus expensive edge, or \emph{(ii)} has many long disjoint sub-segments in good child-boxes of $Q$, whose costs we can lower-bound by induction. See Figure~\ref{fig:lower-bound-proof}.

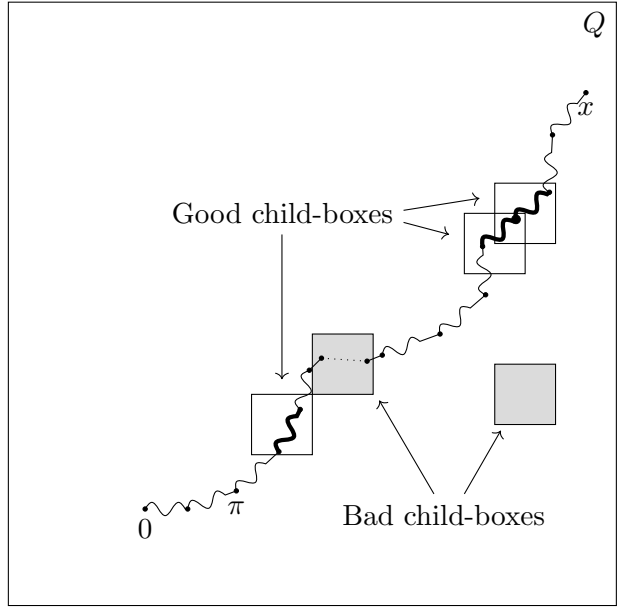
\begin{figure}[t]

\centering
\begin{tikzpicture}[scale=0.4]
    \draw[black] (-10,-10) rectangle (10,10) node[anchor=north east] {$Q$};
    \draw[fill=grey] (0,-3) rectangle (2,-1);
    \draw[fill=grey] (6,-4) rectangle (8,-2);

    \filldraw[black] (9,7) circle (2pt) node[anchor=north] {$x$};
    \filldraw[black] (-5.5,-6.8) circle (2pt) node[anchor=north] {$0$};
    
    \draw[snake=coil,segment aspect=0] (-5.5,-6.8) -- (-4.1,-6.8);
    \filldraw[black] (-4.1,-6.8) circle (2pt);
    \draw[snake=coil,segment aspect=0]  (-4.1,-6.8) -- (-2.5,-6.2) node[anchor=north] {$\pi$};
    \filldraw[black] (-2.5,-6.2) circle (2pt);
    \draw[snake=coil,segment aspect=0] (-2.5,-6.2) -- (-1.1,-4.9) ;
    \filldraw[black] (-1.1,-4.9) circle (2pt);
    \draw[snake=coil,segment aspect=0, ultra thick] (-1.1,-4.9) -- (-0.4,-3.5);
    \draw (-2,-5) rectangle (0,-3);
    \filldraw[black] (-0.4,-3.5) circle (2pt);
    \draw[snake=coil,segment aspect=0] (-0.4,-3.5) -- (-0.1,-2.2);
    \filldraw[black] (-0.1,-2.2) circle (2pt);
    \draw (-0.1,-2.2) -- (0.3,-1.8);
    \filldraw[black] (0.3,-1.8) circle (2pt);
    \draw[dotted] (0.3,-1.8) -- (1.8,-1.9);
    \filldraw[black] (1.8,-1.9) circle (2pt);
    \draw (1.8,-1.9) -- (2.3,-1.7);
    \filldraw[black] (2.3,-1.7) circle (2pt);
    \draw[snake=coil,segment aspect=0] (2.3,-1.7) -- (4.2,-1);
    \filldraw[black] (4.2,-1) circle (2pt);
    \draw[snake=coil,segment aspect=0] (4.2,-1) -- (5.7,0.3);
    \filldraw[black] (5.7,0.3) circle (2pt);
    \draw[snake=coil,segment aspect=0] (5.7,0.3) -- (5.6,1.9);
    \filldraw[black] (5.6,1.9) circle (2pt);
    \draw[snake=coil,segment aspect=0, ultra thick] (5.6,1.9) -- (6.7,2.8);
    \draw (5,1) rectangle (7,3);
    \filldraw[black] (6.7,2.8) circle (4pt);
    \draw (6,2) rectangle (8,4);
    \draw[snake=coil,segment aspect=0, ultra thick] (6.7,2.8) -- (7.8,3.7);
    \filldraw[black] (7.8,3.7) circle (2pt);
    \draw[snake=coil,segment aspect=0] (7.8,3.7) -- (7.9,5.6);
    \filldraw[black] (7.9,5.6) circle (2pt);
    \draw[snake=coil,segment aspect=0] (7.9,5.6) -- (9,7);

    \draw (3,3) node[left] {Good child-boxes};
    \draw[->] (-1,2.3)  -- (-1,-2.5);
    \draw[->] (3,2.7)  -- (4.5,2.3);
    \draw[->] (3,3.1)  -- (5.5,3.5);

    \draw (8,-7) node[left] {Bad child-boxes};
    \draw[->] (4,-6.3)  -- (2.2,-3.2);
    \draw[->] (5,-6.3)  -- (6.2,-4.2);
\end{tikzpicture}

\caption{\small Illustration of the inductive step for the lower bound proof. The big box is good, hence all edges in it that are longer than one-hundredth of the smaller boxes' side-length are already too expensive.  So, any cheap-enough path $\pi$ between $0$ and $x$ contained in $Q$ only uses edges shorter than that. There are only a few bad child-boxes, since $Q$ is good, and we remove the edges of $\pi$ with at least one end-point in bad child-boxes. We select from the remaining edges of $\pi$ enough sub-segments such that each sub-segment is fully contained in a good child-box of $Q$ and connects two sufficiently far vertices within that child-box. Three sub-segments are shown in bold, together with their corresponding child-boxes. By induction we have a lower bound on the costs of these sub-segments, and summing up those yields a lower bound on $\cost{\pi}$.}
\label{fig:lower-bound-proof}
\end{figure}

We give the definition of good boxes in the FPP setting in Section~\ref{sec:good-boxes}, after calculating the probability that long but cheap edges exist in a box in Section~\ref{sec:long-cheap-edges}.  In Section~\ref{sec:inside-good-blocks} we then show that there are no cheap paths \emph{within} a good box, and in Section~\ref{sec:lower_bound_polynomial} that there are no cheap paths at all, excluding thus cheap paths that leave and return to the same box.  We avoid Kingman's subadditive theorem in Proposition \ref{prop:path_in_good_block}, see also the proof on page \pageref{proof:linear_polynomial_lower_bound}.
Finally, in Section~\ref{sec:limit-cases-lower} we show the analogous lower bounds for $\alpha = \infty$ and/or $\beta = \infty$ in Theorem~\ref{thm:threshold_regimes} using coupling arguments.

The proof of the upper bound (that we sketched already in Section \ref{proof:idea-upper}) can be found in Section \ref{sec:linear-regime}.

\subsection{Notation}\label{sec:notation}

The graphs in this paper are undirected, simple graphs with vertex set $\mathcal V \subseteq \R^d$. For a graph $G=(\mathcal V,\mathcal E)$ and a set $A\subseteq \R^d$, we denote by $G[A]$ the induced subgraph of $G$ with vertex set $\mathcal V \cap A$. For two vertices $x,y\in \mathcal V$, we denote the edge between them by $xy$. For a path $\pi$, we denote the number of edges in $\pi$ by $|\pi|$. 

Given a cost function $\mathcal C: \mathcal E \to [0,\infty]$ on the edges, the cost of a collection of edges $\pi$ is the sum of the costs of the edges in the collection, $\cost{\pi}:=\sum_{e\in\pi}\cost{e}$. For convenience we define $\calC(xx):=0$ for all $x\in \calV$. We define the \emph{cost-distance} between vertices $x$ and $y$ as
\begin{align}\label{eq:cost_distance}
    d_{\mathcal{C}}(x,y):=\inf\{\cost{\pi} : \pi \textnormal{ is a path from } x \textnormal{ to } y \textnormal{ in $G$}\}.
\end{align}
We define the graph distance $d_G(x,y)$ similarly, when we set all edge-costs to $1$. The graph will usually be clear from the context.

We denote the Euclidean norm of $x \in \mathbb R^d$ by $|x|$. We denote by $B_r(x) := \{y\in \R^d : |x-y|\le r\}$ the Euclidean ball with radius $r\geq 0$ around $x$, and by $\mathcal B_r(x):=\{y\in\mathcal V : |x-y|\le r\} = B_r(x) \cap \mathcal V$ the set of vertices in this ball. The minimal notation difference is intentional, as we do not expect any confusion to arise between the two sets. The \emph{graph-distance ball} and \emph{cost-distance ball} (or \emph{cost-ball} for short) around a vertex~$x$ are the vertex sets $\mathcal B_r^G(x):=\{y\in\mathcal V : d_G(x,y)\le r\}$ and $\mathcal B_r^{\mathcal C}(x):=\{y\in\mathcal V : d_{\mathcal C}(x,y)\le r\}$, respectively. We set $B_r := B_r(0)$, and similarly for $\mathcal B_r$, $\mathcal B_r^G$, and $\mathcal B_r^{\mathcal C}$, if $0$ is a vertex. We denote by $\partial B_r(x) := B_r(x)\setminus \{y\in \R^d : |x-y| < r\}$, and similarly for $\partial \calB_r$, $\partial \calB_r^G$, and $\partial \calB_r^{\calC}$. In particular, $\partial \calB_1^G(v)$ is the neighbourhood of a vertex $v$.

For parameters $a,b >0$ (model parameters or otherwise), we use ``for all $a\ggs b$'' as shortcut for ``$\forall b>0:\, \exists a_0 = a_0(b):\, \forall a\ge a_0$''. We also say ``$a \ggs b$'' to mean that $a \ge a_0(b)$. We use $a \lls b$ analogously, and may use more than two parameters. For example, ``for $a\ggs b,c$'' means ``$\forall b,c>0:\, \exists a_0=a_0(b,c):\, \forall a \ge a_0$''. A measurable function $\ell:(0,\infty) \to (0,\infty)$ is \emph{slowly varying} if $\lim_{x\to \infty} \ell(ax)/\ell(x) =1$ for all $a>0$. We denote the natural logarithm by $\log$. For $n \in \N$ we abbreviate $[n]:= \{1,\ldots,n\}$. For $S \subseteq \R^d$, we denote the Lebesgue measure of $S$ by $\vol(S)$.

%% file: lower_bounds.tex
\section{Lower Bounds}\label{sec:lower-bound}

The main part of this section is the proof of Theorem~\ref{thm:linear_polynomial_lower_bound}, which in turn implies the lower bounds in Theorem~\ref{thm:polynomial_regime} and Theorem~\ref{thm:linear_regime}. Note that the corresponding lower bound for GIRG in Theorem~\ref{thm:finite_graph} also follows trivially since GIRG is a subgraph of IGIRG. 
Throughout the section, we will maintain Assumption~\ref{assu:L} on $L$ unless explicitly noted otherwise.

\medskip

Before we start with the proof, we note an elementary lemma stating that the product of two random variables with regularly varying tails sharing the same tail exponent again has a regularly varying tail with the same tail exponent.
\begin{lemma}\label{lem:product_distribution}
	Let $X, Y$ be two independent positive random variables with cumulative distributions $F_X(x)=1-\ell_1(x)x^{-\tau}$ and $F_Y(y)=1-\ell_2(y)y^{-\tau}$, respectively, where $\ell_1$ and $\ell_2$ are slowly varying functions. Then the cumulative distribution of their product $Z:=XY$ is given by $F_Z(z)=1- \ell^{\star}(z)z^{-\tau}$ for some slowly varying function $\ell^{\star}$.
\end{lemma}
\begin{proof}
	This is a consequence of~\cite[Corollary, Page 3]{embrechts1980closure}.
\end{proof}

\subsection{Upper bound on the number of long and cheap edges}\label{sec:long-cheap-edges}
As mentioned in Section \ref{sec:outline}, we start with developing a renormalisation argument. 
The argument builds on the basic observation that a box is unlikely to contain an edge that is both long in Euclidean distance and cheap, i.e., its cost is small. This follows from a straightforward but tedious calculation and is summarised in the following lemma. The more interesting part of the proof comes afterwards, when we turn this basic property into the statement that all paths between a vertex $x$ and $0$ are at least polynomially expensive.

\begin{lemma}\label{lem:no_long_cheap_edge}
Consider $1$-FPP on IGIRG or SFP of Definition \ref{def:girg} satisfying the assumptions given in \eqref{eq:power_law}--\eqref{eq:F_L-condition} with $0\in\calV$. Assume that $\tau\in(2,\infty)$ and $\beta>0$.
    For all $\eps > 0$, if $N>0$ is sufficiently large relative to $\eps$, then the following holds. Let $A>N$ and $a > 0$. Then the expected number of edges in $[-A/2,A/2]^d$ with length at least $N$ and cost at most $N^{ad}$ is at most
	\begin{align}\label{eq:no_long_cheap_edge}
	 \begin{cases}
	A^dN^{\eps}\Big( N^{-d(\alpha-1)} + N^{-d(\tau-2)}\Big) & \mbox{ if $a\ge\mu$,}\\
	A^dN^{\eps}\Big(N^{-d(\alpha-1)} + N^{-d(\alpha - 1 - \frac{a}{\mu}(\alpha-(\tau-1)))} + N^{-d(\tau-2+(\mu-a)\beta)} \Big) & \mbox{ if $a<\mu$.}
	\end{cases}
\end{align}
	The formula for the case $a \ge \mu$ remains valid without the restriction on the cost of the edges, and in this case $L$ does not need to satisfy Assumption \ref{assu:L}. Further, for IGIRG, the formula is also valid under the Palm measures $\mathbb P^0$ or $\mathbb P^{0,x}$ where we condition on having a vertex with unknown weight at $0\in \R^d$ or on having two vertices with unknown weight at location $0,x\in \R^d$, respectively.
\end{lemma}

Before the proof, let us informally explain the formula, suppressing slowly varying factors in the discussion. The factor $A^d$ is simply the (expected) number of vertices in $[-A/2,A/2]^d$. The term $N^{\eps}$ comes from applying Potter's bound~\cite{bingham1989regular} to bound the slowly varying function that appears in the distribution function of the vertex weights in \eqref{eq:power_law} when we integrate over the distribution of the products of weights $W_uW_v$ in the connectivity function in \eqref{eq:connection_prob}. 

For the case $a\geq \mu$, the two terms in the brackets represent two types of edges. Consider a vertex $v$ of constant weight. Then first term $N^{-d(\alpha-1)}$ counts the expected number of neighbours $u$ in distance $\sim N$ (e.g., in distance $[N,2N]$) of constant weight, up to negligible factors. The second term $N^{-d(\tau-2)}$ counts the number of neighbours $u$ in distance $\sim\! N$ of weight $\sim\! N^{d}$, which is exactly the weight needed to get constant connection probability from~\eqref{eq:connection_prob}. When $a\geq \mu$, the cost of such edges is typically less than $N^{ad}$, which is why we can ignore the cost condition in this case. There are potentially other possibilities, but their contribution to the expectation is negligible. 

In the case $a < \mu$, the first term $N^{-d(\alpha-1)}$ is similar to the case $a\ge \mu$, coming from edges between vertices of constant weight, and so the cost is typically constant as well. The third term $N^{-d((\tau-2)+ (\mu-a)\beta)}$ comes from edges between vertices of weights $W_v \sim 1$ and $W_u \sim N^d$. These receive a cost penalty of $N^{d\mu}$, so the probability that the edge has cost at most $N^{ad}$ is roughly $F_L(N^{-d(\mu-a)}) \approx N^{-d(\mu-a)\beta}$.
Finally, the second term $N^{-d(\alpha-1 - a/\mu(\alpha -(\tau-1)))}$ comes from edges between vertices of weights $W_v \sim 1$ and $W_u \sim N^{ad/\mu}$. With this weight, the typical cost of the edge is $N^{ad}$, so the cost condition is satisfied. For fixed $v$, there are $\Theta(N^{d})$ vertices $u$ in distance $\sim\! N$, they have probability $N^{-(ad/\mu)(\tau-1)}$ to be of weight $W_u \sim N^{ad/\mu}$, and the probability to be adjacent to $v$ with this weight is $(N^{ad/\mu}/N^{d})^\alpha$ by~\eqref{eq:connection_prob}. (All up to negligible terms.) Together, this yields the second term. Note that the term $(N^{ad/\mu}/N^{d})^\alpha$ for the connection probability is only correct if the bracket is at most one, i.e., if $a<\mu$. 
\begin{proof}[Proof of Lemma~\ref{lem:no_long_cheap_edge}]
Throughout this proof, we will denote by $C_1, C_2, \ldots$ finite positive constants depending on $\eps$ and the set $\mpar$ of model parameters. For readability, we allow these to appear in the middle of a calculation without necessarily being defined beforehand. 
Note that the statement of the lemma is stronger for smaller $\eps>0$. So without loss of generality, we can assume that we take a $\eps>0$ small enough so that
	\begin{align}\label{eq:delta_no_long_cheap_edge_1}
	    -d(\tau-2)+\eps/2<0 \quad \textnormal{and} \quad -d(\alpha-1)+\eps/2<0.
	\end{align}
	If $\alpha-(\tau-1)-\mu\beta<0$, we can also assume that
	\begin{align}\label{eq:delta_no_long_cheap_edge_2}
	    \alpha-(\tau-1)-\mu\beta+\eps \frac{\mu}{2ad}<0.
	\end{align}
	Let $E(A,N,a)$ denote the expected number of edges in $[-A/2,A/2]^d$ with length at least $N$ and cost at most $N^{ad}$. We first compute $E(A,N,a)$ for SFP, having vertex set $\Z^d$. Let
	\begin{equation}\label{eq:lambda_r}
	    \Lambda(r):=\E[(1 \wedge W_xW_y/r^d)^{\alpha} F_L(N^{ad}(W_xW_y)^{-\mu})].
	\end{equation} 
	Using conditional expectation we have	
	\begin{align*}
		E(A,N,a) &=
		\sum_{\substack{x,y \in [-A/2,A/2]^d\cap \Z^d\\|x-y| \ge N}}
		\E\left[\mathbbm{1}_{\{xy \textnormal{ is an edge}\}}\cdot\mathbbm{1}_{\{\cost{xy}\le N^{ad}\}}\right] \\&\le
		\sum_{\substack{x,y \in[-A/2,A/2]^d\cap\Z^d\\N\le|x-y|\le dA}}
		\E\left[\overline{c}\left(1 \wedge \dfrac{W_xW_y}{|x-y|^d}\right) ^\alpha \cdot F_L(N^{ad}(W_xW_y)^{-\mu}) \right] \\
		&= \overline{c} \sum_{x \in [-A/2,A/2]^d\cap \Z^d}
		\sum_{\substack{y \in [-A/2,A/2]^d\cap \Z^d\\N\le|x-y|\le dA}}
		\Lambda(|x-y|).
	\end{align*}
	Note that the number of vertices $\newconstant{cst:no_long_cheap_edge2}|\Z^d\cap [-A/2,A/2]^d|\le \useconstant{cst:no_long_cheap_edge2}A^d$. In order to simplify calculations, we will replace the second sum by an integral. More precisely, by usual isoperimetric inequalities for $\Z^d$, there is a constant $\newconstant{cst:no_long_cheap_edge3}\useconstant{cst:no_long_cheap_edge3}=\useconstant{cst:no_long_cheap_edge3}(d)$ such that 
	\begin{align}\label{eq:ENA}
		E(A,N,a)\le \useconstant{cst:no_long_cheap_edge3}A^d \int_{r=N}^{dA} r^{d-1} \Lambda(r)\,\mathrm{d}r.
	\end{align}
	For IGIRG, we obtain the same formula~\eqref{eq:ENA} in a simpler way. In expectation, there are $A^d$ vertices in $[-A/2,A/2]^d$. For any fixed vertex, the expected number of neighbours in IGIRG in distance between $N$ and $dA$ with edge cost at most $N^{ad}$ is given by the integral $\int_{r=N}^{dA} r^{d-1} \Lambda(r)\,\mathrm{d}r$. Since $dA$ bounds the diameter of $[-A/2,A/2]^d$, this includes all neighbours of this type in $[-A/2,A/2]^d$. Hence, the upper bound~\eqref{eq:ENA} also holds for IGIRG and its subgraph GIRG. The same bound remains true under the Palm measures $\mathbb P^{0},\mathbb P^{0,x}$ by possibly increasing the constant prefactor $C_2$ to account for the edges emanating from the two extra vertices.  
		
	Next we bound $\Lambda(r)$ in \eqref{eq:ENA}. Defined in  
 \eqref{eq:lambda_r}, $\Lambda(r)$ only depends on the product $W_xW_y=:Z$, not on the individual weights of the two vertices. By Lemma~\ref{lem:product_distribution}, the distribution of $Z$ is of the form $f_Z(z)=\ell^{\star}(z)z^{-\tau}$, where $\ell^{\star}$ is a slowly varying function. For the sake of simplicity, we will work with $Z$ having a density, but a proof using only Lebesgue-Stieltjes integration is similar. 
We recall from \eqref{eq:ENA} that $r>N$ and rewrite  \eqref{eq:lambda_r} using law of total probability as
\begin{equation}\label{eq:lambda-r-detailed}
    \Lambda(r)= \int_{z=1}^{\infty} \left(1\wedge\dfrac{z}{r^d}\right)^\alpha \cdot F_L(N^{ad}z^{-\mu}) \cdot \dfrac{\ell^{\star}(z)}{z^{\tau}}\,\mathrm{d}z.
\end{equation} 
	We now split into cases depending on the value of $a$.
	
	\medskip\noindent\textbf{Case 1: $\boldsymbol{a \ge \mu}$.}
	In this case we first show that 
 for all $\eps>0$,  for all sufficiently large $r$ (i.e., larger than some $r_0(\eps)$), 
 \newconstant{local-lambda}
\begin{equation}\label{eq:local-lambda}
    \Lambda(r)\le \useconstant{local-lambda} (r^{-d(\tau-1)} + r^{-d\alpha})r^{\eps/2}. \end{equation}
	We split the inner integral of~\eqref{eq:lambda-r-detailed} in two parts, at $r^d$. The first part is given by
	\begin{equation}
	\begin{aligned}\label{eq:I_1}
		I_1 := \int_{z=1}^{r^d} \left(1\wedge\dfrac{z}{r^d}\right)^\alpha \cdot\underbrace{F_L(N^{ad}z^{-\mu})}_{\le 1} \cdot \dfrac{\ell^{\star}(z)}{z^{\tau}} \,\mathrm{d}z \le r^{-\alpha d} \int_{z=1}^{r^d} z^{\alpha-\tau}\ell^{\star}(z) \,\mathrm{d}z.
	\end{aligned}
	\end{equation}
	Since $r>N$, and we will later let $N\to \infty$, we can use Karamata's theorem \cite[Prop. 1.5.6]{bingham1989regular}, and obtain that for $\alpha-(\tau-1)>0$, 
	\newconstant{cst:no_long_cheap_edge5}
	\begin{align}\label{eq:I1-case1}
	   I_1\le r^{-\alpha d} \int_{z=1}^{r^d}z^{\alpha-\tau}\ell^{\star}(z)\,\mathrm{d}z \le r^{-\alpha d} \useconstant{cst:no_long_cheap_edge5}r^{d(\alpha-(\tau-1))}\ell^{\star}(r^d)\le \useconstant{cst:no_long_cheap_edge5}r^{-d(\tau-1)+\varepsilon/2}
	\end{align}
	for $N$ (and thus $r$) large enough, where we used Potter's bound~\cite{bingham1989regular} to get $\ell^{\star}(r)=o(r^{\eps/(2d)})$ as $r\to\infty$, and thus for sufficiently large $N$ (and hence $r$) we obtain that $I_1$ in \eqref{eq:I_1} in the $\alpha>\tau-1$ case is bounded from above by the first term in \eqref{eq:local-lambda}.
	If $\alpha-\tau+1=0$, we again use $\ell^{\star}(z)=o(z^{\eps/(2d)})$ as $z\to\infty$ to get
	\newconstant{cst:no_long_cheap_edge7} \newconstant{cst:no_long_cheap_edge8}
	\begin{align}\label{eq:case1-I1-2}
	    I_1 \le \useconstant{cst:no_long_cheap_edge7}r^{-\alpha d}\int_{z=1}^{r^d} z^{\eps/(2d)-1}\,\mathrm{d}z \le \useconstant{cst:no_long_cheap_edge8}r^{-\alpha d + \eps/2},    
	\end{align}
	which is the second term in \eqref{eq:local-lambda}.	Finally, when $\alpha<\tau-1$, we use Potter's bound~\cite{bingham1989regular} to get\newconstant{cst:no_long_cheap_edge9}
		$\ell^{\star}(z)\le 	\useconstant{cst:no_long_cheap_edge9} z^{\eps}$ as $z\to\infty$, and since $\alpha-\tau+\eps < -1$ we get that the integral is bounded by some constant and hence the bound \eqref{eq:case1-I1-2} remains valid.
	Combining equations~\eqref{eq:I1-case1}--\eqref{eq:case1-I1-2}, we obtain that regardless of the relation between $\alpha$ and $\tau-1$, $I_1$ satisfies the bound in \eqref{eq:local-lambda}.
	The second part of the inner integral in \eqref{eq:lambda-r-detailed}  is
	\begin{align*}
		I_2 := \int_{z=r^d}^{\infty} \left(1\wedge\dfrac{z}{r^d}\right)^\alpha \cdot F_L(N^{ad}z^{-\mu}) \cdot \dfrac{\ell^{\star}(z)}{z^{\tau}}\,\mathrm{d}z \le \int_{z=r^d}^{\infty} z^{-\tau}\ell^{\star}(z)\,\mathrm{d}z,
	\end{align*}
	since $F_L\le 1$ always holds.
	Since $r^d\rightarrow\infty$ as $N\rightarrow\infty$, Proposition 1.5.10 of~\cite{bingham1989regular} tells us that for $N$ large enough
	\newconstant{cst:no_long_cheap_edge12}
	\begin{align*}
	   I_2\le  \int_{z=r^d}^{\infty} z^{-\tau}\ell^{\star}(z)\,\mathrm{d}z \le \useconstant{cst:no_long_cheap_edge12}r^{-d(\tau-1)}\ell^{\star}(r^d)
	\end{align*}
Again using Potter's bound, for sufficiently large $N$ we have
that $I_2$ is dominated by the first term in \eqref{eq:local-lambda}. This finishes the proof that \eqref{eq:local-lambda} holds when $a\ge \mu$.

Returning the attention to $E(A, N, a)$ in \eqref{eq:ENA}, and using now \eqref{eq:local-lambda}, we have for $\tau>2$
    \newconstant{cst:no_long_cheap_edge14}
	\newconstant{cst:no_long_cheap_edge15}
	\begin{align*}
		E(A,N,a)&\le \useconstant{cst:no_long_cheap_edge3} A^d \int_{r=N}^{dA} r^{d-1}  \useconstant{local-lambda} (r^{-d(\tau-1)} + r^{-d\alpha})r^{\eps/2}\mathrm{d}r \\
		&= \useconstant{cst:no_long_cheap_edge14} A^d\int_{r=N}^{dA} r^{-d(\tau-2)+\eps/2-1} + r^{-d(\alpha-1)+\eps/2-1}\,\mathrm{d}r \\
		&\stackrel{\eqref{eq:delta_no_long_cheap_edge_1}}{\le} \useconstant{cst:no_long_cheap_edge15}A^dN^{\eps/2}\left(N^{-d(\tau-2)}+N^{-d(\alpha-1)}\right)
		\le A^dN^{\eps}\left(N^{-d(\tau-2)}+N^{-d(\alpha-1)}\right),
	\end{align*}
	where the last inequality holds for $N$ large enough. Thus we have proved~\eqref{eq:no_long_cheap_edge} when $a \ge \mu$. Observe that we used a trivial bound $F_L\le 1$ in the proofs, hence we get that the statement also holds without any restriction on the edge-costs and without Assumption~\ref{assu:L}. 
     
	\medskip\noindent\textbf{Case 2: $\boldsymbol{a < \mu}$.}
	In this case we also start bounding $\Lambda(r)$ in \eqref{eq:lambda-r-detailed} first.
	We recall the constant $t_0$ from \eqref{eq:F_L-condition}. Note that $t_0^{-1/\mu}N^{ad/\mu}$ is smaller than $N^d$ (and thus $r^d$) for $N$ large enough. We assume this inequality and split the  integral of~\eqref{eq:lambda-r-detailed} into three parts. In the first part we bound the factor $F_L$ by $1$ from above: 
	\begin{align*}
		\widetilde I_1 := \int_{z=1}^{t_0^{-1/\mu}N^{ad/\mu}} \left(1\wedge\dfrac{z}{r^d}\right)^\alpha \cdot F_L(N^{ad}z^{-\mu}) \cdot z^{-\tau}\ell^{\star}(z)\,\mathrm{d}z \le r^{-d\alpha}\int_{z=1}^{t_0^{-1/\mu}N^{ad/\mu}} z^{\alpha-\tau}\ell^{\star}(z)\,\mathrm{d}z.
	\end{align*}
	This is the same as $I_1$ in the previous case, except for the upper limit of integration. Using the same reasoning, we get that
	\newconstant{cst:no_long_cheap_edge16}
	\begin{align}\label{eq:I_1_bound_case2}
		\widetilde I_1 \le \useconstant{cst:no_long_cheap_edge16}(t_0)r^{-d\alpha} \left(N^{\eps/2}+N^{\frac{ad}{\mu}(\alpha-\tau+1)+\eps/2}\right)
	\end{align}
	for $N$ large enough.
	In the second part the argument of $F_L$ will be at most $t_0$, hence we can use that $F_L(t)\le c_2t^\beta$ in this regime:
	\begin{align*}
		\widetilde I_2 &:= \int_{z=t_0^{-1/\mu}N^{ad/\mu}}^{r^d} \left(1\wedge\dfrac{z}{r^d}\right)^\alpha \cdot F_L(N^{ad}z^{-\mu}) \cdot z^{-\tau}\ell^{\star}(z)\,\mathrm{d}z \\
		&\stackrel{\eqref{eq:F_L-condition}}{\le} c_2r^{-d\alpha} \int_{z=t_0^{-1/\mu}N^{ad/\mu}}^{r^d} z^{\alpha-\tau}(N^{ad}z^{-\mu})^{\beta}\ell^{\star}(z)\,\mathrm{d}z \\
		&= c_2N^{ad\beta}r^{-d\alpha} \int_{z=t_0^{-1/\mu}N^{ad/\mu}}^{r^d} z^{\alpha-\tau-\mu\beta}\ell^{\star}(z)\,\mathrm{d}z.
	\end{align*}
	If $\alpha-\tau-\mu\beta+1>0$, we can again use Proposition 1.5.8 of~\cite{bingham1989regular} on the integral to get that
	\newconstant{cst:no_long_cheap_edge17}
		\newconstant{cst:no_long_cheap_edge18}
	\begin{align}\label{eq:I2-tilde-new}
	 \widetilde I_2   \le c_2N^{ad\beta}r^{-d\alpha} \useconstant{cst:no_long_cheap_edge17}r^{d(\alpha-\tau-\mu\beta+1)}\ell^{\star}(r^d)   \le \useconstant{cst:no_long_cheap_edge18}N^{ad\beta}r^{-d(\tau+\mu\beta-1)+\eps/2}
	\end{align}
	for $N$ large enough by Potter's bound. 
	If $\alpha-\tau-\mu\beta+1=0$, we use $\ell^{\star}(z)=o(z^{\frac{\eps}{2d}})$ to get
	\newconstant{cst:no_long_cheap_edge19} \newconstant{cst:no_long_cheap_edge20}
	\begin{align}\label{eq:I_2_bound_case2-2}
	    \widetilde I_2 \le \useconstant{cst:no_long_cheap_edge19}N^{ad\beta}r^{-d\alpha} \int_{z=t_0^{-1/\mu}N^{ad/\mu}}^{r^d} z^{\frac{\eps}{2d}-1} \,\mathrm{d}z  \le  \useconstant{cst:no_long_cheap_edge20}N^{ad\beta}r^{-d(\tau-1+\mu\beta)+\eps/2}.
	\end{align}
	Finally, when $\alpha-\tau-\mu\beta+1<0$, we use Potter's bound to get $\ell^{\star}(z)=o(z^{\frac{\mu\eps}{2ad}})$, and since in this case we assume that~\eqref{eq:delta_no_long_cheap_edge_2} holds, the integral is convergent and we have
	\newconstant{cst:no_long_cheap_edge21} \newconstant{cst:no_long_cheap_edge22}
	\begin{align}\label{eq:I_2_bound_case2-3}
	    \widetilde I_2 \le \useconstant{cst:no_long_cheap_edge21}N^{ad\beta}r^{-d\alpha} \int_{z=t_0^{-1/\mu}N^{ad/\mu}}^{r^d} z^{\alpha-\tau-\mu\beta+\frac{\mu\eps}{2ad}} \,\mathrm{d}z \le \useconstant{cst:no_long_cheap_edge22}r^{-d\alpha}N^{\frac{ad}{\mu}(\alpha-\tau+1)+\eps/2}.
	\end{align}
	Combining equations~\eqref{eq:I2-tilde-new}--\eqref{eq:I_2_bound_case2-3}, and using the fact that $r \ge N$ in the inner integral, we get
	\newconstant{cst:no_long_cheap_edge23}
	\begin{align}\label{eq:I_2_bound_case2}
	\widetilde 	I_2\le \useconstant{cst:no_long_cheap_edge23}\left(N^{ad\beta}r^{-d(\tau-1+\mu\beta)} + r^{-d\alpha}N^{\frac{ad}{\mu}(\alpha-\tau+1)}\right) r^{\eps/2}.
	\end{align}
	Finally, the third part of the integral in $\Lambda(r)$ in \eqref{eq:lambda-r-detailed} is given by
	\begin{align*}
		\widetilde I_3 := \int_{z=r^d}^{\infty} \left(1\wedge\dfrac{z}{r^d}\right)^\alpha \cdot F_L(N^{ad}z^{-\mu}) \cdot z^{-\tau}\ell^{\star}(z)\,\mathrm{d}z \le c_2 N^{ad\beta} \int_{z=r^d}^{\infty}z^{-\mu\beta-\tau}\ell^{\star}(z)\,\mathrm{d}z.
	\end{align*}
	Again, we can apply Proposition 1.5.10 of~\cite{bingham1989regular} to treat the inner integral and get, by Potter's bound when $N$ is sufficiently large,
	\newconstant{cst:no_long_cheap_edge24}
	\newconstant{cst:no_long_cheap_edge25}
	\begin{align}\label{eq:I_3_bound_case2}
		\widetilde I_3 \le \useconstant{cst:no_long_cheap_edge24} N^{ad\beta}r^{-d(\tau-1+\mu\beta)}\ell^{\star}(r^d)\le  \useconstant{cst:no_long_cheap_edge25}N^{ad\beta}r^{-d(\tau-1+\mu\beta)+\eps/2}.
	\end{align}
	Combining~\eqref{eq:I_1_bound_case2},~\eqref{eq:I_2_bound_case2} and~\eqref{eq:I_3_bound_case2}, while keeping in mind that $r\ge N$, we get
	\newconstant{cst:no_long_cheap_edge26}
	\begin{align*}
	\Lambda(r)	= \widetilde I_1 + \widetilde I_2 + \widetilde I_3 \le \useconstant{cst:no_long_cheap_edge26}\left(r^{-d\alpha}+r^{-d\alpha}N^{\frac{ad}{\mu}(\alpha-\tau+1)} + N^{ad\beta}r^{-d(\tau+\mu\beta-1)}\right)r^{\eps/2}.
	\end{align*}
	Returning again to \eqref{eq:ENA}, using this bound on $\Lambda(r)$ we obtain
	\newconstant{cst:no_long_cheap_edge28} \newconstant{cst:no_long_cheap_edge27}
	\begin{align*}
		E(A,N,a)
		\le &\useconstant{cst:no_long_cheap_edge28}A^d\int_{r=N}^{dA} r^{-d(\alpha-1)+\eps/2-1} \left(1+N^{\frac{ad}{\mu}(\alpha-\tau+1)}\right) + r^{-d(\tau-2+\mu\beta)+\eps/2-1}N^{ad\beta}\,\mathrm{d}r \\
		\stackrel{\eqref{eq:delta_no_long_cheap_edge_1}}{\le} &\useconstant{cst:no_long_cheap_edge27}A^dN^{\eps/2} \left(N^{-d(\alpha-1)} + N^{-d(\alpha-1-\frac{a}{\mu}(\alpha-\tau+1))} + N^{-d(\tau-2+(\mu-a)\beta)}\right) \\
		\le &A^dN^{\eps}\left(N^{-d(\alpha-1)} + N^{-d(\alpha-1-\frac{a}{\mu}(\alpha-\tau+1))} + N^{-d(\tau-2+(\mu-a)\beta)}\right),
	\end{align*}
	where the last inequality holds for $N$ large enough. Thus we have proved~\eqref{eq:no_long_cheap_edge} when $a < \mu$.
\end{proof}
\setcounter{constant}{0}

\subsection{Good blocks}\label{sec:good-boxes}
In the renormalisation scheme, we shall cover the space with blocks (i.e., boxes) that are iteratively contained in larger and larger blocks. The smallest blocks are level-$1$ blocks, and we group them together in level-$2$ blocks, and so on. A level-$k$ block thus contains sub-blocks of level $k-1$, which in turn contains sub-blocks of level $k-2$, and so on until we reach the level-$1$ blocks. 
Now we turn to the definition of good blocks. Berger~\cite{berger2004lower} defined them as blocks that do not contain edges of linear length (in the block size), and additionally the same property must hold for certain translates and recursively for most of its sub-blocks. In our setting, we have to modify the definition since long edges do exist. However, we can define a block to be `good' if all edges of linear length are costly enough and arrive at the following definition. 

\begin{definition}\label{def:good_box}
    Let $k_0:=16\cdot30^d$, let $A_1>1$ and $u\in (0,1)$ be constants, and define $A_k := A_1(k!)^2$ for all $k \ge k_0$.  

    A \emph{$k$-block} is defined as a $d$-dimensional cube of side length $A_k$. For $k > k_0$, there is a natural partition of a $k$-block $Q$ into $k^{2d}$ many $(k-1)$-blocks; we call these the \emph{children} of $Q$. We denote by $Q_k$ the $k$-block centred at the origin.
    
    We will define the notion of goodness recursively. Let $\eta\in(0,1]$. We say that a $k_0$-block $Q$ is \emph{$\eta$-good} if every edge internal to $Q$ has cost at least $u$. For $k > k_0$ and a $k$-block $Q$, consider all $3^d$ $k$-blocks $Q'$ of the form $Q+jA_{k-1}/2$ for $j\in\{ -1, 0, 1 \}^d$. We say that $Q$ is \emph{$\eta$-good} if for all $Q'=Q+jA_{k-1}/2$:
	\begin{enumerate}[(i)]
		\item Every edge internal to $Q'$ of length larger than $A_{k-1}/100$ has cost at least $uA_k^{\eta}$.
		\item All but at most $3^d$ children of $Q'$ are $\eta$-good.
	\end{enumerate}	
\end{definition}

We will fix $A_1$ and then $u$ in Definition~\ref{def:good_box} such that they satisfy equations~\eqref{eq:u_cond}, \eqref{eq:A_1_cond_3}, \eqref{eq:A_1_cond_2}, \eqref{eq:A_1_cond_1} and~\eqref{eq:A_1_cond} below. We will also assume that $A_1$ is large enough so that the upper bound of Lemma~\ref{lem:no_long_cheap_edge} holds for certain choices of $N$ and $\eps$ specified case-by-case below in \eqref{eq:general_E(k)-2}, \eqref{eq:general_E(k)}, \eqref{eq:eps-1-def}, \eqref{eq:eps-2-def} and \eqref{eq:eps-3-def}.
Recall the definitions of $\mu_{\mathrm{pol}}$ from \eqref{eq:mu_pol_log} and $\eta_0$ from \eqref{eq:eta_0}.

\begin{proposition}\label{prop:large_box_good} 
Consider $1$-FPP on IGIRG or SFP of Definition \ref{def:girg} satisfying the assumptions given in \eqref{eq:power_law}--\eqref{eq:cost} with $0\in\calV$, but $L$ not necessarily satisfying Assumption~\ref{assu:L}.
Assume that 
	\begin{enumerate}[1)]
	\item\label{item:1}
	$\boldsymbol{\alpha>2}$, $\boldsymbol{\tau>3}$, $\mu\ge 0$ arbitrary, and the distribution of $L$ is arbitrary satisfying $\mathbb P(L>0)=1$.  Then there are choices for $A_1$ and $u$ in Definition~\ref{def:good_box} for which a.s.\ there exists $k_1\ge k_0$ such that $Q_k$ is $1$-good for all $k\ge k_1$.
	    \item\label{item:2}
	    $\boldsymbol{\alpha>2}$,  $\boldsymbol{\tau\in(2,3]}$, $\boldsymbol{\mu>\mu_{\log}}$, and $L$ satisfies \eqref{eq:F_L-condition} in Assumption \ref{assu:L}. Then for any $\delta>0$ there are choices for $A_1$ and $u$ in Definition~\ref{def:good_box} for which a.s.\ there exists $k_1\ge k_0$ such that $Q_k$ is $(\eta_0-\delta)$-good for all $k\ge k_1$.
	    \item\label{item:3} $\boldsymbol{\alpha>2}$ , $\boldsymbol{\tau\in(2,3]}$, $\boldsymbol{\mu>\mu_{\mathrm{pol}}}$, and $L$ satisfies \eqref{eq:F_L-condition} in Assumption \ref{assu:L}. Then there are choices for $A_1$ and $u$ in Definition~\ref{def:good_box} for which a.s.\ there exists $k_1\ge k_0$ such that $Q_k$ is $1$-good for all $k\ge k_1$.
	\end{enumerate}
	\end{proposition}
\begin{remark}\label{rem:large_box_good}
    Proposition \ref{prop:large_box_good} implicitly limits the type of paths present between two vertices: a path either uses short edges (of which it needs to use many if the endpoints are far away) or it uses long edges, which do have high cost. We will see that in Case \ref{item:1} paths that have long edges are simply not present. In Case \ref{item:3} they are present, but the cost of long edges is so high that the paths are not more efficient than paths which only use short edges (corresponding to case $\eta_0=1$ of linear cost-distances). Case \ref{item:2} is most exotic: long edges are present, but their cost is on a polynomial scale compared to their Euclidean length, and their precise exponent will give polynomial (but nonlinear) cost-distances.   
    
	The proof we will give follows closely the one in~\cite[Lemma 14]{deprez2015inhomogeneous}. We show that
    \begin{align}\label{eq:tail_bound_k1}
        \pr(Q_k \textnormal{ is not } (\eta_0-\delta)\textnormal{-good}) \le e^{-k}
    \end{align}
  for $k\ge k_0$, where $\delta$ is an arbitrary positive constant for Case \ref{item:2} and $\delta=0$ for Cases \ref{item:1} and \ref{item:3}. (The same inequality holds true under both Palm measures $\pr^{0}, \pr^{0,x}$ as well). Thus 
    \begin{align*}
	    \pr(\forall k\ge k': Q_{k} \textnormal{ is } (\eta_0-\delta)\textnormal{-good})) \ge 1- \sum_{k\ge k'} e^{-k} \ge 1- 2e^{-k'}.
	\end{align*}
     The existence of $k_1$ (so that $Q_k$ is $(\eta_0-\delta)$-good for all $k\ge k_1$) then follows from the Borel-Cantelli Lemma, and for a fixed $q>0$ we can achieve $\pr(k_1 \le k') \ge 1-q$ for ${k'} = |\log(q/2)|$. We will show in the later parts of the section that the lower bounds on cost-distances follow \emph{deterministically} from Proposition \ref{prop:large_box_good}, for all vertices in distance at least $r_0:=c\cdot A_{k_1}$ for a constant $c$, and thus for all vertices in distance $r\ge r_0=c\cdot A_{|\log(q/2)|} = cA_1\cdot (|\log(q/2)|!)^2 \ge e^{c'\cdot \log(1/q) \cdot \log \log(1/q)}$ for a constant $c'>0$.  However, we did not try to optimise these bounds.
    
Compared to~\cite{deprez2015inhomogeneous} developed for graph distances, in order to show~\eqref{eq:tail_bound_k1} we additionally need to derive an upper bound on the probability that property \emph{(i)} of Definition~\ref{def:good_box} fails, i.e., to handle the cost of edges as well, not just their length. We do this using Lemma~\ref{lem:no_long_cheap_edge}.
\end{remark}

\begin{proof}[Proof of Proposition \ref{prop:large_box_good}]
 In order to unify notation, we define $\eta_0 := 1$ and $\delta := 0$ in Cases \ref{item:1} and \ref{item:3}, which is also consistent with the definition of $\eta_0$ in~\eqref{eq:eta_0}. So we need to show $(\eta_0-\delta)$-goodness in all three cases. We fix the values of $A_1$ and $u$ from Definition~\ref{def:good_box}. They will both depend on the set $\mpar$ of model parameters and on $\delta$ (in Case \ref{item:2}). We will first choose $A_1$ to be suitably large, and then choose $u$ to be suitably small as a function of $A_1$. We will prove the result using the Borel-Cantelli Lemma, by showing that $\sum_{k\ge k_0} \pr\left(Q_k \textnormal{ is not } (\eta_0-\delta)\textnormal{-good}\right)$ is a finite sum. 
By writing $\pr^{\max}(\cdot):=\max\{\pr(\cdot), \pr^{0}(\cdot), \pr^{0,x}(\cdot)\}$, let 
 \begin{equation}\label{eq:psi-k-def}
 \psi_k:=\pr^{\max}(Q_k \textnormal{ is not } (\eta_0-\delta)\textnormal{-good}).
 \end{equation}
 This choice ensures that $k$-blocks appearing in larger $k'$-blocks not centred at $0$ are also bad with probability at most $\psi_k$.
 We will show inductively that $\psi_k \le e^{-k}$. Note that our choice of $k_0=16\cdot 30^d$ in Definition~\ref{def:good_box} ensures that
	\begin{align}\label{eq:k_0_cond}
	    3^dk^{4d}e^{-2(k-1)} \le \tfrac{1}{2}e^{-k} \qquad\text{ for all $k\geq k_0$.}
	\end{align}
	
	\medskip\noindent\textbf{Base case.}
	We start by bounding $\psi_{k_0}$. Let $\E^{0}, \E^{0,x}$ denote expectation under the Palm measures $\pr^0, \pr^{0,x}$, respectively, and denote by $\E^{\max}(\cdot)$ the maximum of $\E[\cdot], \E^0[\cdot], \E^{0,x}[\cdot]$. 
 Let $|\calE(Q_{k_0})|$ be the number of edges inside  $Q_{k_0}$. Then $|\calE(Q_{k_0})|$ has finite expectation because the number of vertices in $Q_k$ is fixed (in SFP) or has a finite second moment (in IGIRG), and the number of edges with both endpoints within $Q_k$ is at most the square of the number of the vertices. The same is true under both Palm measures $\pr^{0}, \pr^{0,x}$.
 Let 
\begin{equation}
    T:=\E^{\max}[|\calE(Q_{k_0})|].
\end{equation}
We always assume $\mathbb P(L>0)=1$ (Assumption~\ref{assu:L} is a stronger assumption), so we can choose the constant $u\in (0,1)$ from Definition~\ref{def:good_box} small enough so that
\begin{equation}\label{eq:u_cond}
    F_L(u)\le e^{-2k_0}/(4T).
\end{equation}
Since the vertex-weights are always at least $1$ and $\mu\ge 0$, for each edge it holds that $\mathcal C(e) \ge L_e$. Hence by Markov's inequality and a union bound,
\begin{align*}
    \psi_{k_0} &= \pr^{\max}(\exists e \in Q_{k_0} \textnormal{ with } \cost{e} < u)
    \le \pr^{\max}(\exists e \in Q_{k_0} \textnormal{ with } L_e\le u) \\
    &\le\pr^{\max}(|\calE(Q_{k_0})|>2e^{k_0}T) + \pr^{\max}(\exists e \in Q_{k_0} \textnormal{ with } L_e\le u \mid |\calE(Q_{k_0})|\le 2e^{k_0}T) \\
    &\le \dfrac{1}{2}e^{-k_0} + 2e^{k_0}\E^{\max}[|\calE(Q_{k_0})|]\cdot \dfrac{e^{-2k_0}}{4T} = e^{-k_0}.
\end{align*}

\medskip\noindent\textbf{Bounding the failure probability of property \emph{(i)}.} 
For the induction step, we need to derive an upper bound for the probability that property \emph{(i)} in the definition of $(\eta_0-\delta)$-goodness fails.
 Since we work under the conditioning that $0\in \calV$, (and later possibly under the presence of an extra vertex at $x\in \R^d$), not all blocks have the same distribution under these measures. Yet, since Lemma \ref{lem:no_long_cheap_edge} is valid with or without the presence of these vertices, the same upper bound holds for all blocks when counting the expected number of edges that are long but too cheap. 
 Using Markov's inequality and translation invariance, for any fixed $j\in\{ -1, 0, 1 \}^d$ we have
	\begin{align*}
		&\pr^{\max}(Q_k+jA_{k-1}/2 \textnormal{ fails to have property \emph{(i)}}) \\
		&\le \E^{\max}[|\{vw\subset Q_k \textnormal{ edge} : |v-w|\ge A_{k-1}/100, \cost{vw}\le uA_k^{\eta_0-\delta}\}|] =: E(k).
	\end{align*}
	We will bound $E(k)$ from above using Lemma~\ref{lem:no_long_cheap_edge}. We will specify $\eps,a$ in Lemma \ref{lem:no_long_cheap_edge} later, only depending on \mpar. 
	Then we will set $A := A_k$, $N:=A_{k-1}/100<A$, and choose $A_1$ large enough (with respect to $\eps, a$) so that Lemma \ref{lem:no_long_cheap_edge} can be applied. When we are able to set $a\ge \mu$, with these choices of $A, N$ then we obtain, for some constant $C>0$ that depends only on \mpar,
	\begin{equation}\label{eq:general_E(k)-2}
	\begin{aligned}
	    E(k) &\le A_k^d\cdot\frac{A_{k-1}^{\eps}}{100^{\eps}}\left(\frac{A_{k-1}^{-d(\alpha-1)}}{100^{-d(\alpha-1)}}  + \frac{A_{k-1}^{-d(\tau-2)}}{100^{-d(\tau-2)}} \right) \\
	    &\le 100^{C}k^{2d}\left(A_{k-1}^{-d(\alpha-2)+\eps}  + A_{k-1}^{-d(\tau-3)+\eps}\right),
	   \end{aligned}
	\end{equation}
	where we substituted $A_k = k^2 A_{k-1}$ for the second inequality. Recall that this bound holds for any $L$ and does not require Assumption~\ref{assu:L}. Observe that the second term only tends to zero when $\tau>3$, so we can set $a\ge \mu$ only when $\tau>3$. For $\tau\le 3$, we will have to set  $a<\mu$, again $A := A_k$, $N:=A_{k-1}/100<A$. By possibly increasing the constant $C>0$, if $A_1$ is sufficiently large so that Lemma~\ref{lem:no_long_cheap_edge} can be applied then~\eqref{eq:no_long_cheap_edge} yields
	\begin{equation}\label{eq:general_E(k)}
	\begin{aligned}
	    E(k) &\le A_k^d\cdot\frac{A_{k-1}^{\eps}}{100^{\eps}}\left(\frac{A_{k-1}^{-d(\alpha-1)}}{100^{-d(\alpha-1)}} + \frac{A_{k-1}^{-d(\alpha-1-\frac{a}{\mu}(\alpha-(\tau-1)))}}{100^{-d(\alpha-1-\frac{a}{\mu}(\alpha-\tau+1))}} + \frac{A_{k-1}^{-d(\tau-2+(\mu-a)\beta)}}{100^{-d(\tau-2+(\mu-a)\beta)}} \right) \\
	    &\le 100^{C}k^{2d}\left(A_{k-1}^{-d(\alpha-2)+\eps} + A_{k-1}^{-d(\alpha-2-\frac{a}{\mu}(\alpha-(\tau-1)))+\eps} + A_{k-1}^{-d(\tau-3+(\mu-a)\beta)+\eps}\right),
	   \end{aligned}
	\end{equation}
	Note that this case of Lemma~\ref{lem:no_long_cheap_edge} required Assumption~\ref{assu:L} on $L$.
	
	We now split into cases depending on the values of $\tau$ and $\mu$ to specify $a$ and $\eps$. Proposition \ref{prop:large_box_good} always assumes $\alpha>2$ but we emphasise this for readability.
	
	\medskip\noindent\textbf{Case 1: $\boldsymbol{\tau >3, \alpha>2}$,  and  $\boldsymbol{\mu\ge 0}$, $L>0$ a.s., otherwise arbitrary.}
	In this case we have $\eta_0=1$ and $\delta=0$. Let us choose $a>\max\{\mu,1/d\}$. Then since $ad>1=\eta_0$ we can choose $A_1$ so large that
	\begin{align}\label{eq:A_1_cond_3}
	    \left(\frac{A_{k-1}}{100}\right)^{ad} \ge A_{k-1}k^2 = A_{k} \quad \textnormal{for all } k > k_0.
	\end{align}
	Since $\tau>3, \alpha>2$, we will set in \eqref{eq:general_E(k)-2}
	\begin{align}\label{eq:eps-1-def}
		\eps:=\frac{d}{2}\min\left\{\alpha-2, \tau-3 \right\}>0.
	\end{align}
	 Substituting this $\eps$ into the upper bound~\eqref{eq:general_E(k)-2} yields
	\begin{align*}
	    E(k)&\le  2\cdot100^{C}k^{2d}A_{k-1}^{-\eps} \le 100^{C+1} A_1^{-\eps} k^{2d} ((k-1)!)^{-2\eps}.
	\end{align*}	
	
		\medskip\noindent\textbf{Case 2: $\boldsymbol{\tau \in (2,3], \alpha>2}$, and $\boldsymbol{\mu \in (\mu_{\log}, \mu_{\mathrm{pol}}]}$.} In this case, $\delta >0$ and Assumption \ref{assu:L} needs to hold for $L$.
	Define
	\begin{align}\label{eq:local-a-case2}
		a:=\min\left\{\mu-\frac{3-\tau}{\beta}, \frac{\mu(\alpha-2)}{\alpha-(\tau-1)}\right\}-\frac{\delta}{2d} = \frac{\eta_0}{d} -\frac{\delta}{2d}, 
	\end{align}
	where the last equation can be seen by using $\eta_0$ from \eqref{eq:eta_0} and  $\mu_{\log} = (3-\tau)/\beta$ from  \eqref{eq:mu_pol_log}. Also note that $\eta_0>0$ since $\mu >\mu_{\log} = (3-\tau)/\beta$, and $\alpha > 2\ge (\tau-1)$, and that the statement of Proposition~\ref{prop:large_box_good} is stronger for smaller $\delta$. Therefore, we can assume that $a>0$. Note also that the first term of the minimum is at most $\mu$ and $\delta>0$, which implies $a<\mu$. (In fact, the second term of the minimum is also at most $\mu$.)
	Moreover, by the definition of $a$ and the fact that $\alpha>2$ implies $\alpha-(\tau-1)>0$, rearranging \eqref{eq:local-a-case2} yields that
	\begin{align}\label{eq:large-box-good-consts}
	    3-\tau-(\mu-a)\beta\le-\frac{\beta\delta}{2d}, \qquad 2-\alpha+\frac{a}{\mu}(\alpha-(\tau-1))\le -\frac{(\alpha-(\tau-1))\delta}{2\mu d}.
	\end{align}
	We choose $A_1$ large enough so that
	\begin{align}\label{eq:A_1_cond_2}
	    \left(\frac{A_{k-1}}{100}\right)^{\eta_0-\delta/2} = \left(\frac{A_k}{100k^2}\right)^{\eta_0-\delta/2}\ge A_k^{\eta_0-\delta} > uA_k^{\eta_0-\delta} \quad \textnormal{for all } k >k_0.
	\end{align}
	By definition of $a$, cf.\ equation~\eqref{eq:local-a-case2}, $(A_{k-1}/100)^{ad}=(A_{k-1}/100)^{\eta_0-\delta/2} > uA_k^{\eta_0-\delta}$ as desired. Using $\alpha-(\tau-1) >0$ the first term in the upper bound~\eqref{eq:general_E(k)} is dominated by the second term, and \eqref{eq:general_E(k)}  becomes
		\begin{align*}
		E(k)&\le 2\cdot100^{C}k^{2d}\left(
		A_{k-1}^{-d(\alpha-2-\frac{a}{\mu}(\alpha-\tau+1))+\eps} + A_{k-1}^{-d(\tau-3+(\mu-a)\beta)+\eps}\right).
	\end{align*}
	Now we set 
	\begin{equation}\label{eq:eps-2-def}
	    \eps:=\frac{\delta}{4}\min\left\{\beta, \frac{\alpha-(\tau-1)}{\mu}\right\}>0.
	\end{equation} 
	Using this $\eps$, combined with \eqref{eq:large-box-good-consts}, it follows that
	\begin{align*}
		E(k)&\le  100^{C+1}k^{2d}A_{k-1}^{-\eps} = 100^{C+1} A_1^{-\eps} k^{2d} ((k-1)!)^{-2\eps}.
	\end{align*}

	\medskip\noindent\textbf{Case 3: $\boldsymbol{\tau \in (2,3]}$, $\boldsymbol{\alpha>2}$, and $\boldsymbol{\mu >\mu_{\mathrm{pol}}}$.} Here again we need Assumption \ref{assu:L} to hold for $L$, and we have $\eta_0=1$ and $\delta = 0$. Remembering the definition of $\mu_{\mathrm{pol}}$ in~\eqref{eq:mu_pol_log}, rearranging $\mu>\mu_{\mathrm{pol}}$ yields that both  $d(\alpha-2)-(\alpha-\tau+1)/\mu>0$ and $d(\tau-3+(\mu-1/d)\beta)>0$, hold. So we will show that we can set in \eqref{eq:general_E(k)} the following values of $\eps$ and $a$:
	\begin{align}\label{eq:eps-3-def}
		\eps:=\frac{1}{4}\min\left\{d(\alpha-2)-\frac{\alpha-(\tau-1)}{\mu}, d\left(\tau-3+(\mu-1/d)\beta\right)\right\}>0,
	\end{align}
    and
	\begin{align}\label{eq:local-a-case3}
	    a := \frac{1}{2}\Big(\frac1d + \min\left\{\frac{\mu(\alpha-2)}{\alpha-(\tau-1)}, \mu-\frac{3-\tau}{\beta} \right\}\Big) >0.
	\end{align}
	We establish first that indeed $a<\mu$ and also that $ad>1$. 
	To see the first, observe that $a$ averages $1/d$ with the minimum of two different quantities.
	By \eqref{eq:mu_pol_log} we have $1/d\le\mu_{\mathrm{pol}} <\mu$. Since $\tau\in(2,3]$, $(\alpha-2)/(\alpha-(\tau-1))\le1$ and so $\mu(\alpha-2)/(\alpha-\tau+1)\le\mu$, and $\mu-(3-\tau)/\beta\le\mu$. This implies in particular that $a$ averages one quantity strictly smaller than $\mu$ with another quantity smaller or equal to $\mu$, showing that $a<\mu$. Since $\mu>\mu_{\mathrm{pol}}$, rearranging also yields
	\begin{align*}
	\frac{\mu(\alpha-2)}{\alpha-(\tau-1)} > \frac{1}{d}
	\quad \textnormal{and} \quad
	\mu-\frac{3-\tau}{\beta} > \frac{1}{d},
	\end{align*}
	so $a>1/d$. This implies $ad>1=\eta_0$, and thus, we can choose $A_1$ so large that
	\begin{align}\label{eq:A_1_cond_1}
	    \left(\frac{A_{k-1}}{100}\right)^{ad} \ge A_k = A_k^{\eta_0-\delta} > uA_k^{\eta_0-\delta} \quad \textnormal{for all } k > k_0.
	\end{align}
	Since $\alpha>2$ and $\tau\le 3$, we have $\alpha-(\tau-1)>0$ and thus, as in Case 2 the upper bound~\eqref{eq:general_E(k)} becomes
	\begin{align}\label{eq:Ek_case3}
		E(k)&\le 2\cdot100^{C}k^{2d}\left(
		A_{k-1}^{-d(\alpha-2-\frac{a}{\mu}(\alpha-(\tau-1)))+\eps} + A_{k-1}^{-d(\tau-3+(\mu-a)\beta)+\eps}\right).
	\end{align}
	The definition of $\eps$ in~\eqref{eq:eps-3-def} and $a$ in~\eqref{eq:local-a-case3} implies $2\eps \le \tfrac12d(\tau-3+(\mu-1/d)\beta)$ and $a\le\tfrac12(1/d + \mu -(3-\tau)/\beta)$. From these, it is elementary to check that
	\begin{align*}
	    -d(\tau-3+(\mu-a)\beta) \le -2\eps.
	\end{align*}
	Likewise, from $2\eps \le \tfrac12(d(\alpha-2)-(\alpha-(\tau-1)/\mu))$ and $a\le\tfrac12(1/d + \mu(\alpha-2)/(\alpha-\tau+1))$ we derive
    \begin{align*}
	    -d\left(\alpha-2-\frac{a}{\mu}(\alpha-\tau+1)\right) \le -2\eps,
	   \end{align*}
	so~\eqref{eq:Ek_case3} simplifies to
	\begin{align*}
		E(k) \le 100^{C+1}k^{2d}A_{k-1}^{-\eps} = 100^{C+1}A_1^{-\eps}k^{2d}((k-1)!)^{-2\eps}.
	\end{align*}
We remark that our approach does not give $\delta=0$ in the boundary case $\mu=\mu_{\mathrm{pol}}$ for $\tau \in (2,3)$, which would yield linear distances. The reason is that for $\delta =0$ we can not find a value of $a$ such that $ad > 1$ holds, and which give a negative exponent in the error bound~\eqref{eq:no_long_cheap_edge}. Since Lemma~\ref{lem:no_long_cheap_edge} is essentially tight, this means that there exist long edges whose cost is sublinear in their length.

	\medskip\noindent\textbf{Induction step.}
	Combining the three cases above, we conclude that in all cases we have for an appropriately chosen $\eps>0$ (depending on the values of $\mu$ and $\tau$)
	\begin{align}\label{eq:failing-property1}
		\pr^{\max}(Q_k+jA_{k-1}/2 \textnormal{ fails to have property \emph{(i)}}) \le E(k) \le 100^{C+1} A_1^{-\eps} k^{2d} ((k-1)!)^{-2\eps}.
	\end{align}
 Importantly, both $C$ and $\eps$ depend only on \mpar. We define \begin{equation}\label{eq:Cprime}
	    C'=C'(A_1):=100^{C+1} A_1^{-\eps},
	\end{equation}
	 which is a constant depending on \mpar and $A_1$.
	
	We note that for $k>k_0$, the $k$-block $Q_k$ is not $(\eta_0-\delta)$-good if at least one of its $3^d$ translations $Q_k+jA_{k-1}/2, j\in\{-1, 0, 1\}^d$ fails to have property \emph{(i)} or \emph{(ii)} of Definition~\ref{def:good_box}. The failure probability of property \emph{(i)} is bounded by \eqref{eq:failing-property1}.  Therefore, translation invariance together with a union bound, and recalling property \emph{(ii)}, implies that
	\begin{align*}
		\psi_k \le 3^d\left(C'k^{2d}((k-1)!)^{-2\eps} + \pr(\textnormal{at least } 3^d+1 \textnormal{ children of } Q_k \textnormal{ are not } (\eta_0-\delta)\textnormal{-good})\right).
	\end{align*}
	Observe the following: if a block has at least $3^d+1$ bad child-blocks, one can find at least one pair of non-neighbouring bad child-blocks. Hence the event in the probability above implies that there are at least two children-blocks $Q^{\star}_{k-1}$ and $Q^{\star\star}_{k-1}$ of $Q_k$ that are not $(\eta_0-\delta)$-good and whose centres have distance at least $2A_{k-1}$.  In particular, any vertex in $Q^{\star}_{k-1}$ is at least  $A_{k-1}$ Euclidean distance away from any vertex in $Q^{\star\star}_{k-1}$, which ensures that the events $\{Q^{\star}_{k-1} \textnormal{ is not } (\eta_0-\delta)\textnormal{-good}\}$ and $\{Q^{\star\star}_{k-1} \textnormal{ is not } (\eta_0-\delta)\textnormal{-good}\}$ are independent. Since there are $k^{2d}$ many child-blocks of $Q_k$, there is at most $\binom{k^{2d}}{2}\le k^{4d}$ many ways to choose $Q^{\star}_{k-1}$ and $Q^{\star\star}_{k-1}$, that are independently bad with probability at most $\psi_{k-1}$. So we deduce the bound 
	\begin{align*}
		\psi_k \le  3^d\left(C'k^{2d}((k-1)!)^{-\eps} + k^{4d}\psi_{k-1}^2\right).
	\end{align*}
	Note that $C'\rightarrow0$ as $A_1\rightarrow\infty$, by \eqref{eq:Cprime}. Thus, we can choose $A_1$ large enough (and this choice only depends on \mpar and $\eps$) so that for all $k > k_0$,
	\begin{align}\label{eq:A_1_cond}
		\psi_k \le \tfrac{1}{2}e^{-k} + 3^dk^{4d}\psi_{k-1}^2 .
	\end{align}
	We now prove by induction on $k$ that
	\begin{align}\label{eq:psi_k_bound}
		\psi_k \le e^{-k}
	\end{align}
	for all $k\ge k_0$. Indeed, we have already seen that~\eqref{eq:psi_k_bound} holds for $k=k_0$. For $k>k_0$, assuming that~\eqref{eq:psi_k_bound} holds for $k-1$, we get
	\begin{align*}
		&\psi_k\le \tfrac{1}{2}e^{-k} + 3^dk^{4d}\psi_{k-1}^2 \le \tfrac{1}{2}e^{-k} + 3^dk^{4d}e^{-2(k-1)} \stackrel{\eqref{eq:k_0_cond}}{\le} \tfrac{1}{2}e^{-k} + \tfrac{1}{2}e^{-k} = e^{-k}.
	\end{align*}
    This shows by induction that~\eqref{eq:psi_k_bound} holds for all $k\ge k_0$, which yields
	\begin{align*}
	   \sum_{k\ge k_0} \pr\left(Q_k \textnormal{ is not } (\eta_0-\delta)\textnormal{-good}\right) = \sum_{k\ge k_0}\psi_k \le \sum_{k\ge k_0} e^{-k} < +\infty 
	\end{align*}
	and concludes the proof of Proposition \ref{prop:large_box_good} by the Borel-Cantelli Lemma.
\end{proof}

\subsection{Paths within good blocks}\label{sec:inside-good-blocks}
The next proposition carries out the renormalisation scheme and is an important step in proving Theorem \ref{thm:linear_regime} and Theorem \ref{thm:polynomial_regime}.
\begin{proposition}\label{prop:path_in_good_block}
	Let $\eta\in(0,1]$ and recall $k_0$ from Definition~\ref{def:good_box}. There exists a constant $u^{\star}$ depending only on \mpar and $\eta$ such that for all $k\ge k_0$ the following holds. 
	
	(1) If $Q$ is an $\eta$-good $k$-block and $x,y\in Q$ satisfy $|x-y|> A_k/16$, then every path from $x$ to $y$ within $Q$ has cost at least $u^{\star}|x-y|^{\eta}$ (deterministically).
	
	(2) If the $(k-1)$-block $Q'$ and the $k$-block $Q$ centred at $x$ are both $\eta$-good, and if $y\in Q$ satisfies $|x-y|>A_{k-1}/8$, then every path from $x$ to $y$ within $Q$ has cost at least $u^{\star}|x-y|^{\eta}/  30^{d+2}$ (deterministically).
\end{proposition}

\begin{remark}
	Our proof of (1) is an adaptation of the proof of Lemma 2 in~\cite{berger2004lower},  with the difference that the continuous-valued edge-cost make the argument slightly more complicated. The statement of (2) allows us to prove strictly linear cost-distances in the case $\eta = 1$, avoiding Kingman's subadditive ergodic theorem that finishes the proof in \cite{berger2004lower}. Note that since (2) is symmetric, we could also require a condition for the blocks centred at $y$ instead of $x$.
\end{remark}

\begin{proof}[Proof of Proposition \ref{prop:path_in_good_block}]
\newconstant{cst:path_in_good_block}
We will show the following claim by induction: there exists a constant $\useconstant{cst:path_in_good_block}$ (which depends on the same parameters as $u^{\star}$ does) such that for every $k\ge k_0$, if $Q$ is an $\eta$-good $k$-block and $x,y\in Q$ satisfy $|x-y|> A_k/16$, then every path $\pi$ from $x$ to $y$ within $Q$ has cost at least
\begin{align}\label{eq:inductive_lower_bound}
	\mathcal C(\pi)\ge 	\useconstant{cst:path_in_good_block}\Lambda(k)|x-y|^{\eta}, \qquad \textnormal{where } \qquad \Lambda(k):=\prod_{h=k_0}^k\left(1-\frac{k_0}{h^2}\right).
\end{align}
Then taking $u^{\star}=\useconstant{cst:path_in_good_block} \prod_{h=k_0}^{\infty}(1-k_0/h^2)>0$ shows (1). To show (2), we will slightly modify the last step of the induction in (1).
To ease notation, we will assume that $Q=Q_k$, i.e., $Q$ is the $k$-block centred at the origin. 
	
For the base case, consider $x,y\in Q_{k_0}$ that satisfy $|x-y|>A_{k_0}/16$, where $Q_{k_0}$ is a $\eta$-good $k_0$-block. Then any path between $x$ and $y$ contains at least one edge and since $Q_{k_0}$ is an $\eta$-good $k_0$-block, the path has cost at least $u$ by Definition \ref{def:good_box}. It is here that we use the assumption that $\mathbb P(L > 0)=1$, compare it to the base case in the proof of Proposition \ref{prop:large_box_good}, i.e., at and below \eqref{eq:u_cond}.
	
Since $|x-y|\le A_{k_0}\sqrt{d}$, and $\eta\le 1$, we define $\useconstant{cst:path_in_good_block}:=u/(A_{k_0}\sqrt{d})$ so that~\eqref{eq:inductive_lower_bound} holds for $k=k_0$. Note that $\useconstant{cst:path_in_good_block}\le u/\sqrt{d}$.
	
Now assume that \eqref{eq:inductive_lower_bound} holds already until $k-1$ for some $k > k_0$. Let us consider an $\eta$-good $k$-block $Q_k$, and $x,y\in Q_k$ with $|x-y|>A_k/16$. Let $\pi = (v_1, \ldots, v_l)$ be a path from $v_1=x$ to $v_l=y$ within $Q_k$.

\emph{Case A:}	If $\pi$ contains one or more edges of length greater than $A_{k-1}/100$, then since $Q_k$ is $\eta$-good, by property \emph{(i)} of Definition \ref{def:good_box}, the path has cost at least
\begin{align*}
	uA_k^{\eta} \ge \frac{u|x-y|^{\eta}}{\sqrt{d}^{\eta}}  \ge \useconstant{cst:path_in_good_block}|x-y|^{\eta} \ge \useconstant{cst:path_in_good_block}\Lambda(k)|x-y|^{\eta},
\end{align*}
 since $\eta\le 1$, so we are done.	

\emph{Case B:} When every edge of $\pi$ has length at most $A_{k-1}/100$, we will split $\pi$ in subpaths as follows. Remember that for every translation $Q_k+jA_{k-1}/2$ of $Q_k$, $j\in\{-1, 0, 1\}^d$, at most $3^d$ children of $Q_k+jA_{k-1}/2$ are not $\eta$-good (call these bad), thus we have in total at most $9^d$ bad child-blocks in the union of all translations (some might overlap). Denote these bad children by $B_1,B_2, \ldots, B_{p}$ with $p\le9^d$, and let $B:=B_1 \cup\ldots\cup B_{p}$. If $\pi\cap B = \emptyset$, define $\pi_1:=\pi$. Otherwise, we will decompose $\pi$ into  \emph{good segments} $\pi_s$ followed by \emph{bad segments} $\sigma_t$, so that $\pi$ is the concatenation of the good and bad segments:  $\pi=(\pi_1, \sigma_1, \pi_2, \sigma_2, \ldots, \pi_{S-1}, \sigma_T, \pi_{S})$ for some $S, T$, where some of these segments might be empty, and the last vertex of a segment is the first vertex of the next segment.

We now divide the vertices of $\pi=(v_1, \ldots, v_l)$ into the segments. Intuitively, the edge-set of good segments stay fully outside of bad child-blocks.
Let $a_1$ be the smallest index $i\le l$ so that $v_i\in B$ and let $b_1$ be the index of the containing bad child-blocks: $v_{a_1}\in B_{b_1}$ (choose $b_1$ arbitrarily if there is more than one possibility). Then we set $\pi_1=(v_1, \ldots, v_{a_1-1})$ as the first good segment. Let $z_1$ be the largest value $z$ such that $v_z\in B_{b_1}$, then $\sigma_1=(v_{a_1-1}, \ldots, v_{z_1+1})$ is the first bad segment. (Note that there may be vertices on this segment outside $B_{b_1}$.) Inductively, let $a_{s+1}$ be the smallest $a>z_s$ so that $v_a\in B$, let $b_{s+1}$ be so that $v_{a_{s+1}}\in B_{b_{s+1}}$ and let $z_{s+1}$ be the largest $z$ with $v_z\in B_{b_{s+1}}$. The further good segments are then defined as $ \pi_2:=(v_{z_1+1}, \ldots, v_{a_2-1})$ and so on,  up to $\pi_S$, and the bad segments as $\sigma_1:=(v_{a_1-1}, \ldots, v_{z_1+1}), \sigma_2:=(v_{a_2-1}, \ldots, v_{z_2+1})$ and so on up to $\sigma_T$. Observe that the bad segment $\sigma_t$ contains the two edges $v_{a_t-1} v_{a_t}$ and $v_{z_t}v_{z_t+1}$ not fully contained in $B_{b_t}$. So, the good segments contain edges that are \emph{completely outside} the bad set $B$. Note that $S-1\le T\le p \le  9^d$, and $S-1$ may or may not equal $T$ since two bad segments may directly follow each other.
	
For a path $\rho$, denote by $\mathcal{D}(\rho)$ the Euclidean distance between its endpoints. Moreover, for $v,w\in\rho$, let $\rho[v, w]$ be the subpath (segment) from $v$ to $w$ on $\rho$. By the triangle inequality, $|x-y|\le \sum_{s=1}^S\mathcal{D}(\pi_s) + \sum_{t=1}^T\mathcal{D}(\sigma_{t})$. Moreover,  since the diameter of $B_{b_{t}}$ is $\sqrt{d}A_{k-1}$ (since $B_{b_t}$ is a level $k-1$ block), and every edge in $\pi$ has length at most $A_{k-1}/100$ (by assumption of Case B), for every bad segment $\sigma_{t}$ we have
\begin{align}\label{eq:bad-box-span}
	\mathcal{D}(\sigma_{t}) \le |v_{a_{t}}-v_{z_{t}}| + |v_{a_{t}-1}-v_{a_{t}}| + |v_{z_{t}}-v_{z_{t}+1}| \le \sqrt{d}A_{k-1}+2A_{k-1}/100 \le 2dA_{k-1}.
\end{align}
Let 
\begin{align}\label{eq:IS-def}
	I:=\{ s \in \{1,\ldots,S\} \mid \mathcal{D}(\pi_s)>A_{k-1}/2\}
\end{align}
be those good segments where the two endpoints span at least $A_{k-1}/2$ Euclidean distance. Then $\sum_{s\notin I} \mathcal{D}(\pi_s)\le SA_{k-1}/2$, trivially, since there are $S$ good segments. Since the inductive statement assumes $|x-y|>A_k/16$, this and \eqref{eq:bad-box-span} gives a somewhat involved pigeon-hole principle on the triangle inequality above:
\begin{align}
\begin{split} \label{eq:lower-bound-sum-distance}
	\sum_{s\in I} \mathcal{D}(\pi_s) &\ge |x-y| - \sum_{t=1}^T\mathcal{D}(\sigma_{t}) - \sum_{s\notin I} \mathcal{D}(\pi_s) \ge |x-y|-(2dT+S/2)A_{k-1} \\
	&{ \buildrel (\star) \over \ge}\  |x-y|-30^dA_{k-1}
	{ \buildrel (\dagger) \over \ge}\  \left(1-\frac{16\cdot30^d}{k^2}\right)|x-y|\ { \buildrel (\Box) \over  =}\  \left(1-\frac{k_0}{k^2}\right)|x-y|,
\end{split}
\end{align}
where we used that $S-1, T \le 9^d$ to get inequality $(\star)$, that $|x-y|\ge A_k/16$ and that $A_k=k^2 A_{k-1}$ to get $(\dagger)$, and then the definition of $k_0$ in Definition \ref{def:good_box} to get $(\Box)$. For later reference, we note that under the condition $|x-y|\ge 2\cdot 30^d A_{k-1}$ (that we will assume in proving (2)), the following weaker version of~\eqref{eq:lower-bound-sum-distance} still holds:
\begin{align}
\begin{split} \label{eq:lower-bound-sum-distance2}
	\sum_{s\in I} \mathcal{D}(\pi_s) \ge |x-y|-30^dA_{k-1}
	\ge \frac{1}{2}|x-y|.
\end{split}
\end{align}
Returning to the proof of \eqref{eq:inductive_lower_bound}, observe that \eqref{eq:lower-bound-sum-distance} bounds the total Euclidean distance $\mathcal D$ spanned by the endpoints of good segments $\pi_s$. Now we will work towards switching $\mathcal D$ to $\mathcal C$, the cost of the segment.
We claim that if $s\in I$, then
\begin{align}\label{eq:lower_bound_long_non_bad_path}
	\cost{\pi_s} \ge \useconstant{cst:path_in_good_block} \Lambda(k-1)\mathcal{D}(\pi_s)^{\eta}.
\end{align}
In order to prove~\eqref{eq:lower_bound_long_non_bad_path}, we will show inductively that we can partition $\pi_s=(v_{z_s+1}, \ldots,$ $v_{a_{s+1}-1})$ into sub-segments $\pi_{s,i}=\pi_s[x_{s,i}, y_{s,i}]$ (where $y_{s,i}=x_{s,i+1}$, i.e., the end-vertex of a sub-segment is the starting vertex of the next sub-segment) for $i=1,\ldots, q_s$, such that for all $i$:
\begin{enumerate}[(I)] 
    \item $\mathcal{D}(\pi_{s,i})> A_{k-1}/16$,
    \item $\pi_{s,i}\subseteq B_{A_{k-1}/4}(x_{s,i})$,
\end{enumerate}
i.e., the whole segment is contained in the Euclidean ball of radius $A_{k-1}/4$ centred around $x_{s,i}$ and its endpoints span enough Euclidean distance.
We will construct the $\pi_{s,i}$ greedily, with the induction hypothesis that if $\pi_s$ is \emph{not} covered by the first $i$ sub-segments (i.e.\ if $\pi_s\neq\cup_{j=1}^{i}\pi_{s,j}$), then $|y_{s,i}-v_{a_{s+1}-1}|>A_{k-1}/16$ (remembering that $v_{a_{s+1}-1}$ is the last vertex of $\pi_s$).
	
For the base case $i=0$, we take $y_{s,0}$ to be the first vertex of $\pi_s$, and since $s\in I$, by \eqref{eq:IS-def}, we have $|y_{s,0}-v_{a_{s+1}-1}| = \mathcal{D}(\pi_s) > A_{k-1}/2 > A_{k-1}/16$, so the induction hypothesis is satisfied. 
	
Now assume by induction that we already have $\pi_{s,1}, \ldots, \pi_{s,i}$ satisfying (I), (II) and $|y_{s,i}-v_{a_{s+1}-1}|>A_{k-1}/16$ for some $i\ge1$, and let us construct $\pi_{s,i+1}$.
	
\emph{Case B1:} If the segment $\pi_s[y_{s,i},v_{a_{s+1}-1}]$ satisfies (II), that is, if $\pi_s[y_{s,i}, v_{a_{s+1}-1}]\subseteq B_{A_{k-1}/4}(y_{s,i})$ then we define $q_s:=i+1$, $\pi_{s,q_s}:=\pi_s[y_{s,i},v_{a_{s+1}-1}]$ and the procedure terminates. $\pi_{s,q_s}$ also satisfies (I) since $\mathcal{D}(\pi_{s,q_s}) = |y_{s,i}-v_{a_{s+1}-1}|>A_{k-1}/16$ because of our assumption that $q_s>i$.

\emph{Case B2:} If $\pi_s[y_{s,i},v_{a_{s+1}-1}] \nsubseteq B_{A_{k-1}/4}(y_{s,i})$ we distinguish two cases depending on $|y_{s,i}-v_{a_{s+1}-1}|$.
	
\noindent\emph{Case B2a}: $|y_{s,i}-v_{a_{s+1}-1}|\ge 5A_{k-1}/32$. Define $\pi_{s,i+1}$ to be the path obtained by following $\pi_s$ from $x_{s,i+1}:=y_{s,i}$ until reaching the first vertex on $\pi_s$ that spans larger than $A_{k-1}/16$ Euclidean distance from $x_{s,i+1}$, and let this vertex be $y_{s, i+1}$. The vertex $y_{s, i+1}$ exists  since $|x_{s,i+1}-v_{a_{s+1}-1}| = |y_{s,i}-v_{a_{s+1}-1}| > A_{k-1}/16$ was our assumption. Then $\pi_{s,i+1}$ satisfies (I) by our definition of $y_{s, i+1}$. Since every edge of $\pi_s$ has length at most $A_{k-1}/100$ (since we are under Case B), and every vertex of $\pi_{s,i+1}=\pi_s[x_{s,i+1}, y_{s,i+1}]$ except $y_{s,i+1}$ is within distance $A_{k-1}/16$ of $x_{s,i+1}$, $\pi_{s,i+1}$ is also contained in a ball of radius $A_{k-1}/16+A_{k-1}/100<A_{k-1}/4$ centred at $x_{s,i+1}$, and thus (II) is satisfied as well. Finally, by the triangle inequality we have
\begin{align*}
    |y_{s,i+1}-v_{a_{s+1}-1}| &\ge |y_{s,i}-v_{a_{s+1}-1}|-|y_{s,i}-y_{s,i+1}|\\
    & \ge 5A_{k-1}/32-(A_{k-1}/16+A_{k-1}/100) > A_{k-1}/16,
\end{align*}
which shows the induction step for this case.
	
\noindent\emph{Case B2b}: $|y_{s,i}-v_{a_{s+1}-1}|< 5A_{k-1}/32$. Intuitively, this case means that while $y_{s,i}$ is close to $v_{a_{s+1}-1}$ in space, the path $\pi_s[y_{s,i},v_{a_{s+1}-1}]$ wanders far off and then comes back before reaching the last vertex $v_{a_{s+1}-1}$.
Let $y'$ be the first vertex on the path $\pi_s$ starting from $x_{s,i+1}:=y_{s,i}$ such that $|y'-x_{s,i+1}|\ge A_{k-1}/4$ (this vertex must exist since $\pi_s[y_{s,i}, v_{a_{s+1}-1}] \nsubseteq B_{A_{k-1}/4}(y_{s,i})$, since we are under Case B2). Then define $y_{s,i+1}$ to be the vertex right before $y'$ on $\pi_s[x_{s,i+1},y']$ and let $\pi_{s,i+1}:=\pi_s[x_{s,i+1},y_{s,i+1}]$. Again, since every edge of $\pi_s$ has length at most $A_{k-1}/100$ (since we are under Case B), we have $\mathcal{D}(\pi_{s,i+1}) \ge |x_{s,i+1}-y'|-|y_{s,i+1}-y'| \ge A_{k-1}/4-A_{k-1}/100 > A_{k-1}/16$ so (I) is satisfied. Moreover, $\pi_{s,i+1}$ satisfies (II) by our definition of $y_{s,i+1}$. Finally, we use the triangle inequality to get
\begin{align*}
    |y_{s,i+1}-v_{a_{s+1}-1}| &\ge |y'-x_{s,i+1}| - |y_{s,i+1}-y'| - |v_{a_{s+1}-1}-x_{s,i+1}| \\
    &\ge A_{k-1}/4-A_{k-1}/100-5A_{k-1}/32 > A_{k-1}/16,
\end{align*}
which concludes the induction step.

The reason for requiring (II) is that the obtained partition of $\pi_s$ has the property that every sub-segment $\pi_{s,i}$ \emph{is contained in a $\eta$-good $(k-1)$-block} (among the children of $Q_k$ and its translations). This is due to a geometric statement: any ball of radius $A_{k-1}/4$ within $Q_k$ is contained in at least one $k-1$-level block that is a child of some shift $Q_k+\underline j A_{k-1}/2$ for some $\underline j \in \{-1,0,1\}^d$. By (I) and the induction hypothesis~\eqref{eq:inductive_lower_bound}, we have $\cost{\pi_{s,i}} \ge \useconstant{cst:path_in_good_block} \Lambda(k-1)\mathcal{D}(\pi_{s,i})^{\eta}$ and therefore
\begin{align*}
	\cost{\pi_s} &= \sum_{i=1}^{q_s}\cost{\pi_{s,i}} \ge \useconstant{cst:path_in_good_block} \Lambda(k-1)\sum_{i=1}^{q_s} \mathcal{D}(\pi_{s,i})^{\eta} \\
	&{\buildrel (\star) \over \ge}\  \useconstant{cst:path_in_good_block} \Lambda(k-1)\left(\sum_{i=1}^{q_s} \mathcal{D}(\pi_{s,i})\right)^{\eta}
	{\buildrel (\triangle) \over \ge} \useconstant{cst:path_in_good_block} \Lambda(k-1)\mathcal{D}(\pi_s)^{\eta},
\end{align*}
where we got inequality $(\star)$ as follows: since  $\eta\le 1$ the function $x^\eta$ is sublinear, i.e., $x^{\eta} + y^{\eta} \ge (x+y)^{\eta}$, and $(\triangle)$ is the triangle inequality applied on $\pi_s$ and its sub-segments. This shows~\eqref{eq:lower_bound_long_non_bad_path}.
Recall now the set $I$ from \eqref{eq:IS-def}, and the pigeon-hole argument in \eqref{eq:lower-bound-sum-distance}, which holds under the assumption $|x-y|\ge A_k/16$. By the same sublinearity argument $(\star)$ of the function $x^\eta$, we can finally deduce that
\begin{align}
\begin{split}\label{eq:lower-bound-cost-sum}
	\cost{\pi} &\ge \sum_{s\in I}\cost{\pi_s}
	\stackrel{\eqref{eq:lower_bound_long_non_bad_path}}{\ge} \useconstant{cst:path_in_good_block}\Lambda(k-1)\sum_{s\in I} \mathcal{D}(\pi_s)^{\eta}
		\stackrel{(\star)}{\ge}
	 \useconstant{cst:path_in_good_block} \Lambda(k-1)\bigg(\sum_{s\in I} \mathcal{D}(\pi_s)\bigg)^{\eta} \\
	&\stackrel{\eqref{eq:lower-bound-sum-distance}}{\ge} \useconstant{cst:path_in_good_block} \Lambda(k-1) \left(1-\frac{k_0}{k^2}\right)^{\eta} |x-y|^{\eta} \
\stackrel{\eta \le 1}{\ge} \ \useconstant{cst:path_in_good_block} \Lambda(k)|x-y|^{\eta},
\end{split}
\end{align}
where $\Lambda(k)$ is defined in \eqref{eq:inductive_lower_bound}.
This finishes the inductive demonstration of~\eqref{eq:inductive_lower_bound} and concludes the proof of part (1) of the proposition.\smallskip
	
For part (2), we discriminate two sub-cases. If $|x-y|\ge 2\cdot 30^dA_{k-1}$, then~\eqref{eq:lower-bound-sum-distance2} holds, and we can repeat the calculation in~\eqref{eq:lower-bound-cost-sum}, getting  
\begin{align*}
	\cost{\pi} \ge
	 \useconstant{cst:path_in_good_block} \Lambda(k-1)\left(\sum_{s\in I} \mathcal{D}(\pi_s)\right)^{\eta}
	&\stackrel{\eqref{eq:lower-bound-sum-distance2}}{\ge} 2^{-\eta}\useconstant{cst:path_in_good_block} \Lambda(k-1) |x-y|^{\eta},
\end{align*}
which is stronger than required. So consider the remaining case, $A_{k-1}/8 < |x-y|< 2\cdot 30^dA_{k-1}$. As before, if the path $\pi$ from $x$ to $y$ contains an edge of length more than $A_{k-1}/100$, the claim follows from $Q_k$ being $\eta$-good (Case A). So assume otherwise. Then there exists a vertex on $\pi$ with $|x-v| \in (A_{k-1}/16, A_{k-1}/8]$. Let $v$ be the first such vertex on $\pi$ (starting from $x$), and let $\pi'=\pi[x,v]$. Then $v \in Q'$ (the (k-1)-block centred at $x$) and $\pi' \subseteq Q'$. Since $Q'$ is good by prerequisite, we may apply part (1) to the path $\pi'$ and conclude
\begin{align*}
	\cost{\pi} \ge \cost{\pi'} \ge 
	 u^{\star}|x-v|^{\eta} \geq u^{\star}\left(\frac{A_{k-1}}{16}\right)^{\eta} \geq u^{\star}\left(\frac{|x-y|}{32\cdot 30^d}\right)^{\eta} \ \stackrel{\eta \le 1}{\ge} \
	 u^{\star}|x-y|^\eta/30^{d+2},
\end{align*}
as required, finishing the proof of Proposition \ref{prop:path_in_good_block}.
\end{proof}
\setcounter{constant}{0}

\subsection{Proofs of lower bounds}\label{sec:lower_bound_polynomial}

We are now ready to prove Theorem~\ref{thm:linear_polynomial_lower_bound}.

\begin{proof}[Proof of Theorem~\ref{thm:linear_polynomial_lower_bound}]\label{proof:linear_polynomial_lower_bound}
    We will prove both claims~\eqref{eq:lin-poly-lower-bound} and~\eqref{eq:linear-lower} simultaneously by showing that the bound
    \begin{align}\label{eq:lin-poly-lower-bound-1}
        d_{\mathcal{C}}(0,x) \ge \frac{u^{\star}|x|^{\eta}}{30^{d+2}\sqrt{d}}
    \end{align}
    holds simultaneously for all $x$ with sufficiently large $|x|=:r$, where $u^\star$ is the constant from Proposition~\ref{prop:large_box_good}. The only difference between the two statements is that we prove \eqref{eq:lin-poly-lower-bound-1} for all $\eta < \eta_0$ if $\mu \in (\mu_{\log},\mu_{\mathrm{pol}}]$, and for $\eta := \eta_0 = 1$ in the cases $\mu > \mu_{\mathrm{pol}}$ or $\tau>3$. 
    
    By Proposition~\ref{prop:large_box_good}, a.s.\ there is some (random) $k_1\ge k_0$ such that for all $k\ge k_1$, the $k$-block $Q_k$ centred around $0$ is $\eta$-good in the sense of Definition \ref{def:good_box}: every edge of length larger than $A_{k-1}/100$, internal to $Q_k$ and its shifts $Q'=Q_k+\underline j A_{k-1}/2$ by vectors $\underline j\in \{-1,0,1\}^d$, has cost at least $u^\star A_k^\eta$, and there are at most $3^d$ bad child-blocks of $Q_k$ and/or the $Q'$s. This is the only place where we need to discriminate between the cases $\tau\in(2,3)$ and $\mu \in (\mu_{\log},\mu_{\mathrm{pol}}]$, or $\tau\in(2,3)$ and $\mu > \mu_{\mathrm{pol}}$; or $\tau >3$.  Proposition~\ref{prop:large_box_good} allows us to choose $\eta := \eta_0=1$ only in the latter two cases, otherwise we choose $\eta=\eta_0-\delta$ for an arbitrarily small $\delta>0$.
    
    For each vertex $x\in\R^d$, define $k^{\star}(x)$ to be the smallest integer $k\in\N$ with $x \in Q_k$. 
    We will show that~\eqref{eq:lin-poly-lower-bound-1} holds for all $x$ with $k^{\star}(x) \ge k_1 + 1$, which in particular implies the inequality for all $|x| > \sqrt{d}A_{k_1}/2$. So let us fix some $x$ and suppose that $k^{\star}:=k^{\star}(x) \ge k_1+1$. In particular we have $|x|\le \sqrt{d}A_{k^{\star}}/2$. Let $\pi$ be a path from $0$ to $x$, and let $k'$ be the smallest index such that $\pi \subseteq Q_{k'}$. Note that $k' \ge k^\star$ because $x \not\in Q_{k^\star-1}$. Then there exists a vertex $v$ on $\pi$ in $Q_{k'} \setminus Q_{k'-1}$, and in particular $|v| > A_{k'-1}/2$. If $k'=k^\star$ then we pick $v:=x$, otherwise we pick any such vertex $v$. By the second part of Proposition~\ref{prop:path_in_good_block}, we have
\begin{equation}\label{eq:cost-lower-intermed}
        \cost{\pi}\ge u^{\star}|v|^{\eta}/30^{d+2}.
    \end{equation}
    In the case $k'=k^\star$, \eqref{eq:lin-poly-lower-bound-1} follows from $x=v$ and thus $|v|=|x|$. In the case $k'>k^\star$, it follows from $|v| > A_{k'-1}/2 \ge A_{k^\star}/2 \ge |x|/\sqrt{d}$ and $\eta \le 1$, finishing the proof of~\eqref{eq:lin-poly-lower-bound} and \eqref{eq:linear-lower}.
\end{proof}   

Theorem~\ref{thm:linear_polynomial_lower_bound} contains Theorem~\ref{thm:polynomial_regime} and the lower bound in Theorem~\ref{thm:linear_regime} as special cases. Moreover, Corollaries~\ref{cor:ball-growth} and \ref{cor:polynomial-total} follow immediately from Theorem~\ref{thm:polynomial_regime} and from Theorems~\ref{thm:polynomial_regime} and~\ref{thm:polynomial-upper}, respectively. Finally, the lower bound in Theorem~\ref{thm:finite_graph} is a trivial consequence of Theorem~\ref{thm:linear_regime} since GIRG is a subgraph of IGIRG, which can only increase distances. 
We now prove Corollary \ref{cor:linear-dev}, and then 
the limit cases of Theorem~\ref{thm:threshold_regimes} in the next subsection.

\begin{proof}[Proof of Corollary \ref{cor:linear-dev}.]
To prove this corollary we assume the upper bound of Theorem~\ref{thm:linear_regime}, which is proved in Section~\ref{sec:linear-regime}. In other words we assume that there exists a constant $\kappa_2$ such that $\lim_{|x|\to \infty}\pr(\cost{\pi_{0,x}^\star} < \kappa_2 |x|)=1.$
Given this constant $\kappa_2$ we define $\kappa_3:=\kappa_2 30^{d+2}/u^\star$. Let $\pi_{0,x}$ be any path that leaves the ball $B_{\kappa_3|x|}(0)$. We prove now using \eqref{eq:cost-lower-intermed} that $\pi_{0,x}$ cannot be optimal. Indeed, let $v\in \pi_{0,x}$ be any vertex with $|v|\ge \kappa_3|x|$, and apply \eqref{eq:cost-lower-intermed} to $v$, which gives with $\eta=1$ that 
\begin{equation*}
    \calC(\pi_{0,x}) \ge u^\star |v|/30^{d+2} \ge u^\star \kappa_3|x|/30^{d+2} = \kappa_2 |x|.
\end{equation*}
Hence, with probability tending to $1$, any such path $\pi_{0,x}$ cannot be optimal. 
\end{proof}

\subsection{Limit Cases, Lower Bounds}\label{sec:limit-cases-lower}

In this section, we prove the lower bounds in the limit cases, i.e., we prove the lower bounds in Theorem~\ref{thm:threshold_regimes}. We achieve this by coupling a model with $\alpha=\infty$ to a suitable model with $\alpha < \infty$, and likewise for $\beta = \infty$. For $\alpha=\infty$, we will keep the vertex set and the weights identical in the two coupled models, but we use a subset of edges of the $\alpha<\infty$ model to obtain the $\alpha=\infty$ model (with identical costs, where edges are kept). For $\beta =\infty$, we can use the same vertex set and weights, and the same edge set in the two coupled models with $\beta=\infty$ and $\beta<\infty$, but we decrease the edge costs. 

\begin{proof}[Proof of Theorem~\ref{thm:threshold_regimes} (lower bounds)]\label{proof:threshhold_regimes_lower}
\emph{(a)} First consider the case $\alpha = \infty$. For Theorem~\ref{thm:polynomial_regime}, we need to study the case $\mu > \mu_{\log} = (3-\tau)/\beta$ from \eqref{eq:alpha-infty-definitions} and show that for sufficiently large $|x|$,
\begin{equation}\label{eq:limit-poly-rep}
    d_{\mathcal  C}(0,x) \ge|x|^{\eta_{0,\infty}-\varepsilon},
\end{equation}
where, recalling that $\mu_{\mathrm{pol}}=(3-\tau)/\beta + 1/d$ in the $\alpha=\infty$ case (see \eqref{eq:alpha-infty-definitions}),
\begin{align*}
    \eta_{0,\infty} = \begin{cases}
	1 & \mbox{ if $\mu>\mu_{\mathrm{pol}}$,}\\
	d\cdot (\mu-\mu_{\log}) & \mbox{ if $\mu\le\mu_{\mathrm{pol}}$.}
	\end{cases}
\end{align*}
To show \eqref{eq:limit-poly-rep}, we will show that IGIRG (and also GIRG and SFP) are \emph{monotone} in $\alpha$ in the following sense. We first explain it for $\alpha <\infty$. Let $\alpha, \alpha' \in (1,\infty)$ with $\alpha' <\alpha$. Let $\overline c \ge \underline c >0$, and consider a realisation of IGIRG, say $G$, with parameters $\alpha, \overline c, \underline c >0$, and arbitrary other parameters. Set $\overline c' := \underline c' := \overline c$ and consider a second IGIRG $G'$ with parameters $\alpha',\overline c',\underline c'$, and otherwise identical parameters to $G$. Then conditioned on the vertex set $\mathcal V$ and the weight vector $\mathcal W$, by~\eqref{eq:connection_prob}, the edge probabilities of $G$ (of $G'$) are given by a function $h$ (a function $h'$), and $h$ and $h'$ satisfy the relation 
\begin{align*}
    h(x,w_1,w_2) \le \overline c \cdot\min\{1,(w_1w_2)/|x|^d\}^\alpha \le \underline c' \cdot\min\{1,(w_1w_2)/|x|^d\}^{\alpha'} \le h'(x,w_1,w_2).
\end{align*}  Therefore, we can couple $G$ to $G'$ such that $G$ is a subgraph of $G'$. Note that $G'$ has the same model parameters as $G$ except for $\alpha'$, and except for the parameters $\overline c,\underline c$ whose values do not appear in Theorem~\ref{thm:polynomial_regime}, nor in any of the other theorems. For $\alpha=\infty$ and any $\alpha' \in (1,\infty)$, the same coupling is possible, as we show next. In this case, the defining equation~\eqref{eq:alpha_infty} for $h$ includes two parameters $\underline{c}, c' >0$, and requires that $h(x,w_1,w_2) = 0$ for $(w_1w_2)/|x|^d < c'$ and $h(x,w_1,w_2) \geq \underline c$ for $(w_1w_2)/|x|^d \ge 1$. In all other cases, we still have $h(x,w_1,w_2) \leq 1$ because it is a probability. Hence by setting 
\begin{align*}
    h'(x,w_1,w_2) := 
    \begin{cases} 
    \underline c \cdot \min\{1,(w_1w_2)/|x|^d\}^{\alpha'} & \text{ if } \tfrac{w_1w_2}{|x|^d} < c', \\
    1 & \text{ if }\tfrac{w_1w_2}{|x|^d} \ge c',
    \end{cases}
\end{align*} 
we can ensure that $h'(x,w_1,w_2) \ge h(x,w_1,w_2)$ for all $x,w_1,w_2$. Moreover, it can easily be checked that $h'$ satisfies the conditions in~\eqref{eq:connection_prob} for $\alpha'$ with $\overline c' := \max\{1,\underline{c}, (c')^{-\alpha'}\}$ and $\underline c' := \min\{1,\underline c\}$, i.e., we obtain an IGIRG model with parameter $\alpha'$.

To resume the proof of Theorem~\ref{thm:threshold_regimes}, consider IGIRG $G$ for $\alpha = \infty$, and for some $\beta,\mu,\tau$ such that $\mu > \mu_{\mathrm{log}}$, and let $\eps >0$. We claim that there exists $\alpha' <\infty$ such that $\eta_0' := \eta_0(\alpha',\beta,\mu,\tau) > \eta_{0,\infty}-\tfrac\eps2$. Indeed, this follows because we have defined $\eta_{0,\infty}$ such that $\lim_{\alpha'\to\infty} \eta_0(\alpha',\beta,\mu,\tau) = \eta_{0,\infty}$, as can be seen by comparing $\eta_0$ in \eqref{eq:eta_0} (see also\eqref{eq:mu_pol_log}) to $\eta_{0, \infty}$ in \eqref{eq:alpha-infty-definitions}. Now we couple $G$ with an IGIRG $G'$ with parameter $\alpha'$ (and different $\underline c, \overline c$), but identical parameters $\beta,\mu,\tau$. Theorem~\ref{thm:polynomial_regime} applies to $G'$, so we obtain for sufficiently large $|x|$,
\begin{equation}\label{eq:limit-poly-rep2}
     d_{\mathcal  C}^{G'}(0,x) \ge|x|^{\eta_{0,\infty}-\varepsilon}.
\end{equation}
Since $G$ is a subgraph of $G'$, cost-distances in $G$ are larger, so~\eqref{eq:limit-poly-rep2} remains true if we replace $d_{\mathcal  C}^{G'}$ by $d_{\mathcal  C}^{G}$. 

Likewise, for the linear lower bound in Theorem~\ref{thm:linear_regime} with $\alpha=\infty$, for any $\mu > \mu_{\mathrm{pol}}$ (i.e, $\eta_{0,\infty} =1$) we can find $\alpha'$ such that $\eta_0(\alpha',\beta,\mu,\tau) =1$ with the same parameters $\beta,\mu,\tau$. Hence, the same coupling also shows that for sufficiently large $|x|$,
\begin{align*}
		d_{\mathcal{C}}(0,x) > \kappa_1 |x|.
\end{align*}
This implies the lower bound in Theorem~\ref{thm:linear_regime} for $\alpha = \infty$.

Finally, the same argument also holds for GIRG, since for fixed $r>0$, two randomly chosen vertices $u_n,v_n$ in a box of volume $n$ will a.a.s. satisfy $|u_n - v_n| \ge r$ as $n\to \infty$. Thus the lower bound in Theorem~\ref{thm:finite_graph} is also implied. This concludes the case $\alpha = \infty$.
\smallskip

\emph{(b)}\&\emph{(c)} 
Now we come to the case $\beta = \infty$. Note that we have shown part \emph{(a)} already, so in particular the couplings that we construct now are also valid when $\alpha=\beta=\infty$. The idea is the same as for $\alpha=\infty$, but the coupling is much easier. Since the parameter $\beta$ does not influence the graph structure, we can use the same graph, but with different transmission costs. More precisely, consider IGIRG $G$ with $\beta=\infty$, i.e., the cumulative distribution $F_L$ of the cost variables satisfies $\lim_{t\to 0} F_L(t)/t^{\beta} = 0 \mbox{ for all }0<\beta <\infty$. For any $\beta'\in(0,\infty)$, there exists $t_0\in (0,1)$ such that the distribution $F_L'$ given by $F_L'(t) = t^{\beta'}$ for $t<t_0$ and $F_L'(t) = 1$ for $t\geq t_0$  dominates $F_L$. The distribution $F_L'$ satisfies condition~\eqref{eq:F_L-condition}, so it yields an IGIRG model. Hence, we can couple $G$ to an IGIRG $G'$ with any parameter $\beta' \in (0,\infty)$ (and otherwise identical parameters except for the parameters $t_0,c_1,c_2$ that appear in~\eqref{eq:F_L-condition}) such that $G$ and $G'$ are the same graph, and for every edge the transmission cost in $G$ is at least as large as in~$G'$.

We can now turn to proving the lower bound in Theorem~\ref{thm:polynomial_regime} for $\beta =\infty$, so consider IGIRG $G$ for $\beta =\infty$, and let $\mu >0 = \mu_{\mathrm{log}}$. According to~\eqref{eq:beta-infty-definitions}, in this case we set $\eta_{0,\infty} := 1$ for $\mu > \mu_{\mathrm{pol}}$ and $\eta_{0,\infty} := \min\{d\mu,\mu/\mu_{\mathrm{pol},\alpha}\}$ otherwise. Since $\eta_{0,\infty} = \lim_{\beta'\to\infty} \eta_0(\alpha,\beta',\mu,\tau)$, we can find $\beta' \in (0,\infty)$ such that $\eta_0' := \eta_0(\alpha,\beta',\mu,\tau) \geq \eta_{0,\infty} -\tfrac\eps2$, and couple $G$ to an IGIRG $G'$ with parameter $\beta'$, as described above. Moreover, by choosing $\beta'$ large enough, we can also ensure that $ \mu > \mu_{\mathrm{log}}(\beta',\tau,\mu) = (3-\tau)/\beta'$. Then we can apply Theorem~\ref{thm:polynomial_regime} to $G'$, and obtain that for sufficiently large $|x|$,
\begin{equation*}
    d_{\mathcal  C}^{G'}(0,x) \ge|x|^{\eta_0'-\eps/2}.
\end{equation*}
Since cost-distances in $G'$ are less or equal than cost-distances in $G$, the statement remains true if we replace $d_{\mathcal  C}^{G'}$ by $d_{\mathcal  C}^{G}$, and since $\eta_0'-\tfrac\eps2 \geq \eta_{0,\infty}-\eps$, we may replace $\eta_0'-\tfrac\eps2$ by $\eta_{0,\infty}-\eps$. This yields the lower bound of Theorem~\ref{thm:polynomial_regime} for $G$. As before, the proof for SFP is verbatim the same, and the proofs of the linear lower bound (i.e., the lower bound in Theorem~\ref{thm:linear_regime}) and the lower bound for GIRG (the lower bound of Theorem~\ref{thm:finite_graph}) are analogous. This concludes the case $\beta =\infty$, and concludes the proof of the lower bounds of Theorem~\ref{thm:threshold_regimes}.
\end{proof}

%% file: linear_regime.tex
\section{Linear Regime, Upper Bound}\label{sec:linear-regime}

In the following section, we prove the upper bound of Theorem~\ref{thm:linear_regime}, that cost-distance in IGIRG and SFP scales at most linearly with Euclidean distance and the corresponding part of Theorem~\ref{thm:finite_graph} for finite GIRGs in a box. We shall make the renormalisation argument outlined on page \pageref{proof:idea-upper} quantitative. We start by classical results on iid bond percolation, but first a definition:
\begin{definition}[Deviation from straight line]\label{def:deviation}
    Given $u,v \in \R^d$, let $S_{u,v}$ denote the line segment between $u,v$. For $x\in \R^d$ we define the deviation $\mathrm{dev}_{uv}(x):=\|x-S_{u,v}\|$. Given a path $\pi=x_1\ldots x_k$ in a graph $G$ with vertices in $\R^d$, we define the \emph{deviation of $\pi$ from $S_{uv}$} as $\mathrm{dev}_{uv}(\pi):=\max\{\mathrm{dev}_{uv}(x_i) \colon i \in [k]\}$. Finally, let  the \emph{deviation of $\pi$} be $\mathrm{dev}(\pi) := \max\{\mathrm{dev}_{x_1x_k}(x_i) \colon i \in [k]\}$, i.e., its deviation from the segment between the endpoints.
\end{definition}
The next two lemmas state two properties for bond percolation with high enough edge-density in $\Z^d$. First, the density of the infinite component is large in every ball $B_{(\log r)^{3/2}}(y)$, where $y$ may vary through $r^d$ vertices around the origin. Second, vertices in the infinite component can be joined by a path of linear length with small deviation. The first statement is implicit in~\cite{DP-percolation}, the second in~\cite{antal1996chemical}. We give proofs in the appendix for completeness on pages \pageref{proof:lemma3.2} and \pageref{proof:lemma3.3} respectively. Recall the notation $a\ll_\star b,c$ and $a\gg_\star b,c$ meaning that $a$ needs to be respectively sufficiently small/large in terms of the parameters $b,c$ from Section \ref{sec:notation}.
\begin{restatable}[Locally high-density infinite cluster in bond-percolation]{lemma}{BondPercolationDensity}
\label{lem:bond-percolation-density}
    Let $d \in \N$ with $d \ge 2$, let $\eps,\delta,\sigma \in (0,1)$, and let $r>0$.  Let $\omega^\star$ be an iid Bernoulli bond percolation on $\Z^d$ with edge-retention probability $p := 1-\eps$. Then whenever $r \ggs \eps, \delta$, and $\eps \lls\sigma,d$, then almost surely $\omega^\star$ has a unique infinite component $\calC_\infty^\star$ with $\pr(0\in \calC_\infty^\star)\ge 1-\sigma$ and
   \begin{equation}\label{eq:bond-percolation-density}
        \forall x\in \Z^d:\quad \pr\Big(\forall y\in B_r(x), \, \frac{|B_{(\log r)^{3/2}}(y) \cap \calC_\infty^\star|}{|B_{(\log r)^{3/2}}(y)\cap\Z^d|} \ge 1-\sigma\Big) \ge 1-\delta.
    \end{equation}
\end{restatable}

\begin{restatable}[Linear distances in bond-percolation]{lemma}{BondPercolationLinear}
\label{lem:bond-percolation-linear}
    Let $d \in \N$ with $d \ge 2$, let $\zeta,\eps,c \in (0,1)$, and let $\kappa,r>0$. Let $\omega^\star$ be an iid Bernoulli bond percolation on $\Z^d$ with edge-retention probability $p := 1-\eps$. Let $\calC_\infty^\star$ be the infinite component of $\omega^\star$. For all $x\in\Z^d$, let $\calA^\star_{\mathrm{linear}}(r,\kappa,\zeta,x)$ be the event that for all $y\in\calC_\infty^\star \setminus B_r(x)$ there  is a path $\pi$ from $x$ to $y$ with length at most $|\pi|\le\kappa |x-y|$ and deviation at most $\mathrm{dev}(\pi)\le \zeta |x-y|$. Then whenever $r\ggs \zeta,\eps$, and $\kappa, 1/\eps,1/c \ggs d$,
    \begin{align}\label{eq:bond-percolation-distances}
        \pr(\calA^\star_{\mathrm{linear}}(r,\kappa,\zeta,x) \mid x\in\calC_\infty^\star)\ge 1-e^{-cr}.
    \end{align}
\end{restatable}
\subsection{General high-density random geometric graphs and renormalisation-coupling}
In order to apply Lemmas~\ref{lem:bond-percolation-density} and~\ref{lem:bond-percolation-linear} to IGIRG and SFP, we will need a coupling to bond percolation. We provide this in Lemma~\ref{lem:bond-percolation-coupling} for any graph model satisfying Definition~\ref{def:dense-geometric} below. Given a graph $G = (\mathcal V,\mathcal E)$ with vertex set $\mathcal V \subseteq \R^d$ and a set $A \subseteq \R^d$, we write $G[A]$ for the induced subgraph of $G$ on $\mathcal V[A] := \mathcal V\cap A$. Two (half-open) boxes are called \emph{neighbouring} if their closures have non-empty intersection.
\begin{definition}[Dense geometric random graphs, generally]\label{def:dense-geometric}
    Let $d \in \N$  and let $\calS$ be a partition of $\R^d$ into half-open boxes of the form $[a_1,a_1+R)\times\dots\times[a_d,a_d+R)$ for some $a_1,\dots,a_d \in \R$ and side-length $R>0$.
    When $\calS$ is given, for all $z \in \Z^d$, we write $S_z$ for the unique box in $\calS$ containing $R\!\cdot \!z$.
    Let $G=(\calV,\calE)$ be a random graph whose vertex set is either $\Z^d$ or is given by a homogeneous PPP on $\R^d$. For all $\eps > 0$, we say $G$ is an \emph{$(R,\eps)$-dense geometric graph} with \emph{boxing scheme} $\calS$ if it satisfies the following properties (i)-(iii), and we call $G$ a \emph{strong $(R,D,\eps)$-dense geometric graph} if it satisfies (i)-(iv): 
    \begin{enumerate}[(i)]
        \setlength\itemsep{-0.3em}
        \item\label{item:dense-1} $G[S_1]$ and $G[S_2]$ are independent for any disjoint Lebesgue measurable sets $S_1,S_2$;
        \item\label{item:dense-2} for all boxes $S \in \calS$, $\pr(G[S] \text{ is non-empty and connected}\,) \geq 1-\eps$; 
        \item\label{item:dense-3} for all neighbouring boxes $S_1, S_2 \in \calS$,        
         \[
            \pr(G[S_1 \cup S_2] \text{ contains at least one edge from $S_1$ to $S_2$}) \geq 1-\eps;
        \]
        \item\label{item:dense-4} for all boxes $S \in \calS$, $\pr(G[S] \text{ has diameter at most $D$}\,) \geq 1-\eps$.
    \end{enumerate}
\end{definition}

Observe that $z\mapsto S_z$ is a bijection from $\Z^d$ to $\calS$; this forms the basis of the coupling in Lemma~\ref{lem:bond-percolation-coupling} below, essentially acting as renormalisation. In IGIRG and SFP, the random graph $G_M$ described at the start of the section will take the role of the dense geometric graph (see Corollary~\ref{cor:dense-subgraph}).
\begin{lemma}[Renormalisation-coupling to bond-percolation]\label{lem:bond-percolation-coupling}
    Let $d \in \N$, $\eps \in (0,1)$, and $K,R>0$. 
    Suppose $K,1/\eps \ggs R,d$ and let $G$ be an $(R,\eps)$-dense geometric graph with boxing scheme $\calS$, and let $\omega^\star$ be an iid Bernoulli bond percolation with retention probability $1 - 20d\eps$. Then there exists a coupling between $G$ and $\omega^\star$ such that whenever $z_1z_2$ is open in $\omega^\star$: $G[S_{z_1}]$ and $G[S_{z_2}]$ are non-empty and connected; $G[S_{z_1}]$ and $G[S_{z_2}]$ contain at most $K$ vertices each; and there is an edge from $\calV[S_{z_1}]$ to $\calV[S_{z_2}]$ in $G$. 
\end{lemma}
\begin{proof} In Definition \ref{def:dense-geometric}, the vertex set is either a homogeneous PPP or $\Z^d$.  
    If $\calV$ is a homogeneous PPP, there exists $K>0$ such that for all boxes $S \in \calS$,
    \begin{align}\label{eq:dense_geometric_K}
        \pr(|\calV[S]| \le K) \ge 1-\eps.
    \end{align}
    If $\calV=\Z^d$, then $|\calV[S]| \le (R+1)^d$ and so \eqref{eq:dense_geometric_K} holds trivially with $K = (R+1)^d$.
   
    We now carry out a one-step renormalisation and define a site-bond percolation $\omega$ on $\Z^d$. Recall that $\calS$ in Definition \ref{def:dense-geometric} is a boxing scheme with side-length $R$ and $S_z$ is the box containing $R\!\cdot\!z$ for $z\in \Z^d$. In the renormalised lattice, we set a \emph{site} $z\in\Z^d$ \emph{occupied} in $\omega$ if $G[S_z]$ contains at most $K$ vertices and is connected. We set two neighbouring sites $z,z'$ connected by an \emph{open bond} in $\omega$ if both sites are occupied and there is at least one edge in $G$ between $G[S_{z}], G[S_{z'}]$. By Definition~\ref{def:dense-geometric}\eqref{item:dense-1}, sites are occupied independently of each other in $\omega$ with probability at least $1-2\eps$ by \eqref{item:dense-1},\eqref{item:dense-2} and \eqref{eq:dense_geometric_K}, and bonds  that do not share a site are also open independently by \eqref{item:dense-1}. However, bonds that do share a site ($zz'$ and $zz''$) are \emph{not} open independently since they are both influenced by $G[S_z]$.
    
    We now couple $\omega$ to a Bernoulli bond percolation $\omega^\star$ on $\Z^d$. Similar ideas have been used before --- see \cite{andjel1993characteristic, liggett1997domination} in particular. By Definition~\ref{def:dense-geometric}\eqref{item:dense-2} and \eqref{item:dense-3} and by~\eqref{eq:dense_geometric_K}, every bond is open with probability at least $1-5\eps$, where the factor of $5$ comes from a union bound over the two sites being occupied and over \eqref{item:dense-3}. We next show an approximate independence as follows. Consider a bond $zz'$ in $\omega$, and let $N(zz') := \{\text{bonds incident to $z$ or $z'$}\} \setminus \{zz'\}$, with size $|N(zz')| = 4d-2$.
    For any set of bonds $S \subseteq N(zz')$,
    \begin{align*}
    \pr(zz' \text{ open in } \omega\mid \text{all bonds in $S$ open}) \nonumber & \geq \pr(zz' \text{ open} \text{ and } \text{all bonds in $S$ open}) \\ 
    & \geq 1-(4d-1)\cdot 5\eps \geq 1-20d\eps=:p
    \end{align*}
    by a union bound. 
    Using that bonds that do not share a site are independently open, one can iterate this argument to show that for every finite set $S$ of bonds, $\pr(\text{all bonds in $S$ are open}) \geq p^{|S|}$. Hence, by Strassen's theorem~\cite{lindvall1999strassen} we can couple $\omega$ with an independent bond percolation $\omega^{\star}$ on $\Z^d$ with retention-probability $p$, where every open bond in $\omega^{\star}$ is open in $\omega$.
 \end{proof}

\begin{definition}[Bernoulli-induced infinite subgraph]\label{def:dense-geometric-H}
    Let $\eps\in (0,1)$ be such that Lemma~\ref{lem:bond-percolation-density} applies with $\eps_{\ref{lem:bond-percolation-density}}:=20d \eps$ for some $r>0$ and some $\delta,\sigma \in(0,1)$. Let $G$ be an $(R,\eps)$-dense geometric graph with boxing scheme $\calS$ and let $\omega^\star$ be an iid Bernoulli bond percolation process with retention probability $p := 1-20d\eps$ given by the coupling of Lemma~\ref{lem:bond-percolation-coupling}. When $d\ge 2$, write $\calC_\infty^\star$ for the (unique) infinite component of $\omega^\star$ guaranteed by Lemma~\ref{lem:bond-percolation-density}. We define then the following infinite subgraph of $G$:
    \begin{equation}\label{eq:h-infty}
        \calH_\infty(G,\calS,\omega^\star) = \bigcup_{z \in \calC_\infty^\star} \calV[S_z].
    \end{equation}
    When there is no danger of confusion, we write $\calH_\infty := \calH_\infty(G,\calS,\omega^\star)$ and we may not explicitly define the bond percolation process $\omega^\star$.
\end{definition}
 The graph $\calH_\infty$ corresponds to vertices in renormalised boxes that belong to $\calC_\infty^\star$ of the renormalised bond-percolation model, hence we "blow-up" $\calC_\infty^\star$ to correspond to boxes in $G$. By the connectivity of $\calC_\infty$, also $G[\calH_\infty]$ is a.s.\ connected by the properties of the coupling set out in Lemma~\ref{lem:bond-percolation-coupling}. However, $\calH_\infty$ may only be a proper subset of some infinite component $\calC_\infty$ of $G$. Moreover, $G$ need not have a unique infinite component. Indeed, suppose $G$ is a random graph model with a PPP as $\calV$, where each vertex is coloured red or blue independently at random, and there is no edge with endpoints of different colours. If the probability of `red'  is sufficiently high with respect to the box side length $R$, so that with probability at least $1-\eps$ a box $S\in\calS$ contains no blue vertices, then Definition~\ref{def:dense-geometric}\eqref{item:dense-1}--\eqref{item:dense-2} may be satisfied, and there will be an infinite component of red vertices. However, there may also be an infinite component of blue vertices with arbitrary behaviour. Of course, if $\eps$ is small, then 
 $\calH_\infty$ has density larger than $1/2$ and so no other infinite component can have density larger than $1/2$, so $\calH_\infty$ is the uniquely determined as the infinite component with density close to $1$. 
 
 \subsection{Local denseness and linear distances in dense geometric random graphs}
 We now show in Lemma~\ref{lem:dense-geometric-locally-dense} that $\calH_\infty$ is ``locally dense'' in $G$, with density close to one, and in Lemma~\ref{lem:dense-geometric-linear-paths} that $\calH_\infty$ is a well-connected set containing short low-deviation paths between many pairs of vertices.
\begin{restatable}[Locally high-density Bernoulli-induced infinite subgraph]{lemma}{LocallyDense}
\label{lem:dense-geometric-locally-dense}
    Let $d \in \N$ with $d \ge 2$,  let $\eps,\delta,\sigma \in (0,1)$, and let $r,R>0$. 
    Suppose that $r\ggs \eps, \delta,R$, and that $\eps \lls \sigma,d$.
 Let $G$ be an $(R,\eps)$-dense geometric graph with boxing scheme $\calS$. Then a.s.\ $\calH_\infty$ is infinite and $G[\calH_\infty]$ defined in \eqref{eq:h-infty} is connected, for any box $S\in \calS$, $\pr(\calV\cap S\subseteq \calH_\infty)\ge 1-\sigma$ and  
 moreover
    \begin{align}
        \forall x\in\R^d:\quad \pr\left( \forall y\in B_r(x), \, \frac{|B_{(\log r)^2}(y) \cap \calH_\infty|}{|B_{(\log r)^2}(y)\cap\calV|} \ge 1-\sigma\right) \ge 1-\delta.  \label{eq:dense-geometric-density}
    \end{align}
\end{restatable}

The statement in \eqref{eq:dense-geometric-density} gets stronger if one decreases the radius $(\log r)^2$, since then the infinite subgraph $\calH_\infty$ is present already at lower radii near every vertex $y$ inside $B_r(x)$. For our proofs later, $(\log r)^2$ is more than strong enough, but one could improve  \eqref{eq:dense-geometric-density} to having $\Theta(\log r)$ as radius around $y$.
If the vertex set of $G$ is $\Z^d$, then Lemma~\ref{lem:dense-geometric-locally-dense} follows from Lemma~\ref{lem:bond-percolation-density} (applied with some $\sigma_{\ref{lem:bond-percolation-density}}$ that is sufficiently small compared to $\sigma$ in \eqref{eq:dense-geometric-density}) by the coupling in Lemma~\ref{lem:bond-percolation-coupling}. If the vertex set comes from a Poisson point process, we use that the number of vertices in the boxes $B_{(\log r)^2}(y) \cap \calH_\infty$ concentrates around its mean and that only at most $\sigma_{\ref{lem:bond-percolation-density}}$-fraction of boxes have too few/too many vertices in $B_{(\log r)^2}(y) \setminus \calH_\infty$. The details are straightforward and we give a proof of Lemma \ref{lem:dense-geometric-locally-dense} in the appendix on page \pageref{proof:lemma3.7}.

\begin{lemma}[Linear graph distances in dense geometric random graphs]\label{lem:dense-geometric-linear-paths}
    Let $d \in \N$ with $d \ge 2$, let $\zeta,\eps,\delta \in (0,1)$, and let $\kappa,r,R,C>0$. 
    Let $G$ be an $(R,\eps)$-dense geometric graph with boxing scheme $\calS$. For $x \in \calV$, let $\calA_{\mathrm{linear}}(r,\kappa,C,\zeta,x)$ be the event that for all $u \in B_r(x)\cap \calH_\infty$ and $v \in \calH_\infty$, there is a path from $u$ to $v$ in $G$ of length at most $\kappa|u-v|+C$ and deviation at most $\zeta|u-v|+C$. Then whenever $C\ggs r \ggs \eps,\delta,\zeta,R,d$ and $\kappa \ggs R,d$ and $\eps\lls d$, then
$\pr(\calA_{\mathrm{linear}}(r,\kappa,C,\zeta,x))\ge 1-\delta$.
\end{lemma}
\begin{proof}
    Let $\omega^\star$ be a bond percolation process with retention probability $1-20d\eps$ as required in Lemma \ref{lem:bond-percolation-coupling}, so that $\calH_\infty = \calH_\infty(G,\calS,\omega^\star)$ in \eqref{eq:h-infty} exists. Indeed, since $\eps\lls d$, by Lemma~\ref{lem:bond-percolation-density} 
    $\omega^\star$ has a unique infinite component $\calC_\infty^\star$ a.s.\ and $\calH_\infty$ is well-defined. Let $K>0$ be as in Lemma~\ref{lem:bond-percolation-coupling}, so that $K$ is a function of $d$ and $R$,
     and let $\kappa' = \sqrt \kappa$ and $C' = \sqrt{C}$.

    Denote the event $\calA^\star_{\mathrm{linear}}(r/R,\kappa',\zeta,z)$ of Lemma~\ref{lem:bond-percolation-linear} by $\calA^\star_1(z)$, so if $\calA_1^\star(z)$ occurs then there is a short low-deviation path from $z$ to any site in $\calC_\infty^\star \setminus B_{r/R}(z)$. Let $\calA_2^\star(z)$ be the event that for all $v \in \calC_\infty^\star\cap B_{r/R}(z)$, there is a path in $\omega^\star$ from $z$ to $v$ of length and deviation at most $C'$. Observe that if $\calA^\star_1(z) \cap \calA^\star_2(z)$ occurs, then for all $v \in \calC_\infty^\star$ there is a path in $\omega^\star$ from $z$ to $v$ of length at most $C' + \kappa'|z-v|$ and deviation at most $C' + \zeta|z-v|$.
    Recall that for $z\in \Z^d$, we denote by $S_z$ the box in $\calS$ that contains $R\!\cdot\!z$. We define our last event as
        \begin{align*}
        \calA^+(x) &:= \bigcap_{z\in\Z^d\,:\, S_z \cap B_r(x)\ne\emptyset} \Big(\big(\calA_1^\star(z) \cap \calA_2^\star(z)\big) \cup \{z \notin \calC_\infty^\star\}\Big).
    \end{align*}
    Using the coupling of Lemma \ref{lem:bond-percolation-coupling} we will prove that 
    \begin{align}\label{eq:dense-geometric-linear-main-event}
        \calA^+(x) &\subseteq \calA_{\mathrm{linear}}(r,\kappa,C,\zeta,x),\\\label{eq:dense-geometric-linear-A1-bound}
        \sum_{z \in \Z^d: S_z\cap B_r(x)\ne\emptyset}\pr\big(\neg\calA_1^\star(z) \cap \{z \in \calC_\infty^\star\}\big) &\le \delta/2,\\\label{eq:dense-geometric-linear-A2-bound}
        \pr\bigg(\bigcap_{z \in \Z^d: S_z\cap B_r(x)\ne\emptyset}\big(\calA_2^\star(z) \cup \{z\notin\calC_\infty^\star\} \big)\bigg) &\ge 1 - \delta/2.
    \end{align}
    Given~\eqref{eq:dense-geometric-linear-main-event}--\eqref{eq:dense-geometric-linear-A2-bound}, a union bound gives that $\pr(\calA_{\mathrm{linear}}(r,\kappa,C,\zeta,x)) \ge \pr(\calA^+)\ge  1-\delta$, as required.
We first prove~\eqref{eq:dense-geometric-linear-A1-bound}, starting with
    \[
        \sum_{z \in \Z^d\,:\,S_z\cap B_r(x)\ne\emptyset}\pr\big(\neg\calA_1^\star(z) \cap \{z \in \calC_\infty^\star\}\big) \le \sum_{z \in \Z^d\,:\,S_z\cap B_r(x)\ne\emptyset}\pr\big(\neg\calA_1^\star(z) \mid z \in \calC_\infty^\star\big).
    \]
    By Lemma~\ref{lem:bond-percolation-linear}, since $\kappa'\ggs d \ggs \eps$ and $r\ggs \eps,\zeta,R,d$, there exists $c_{\ref{lem:bond-percolation-linear}}$ depending only on $d$ such that each term in the sum is at most $e^{-c_{\ref{lem:bond-percolation-linear}}r/R}$. Since $S_z$ is a box of side-length $R$, the number of terms $|\{z\in \Z^d: S_z\cap B_r(x)\}|$ is at most $(3r/R)^d$, and $r\ggs \delta, R, d$ (and thus $r\ggs c_{\ref{lem:bond-percolation-linear}}$),
    \[
        \sum_{z \in \Z^d\,:\,S_z\cap B_r(x)\ne\emptyset}\pr\big(\neg\calA_1^\star(z) \cap \{z \in \calC_\infty^\star\}\big) \le (3r/R)^de^{-c_{\ref{lem:bond-percolation-linear}}r/R} \le \delta/2
    \]
for all sufficiently large $r$ given $\delta$.
    We next prove~\eqref{eq:dense-geometric-linear-A2-bound}. The maximum length and deviation of a path between any pair of sites in $\calC_\infty^\star \cap B_{3r/R}(x)$ is a random variable which is a.s.\ finite, so since $C'=\sqrt{C}\ggs \delta,r$ this maximum must be at most $C'$ with probability at least $1-\delta/2$, for all sufficiently large $C'$. Thus~\eqref{eq:dense-geometric-linear-A2-bound} follows.

    Finally, we prove~\eqref{eq:dense-geometric-linear-main-event}. Suppose $\calA^+(x)$ occurs, let $u \in B_r(x) \cap \calH_\infty$, and let $v \in \calH_\infty$ two vertices in $G$. Let $u^- \in \Z^d$ be such that $u \in S_{u^-}$, and let $v^- \in \Z^d$ be such that $v \in S_{v^-}$; thus $u^- = \lfloor u/R \rfloor$ and $v^- = \lfloor v/R \rfloor$, and $u^-,v^-\in\calC^\star_\infty$ per definition of $\calH_\infty$ in \eqref{eq:h-infty}. Since $\calA^+(x)$ occurs and $u^-\in\calC^\star_\infty$, the event $\calA_1^\star(u^-) \cap \calA_2^\star(u^-)$ also occurs; thus there exists a path $\pi^\star = z_1\dots z_m$ (for some $m$) from $u^-$ to $v^-$ in $\omega^\star$ with
    \begin{align*}
        |\pi^\star| \le C'+\kappa'|u^- - v^-|,\qquad \mbox{dev}(\pi^\star) \le C'+\zeta|u^- - v^-|.
    \end{align*}
    The statement of the lemma becomes weaker with larger $\kappa$, so we may just prove it for one fixed value $\kappa=\kappa(R,d)$, and deduce it for all larger values of $\kappa$. We may thus assume for this fixed value of $\kappa$ that $C'\ggs \kappa$ and $C' \ggs \kappa'=\sqrt{\kappa}$. For the same reason we may assume $C'\ggs \zeta$. Since also $C'\ggs R,d$ and the diameter of each box is at most $R\sqrt{d}$, we have
    \begin{align}\label{eq:dense-geometric-pistar}
        |\pi^\star| \le 2C'+ \kappa'|u - v|/R,\qquad \mbox{dev}(\pi^\star) \le 2C'+\zeta|u - v|/R.
    \end{align}
    By the properties of the coupling with $\omega^\star$ set out in Lemma~\ref{lem:bond-percolation-coupling}, for each $z_i\in \pi^\star$, the graph $G[B_{z_i}]$ is connected and contains between $1$ and $K$ vertices, and for all $i \le m-1$ there is an edge between $G[B_{z_i}]$ and $G[B_{z_{i+1}}]$. Since $u \in B_{z_1}$ and $v \in B_{z_m}$, we can thus find a path $\pi^G$ from $u$ to $v$ in $G[B_{z_1} \cup \dots \cup B_{z_m}]$. Since each box contains at most $K$ vertices, it follows from~\eqref{eq:dense-geometric-pistar} that 
    \[
        |\pi^G| \le Km = K(|\pi^\star|+1) \le 2K\kappa'|u-v|/R + K(2C'+1).
    \]
    Likewise, since each box has side length $R$ and diameter at most $R\sqrt{d}$,
    \[
        \mathrm{dev}(\pi^G) \le R\sqrt{d} + R\cdot\mathrm{dev}(\pi^\star) \le \zeta |u-v| + 2RC' + R\sqrt{d}.
    \]
    Recall that $K$ is a function of $d$ and $R$, and that $\kappa,C\ggs d,R$; we therefore recover $|\pi^G| \le \kappa|u-v|+C$ and $\mathrm{dev}(\pi^G) \le \zeta |u-v|+C $. Thus $\calA_{\mathrm{linear}}(r,\kappa,C,\zeta,x)$ occurs, and so $ \calA^+(x)\subseteq	\calA_{\mathrm{linear}}(r,\kappa,C,\zeta,x)$ as claimed.
\end{proof}
We will state the following corollary in a stronger form than necessary here;  part \eqref{item:cor-3} and the `strong' version of part \eqref{item:cor-1} below is only used in our companion paper~\cite{komjathy2022one1} but not in this paper. Recall from \eqref{eq:parameters} that $\mpar$ is the set of model parameters. In order to be able to talk about a number of vertices being in the vertex set, we shall work conditional on one of two events; these conditionings might be trivial for SFP. In particular, we write $\calF := \{x_1, \dots, x_t \in \calV\}$ to mean that the locations $x_1, \dots, x_t \in \R^d$ are part of the vertex set with unknown vertex weights from distribution $F_W$, and $\calF_M := \{x_1, \dots, x_t \in \calV_M\}$ for the event that the locations $x_1, \dots, x_t$ are part of the vertex set with unknown vertex weights in the interval  $[M,2M]$ for some large constant $M$. In other words, we (sometimes) work under certain Palm-measures of the underlying Poisson process.
\begin{corollary}[Linear costs in Bernoulli-induced infinite subgraph in IGIRG and SFP]\label{cor:dense-subgraph}
    Consider $1$-FPP in Definition \ref{def:1-FPP} on a graph $G=(\calV,\calE)$ that is an IGIRG or SFP satisfying assumptions \eqref{eq:power_law}--\eqref{eq:F_L-condition} with $\tau\in(2,3)$. Let $\delta,\eps,\sigma \in (0,1)$, let $r,M,C,\kappa,\zeta,t > 0$, and let $R := M^{2/d}/\sqrt{d}$. Suppose $C\ggs r \ggs M,\zeta,\delta$, and that $\kappa \ggs M \ggs \eps,\sigma,t, \mpar$. Finally, suppose that $\eps \lls\sigma,\mpar$. Let $I_M:=[M,2M]$, and let $G_M=(\calV_M,\calE_M)$ be the subgraph of $G$ formed by vertices $\calV_M:=\{v\in\calV: W_v\in I_M\}$ and edges $\calE_M:=\{uv\in\calE : u,v\in \calV_M, \cost{uv}\le M^{3\mu}\}$. For $t$ as above, let $\{x_1, \dots, x_t\} \subseteq \Z^d$ for SFP and $\{x_1, \dots, x_t \}\subseteq \R^d$ for IGIRG, and let $\calF := \{x_1, \dots, x_t \in \calV\}$ and $\calF_M := \{x_1, \dots, x_t\in\calV_M\}$. 
    Then:
\begin{enumerate}[(i)]
    \setlength\itemsep{-0.3em}
    \item\label{item:cor-1} For all dimensions $d\ge 1$, conditioned on either $\calF$ or $\calF_M$, the graph $G_M$ is a strong $(R,2, e^{-M^{3-\tau-\eps}})$-dense and an $(R,\eps)$-dense geometric graph in the sense of Definition \ref{def:dense-geometric}.
    \item\label{item:cor-2} For all dimensions $d\ge 2$, a.s.\ $\calH_\infty$ is infinite, $G[\calH_\infty]$ is connected, $G$ has a unique infinite component $\calC_\infty$, and $\calH_\infty \subseteq \calV(\calC_\infty) \cap \calV_M$.
    \item\label{item:cor-3} For all dimensions $d\ge 2$, and $x \in \R^d$,
    \begin{equation}\label{eq:dense-C_infinity^M}
        \pr\Big(\forall y\in B_r(x), \, \frac{|B_{(\log r)^2}(y) \cap \calH_\infty|}{|B_{(\log r)^2}(y)\cap\calV_M|} \ge 1-\sigma\,\Big|\,\calF\Big) \ge 1-\delta.
    \end{equation}
    \item\label{item:cor-4} For $x \in \R^d$, let $\calA_\mathrm{linearcost}(r,\kappa,C,\zeta,x)$ be the event that for all $u \in \calB_r(x) \cap \calH_\infty$ and all $v \in \calH_\infty$, there is a path from $u$ to $v$ of cost at most $\kappa |u-v| + C$ and deviation at most $\zeta |u-v| + C$. Then for all dimensions $d\ge 2$, 
    $\pr(\calA_{\mathrm{linearcost}}(r,\kappa,C,\zeta,z)\mid \calF) \ge 1-\delta$.
\end{enumerate}
\end{corollary}
\begin{proof}
    Let $\calS$ be a boxing scheme of side length $R=M^{2/d}/\sqrt{d}$.
       Using Definition \ref{def:dense-geometric}, we now prove that $G_M$ is a strong $(R,2,\eps_{\ref{def:dense-geometric}})$-dense geometric graph for $\eps_{\ref{def:dense-geometric}}:= e^{-M^{3-\tau-\eps}}$ with boxing scheme $\calS$. Note that this automatically implies that  $(R,\eps)$-dense since $M \ggs \eps$. 
    Definition~\ref{def:dense-geometric}(i) follows from the definitions of the models. For the other conditions we first lower-bound the expected number of vertices in a box $S\in \calS$. In both IGIRG and SFP, we use~\eqref{eq:power_law} to bound:
    	\begin{equation}\label{eq:W-in-IM}
\begin{aligned}    	\pr(W\in I_M) &= (1-F_W(M)) - (1- F_W(2M))
    	= \frac{\ell(M)}{M^{\tau-1}} - \frac{\ell(2M)}{(2M)^{\tau-1}} \\
    	&= M^{-(\tau-1)} \ell(M) \left(1- 2^{-(\tau-1)} \frac{\ell(2M)}{\ell(M)}\right)\ge M^{-(\tau-1)-\eps/4},
    	\end{aligned}
	\end{equation}
    where to obtain the last inequality we used that $\ell$ is slowly-varying, so $\ell(2M)/\ell(M)\to 1$, and that $M\ggs \eps$. In IGIRG, $\calV_M$ follows a homogeneous PPP with intensity $\pr(W\in I_M)$ on $\R^d$; since $S$ has side length $R$, $|\calV_M[S]|$ is therefore a Poisson random variable with mean $\lambda_M(\R^d):=R^d\pr(W\in I_M)$, which is at least $R^dM^{-(\tau-1)-\eps/4} = d^{-d/2}M^{3-\tau-\eps/4}$. Since $\tau<3$, the exponent is positive since we assumed $\eps\lls \mpar$ small, and so $\lambda_M \ge M^{3-\tau-\eps/2}$ since $M \ggs d, \eps$. Similarly, in SFP the number of vertices in $\calV_M[S]$ follows a binomial distribution with mean $\lambda_M(\Z^d):=|\Z^d\cap S|\cdot \pr(W\in I_M)\ge d^{-d/2}M^{3-\tau-\eps/4}/2$, (where the factor of $2$ suffices to account for boundary effects since $M$ is large), 
which is again at least $M^{3-\tau-\eps/2}$ for $M\ggs\eps$.    

    We now study the edge probabilities within the box $S$. Consider two vertices $u,v \in \calV_M[S]$.  Their distance is most the diameter of the box, $\sqrt{d}R\le M^{2/d}$, and $W_u,W_v\in I_M=[M, 2M]$ so $W_uW_v/|u-v|^d\ge 1$  holds. Recall that the edges have cost $(W_uW_v)^\mu L_{uv}\le 4M^{2\mu}L_{uv}$ in $1$-FPP, while we keep the edge in $G_M$ only if this edge cost is at most $M^{3\mu}$. Thus for all vertices $u,v\in \calV_M[S]$,
    \begin{align}
    \begin{split}\label{eq:M-mesh-probability}
       \mathbb P\big(uv \in \calE, \cost{uv}\le M^{3\mu}\mid u,v\in \calV_M[S] \big)  
       &\ge \underline{c}\left(1 \wedge \frac{W_uW_v}{|u-v|^d}\right)^{\alpha} \cdot F_L\big((W_uW_v)^{-\mu}M^{3\mu}\big) \\
       &\ge \underline{c}F_L(4^{-\mu}M^{\mu}) \ge \underline{c}/2,
    \end{split}
    \end{align}
    where the last inequality holds since $M \ggs \mpar$. Note that~\eqref{eq:M-mesh-probability} holds \emph{uniformly} over the weights in $I_M$ and locations of vertices in $S$, and is also valid when $\alpha=\infty$ or $\beta=\infty$. With the vertex set $\calV_M$ exposed, the presence of edges in $G_M[S]$ can therefore stochastically dominates an independent collection of Bernoulli$(\underline{c}/2)$ random variables. 
    
   Thus, the graph $G_M[S]$ dominates an Erd\H{o}s-R{\'e}nyi random graph with number of vertices distributed as $\mbox{Poisson}(\lambda_M(\R^d))$ or binomial with mean $\lambda_M(\Z^d)$ with $\lambda_M(\R^d), \lambda_M(\Z^d)$ both at least  $M^{3-\tau-\eps/2}$, and \emph{constant} connection probability $\underline{c}/2$. This Erd\H{o}s-R\'{e}nyi random graph is non-empty and connected with diameter two with probability at least $1-\exp(-\Theta(\lambda_M))\ge 1-\exp(-M^{3-\tau-\eps})=1-\eps_{\ref{def:dense-geometric}}$, see~\cite[Theorem 7.1]{frieze2023random}. Hence, $G_M$ satisfies conditions \eqref{item:dense-2}--\eqref{item:dense-4}
    of Definition~\ref{def:dense-geometric} with $\eps_{\ref{def:dense-geometric}}$ and is 
    thus a strong $(R,2,\eps_{\ref{def:dense-geometric}})$-dense geometric graph, which proves that \eqref{item:cor-1} of Corollary~\ref{cor:dense-subgraph} holds unconditionally. Moreover, since $R\ggs t$, the above argument still holds conditioned on any intersection $\calF$ of at most $t$ events of the form $z \in \calV$ or $z\in\calV_M$; from this point on in the proof, we always condition on $\calF$, and the only property of $G_M$ we use is that it is an $(R,\eps)$-dense geometric graph conditioned on $\calF$.

 When $d\ge 2$,  let $\omega^\star$ be an iid Bernoulli bond percolation  with retention probability $1-20d\eps$, and recall $\calH_\infty := \calH_\infty(G_M, \calS,\omega^\star)$  from \eqref{eq:h-infty} in Definition~\ref{def:dense-geometric-H}. Since $\eps \lls \sigma,d$ and $r\ggs \eps, \delta, R$, Corollary~\ref{cor:dense-subgraph}\eqref{item:cor-2}--\eqref{item:cor-3} are now immediate from Lemma~\ref{lem:dense-geometric-locally-dense}, except the uniqueness of $\calC_\infty$, which was proved earlier in \cite{deijfen2013scale} for SFP and in \cite{deprez2019scale} for IGIRG.  Recall the definition of $\calA_{\mathrm{linear}}(r,\kappa, C, \zeta,x)$ from Lemma~\ref{lem:dense-geometric-linear-paths}: in particular, that the \emph{graph} distance in $G_M$ between two vertices is $\kappa$ times the Euclidean distance. Since every edge of $G_M$ has cost at most $M^{3\mu}$, the \emph{cost} distance then is at most $\kappa M^{3\mu}$ between those vertices in $G_M$. Hence, $\calA_{\mathrm{linearcost}}(r,\kappa,C,\zeta,z) \supseteq \calA_{\mathrm{linear}}(r,\kappa/M^{3\mu},C,\zeta,z)$; thus Corollary~\ref{cor:dense-subgraph}\eqref{item:cor-4} follows from Lemma~\ref{lem:dense-geometric-linear-paths}.
\end{proof}
The above proof works whenever $3-\tau>0$, which guarantees that the expected degree of vertices in $G_M$, i.e.,  $M^{3-\tau +o(1)}$, grows with $M$. We chose the box-size $R=M^{2/d}$ so that it corresponds to the connectivity radius of a vertex in $\calV_M$. For $\tau>3$, the expected degree in $G_M$ tends to $0$ as $M$ increases and so $G_M$ has no infinite component for large $M$ (even if we removed the restriction on edge costs).
\subsection{Connecting to the infinite subgraph of the auxiliary graph}
Corollary~\ref{cor:dense-subgraph}\eqref{item:cor-4} guarantees linear-cost low-deviation paths within $\calH_\infty \subseteq \calC_\infty$ when $d\ge 2$. However, when connecting two arbitrary vertices, e.g. $0$ and $x$, in IGIRG/SFP, we cannot assume that they fall in $\calH_\infty$. We solve this issue using the following two claims.

\begin{claim}[The infinite cluster of IGIRG/SFP is dense]\label{claim:two-in-Cinfty}
    Let $G=(\calV,\calE)$ be an IGIRG or SFP satisfying assumptions \eqref{eq:power_law}--\eqref{eq:F_L-condition} with $d\ge 2$ and $\tau\in(2,3)$. Then there exists $\rho > 0$ such that for all $u,v \in \R^d$ (for IGIRG) or $u,v\in\Z^d$ (for SFP), we have $\pr(u,v \in \calC_\infty \mid u,v \in \calV) \ge \rho$.
\end{claim}
\begin{claim}[Connecting to the Bernoulli-induced infinite subgraph]\label{claim:linear-start-vertices}
    Consider the setting of Corollary~\ref{cor:dense-subgraph} with $d\ge2$, recall $\calH_\infty \subseteq \calV_M$ from \eqref{eq:h-infty}, and let $\calC_\infty$ be the infinite component of $G$ containing $\calH_\infty$. Let $u,v \in \R^d$ (for IGIRG) or $\Z^d$ (for SFP). Then, whenever $r\ggs M, \delta$, 
    \begin{equation}\label{eq:path-contained}
    \begin{aligned}
    \pr\Big( \exists u^\star\in \calH_\infty\!\cap\! B_r(u);  &\ \exists \mbox{ a path }\pi^G_{u,u^\star}\subseteq \calE(G):\\
     &\calV(\pi^G_{u,u^\star}) \subseteq B_r(u),    \cost{\pi^G_{u,u^\star}} \le C \mid u,v \in \calC_\infty\Big)\ge 1-\delta.
     \end{aligned}
    \end{equation}
    \end{claim}

\begin{proof}[Proof of Claim \ref{claim:two-in-Cinfty}]
    Fix two locations $u,v\in \R^d$ or $u,v \in \Z^d$, and let $M\ggs \mpar$, and recall from the calculation in \eqref{eq:W-in-IM}  that
       \[
        \pr(u,v \in \calV_M \mid u,v\in\calV) =\pr(W\in I_M)^2 \ge M^{-2(\tau-1)-\eps/2}\ge  M^{-2\tau}.
    \]
    Further, by Corollary 3.9\eqref{item:cor-1}, $G_M$ is an $(R, \eps)$-dense geometric graph even when conditioned on the event $\calF_M=\{u,v\in \calV_M\}$. This means that the coupling in Lemma \ref{lem:bond-percolation-coupling} remains valid under the conditioning, and the high-density result $\pr(S\cap\calV \subseteq \calH_\infty\mid u,v\in \calV_M) \ge 1-\sigma$ in Lemma \ref{lem:dense-geometric-locally-dense} holds also conditioned on $u,v\in \calV_M$. So, a union bound  yields that 
    \[\begin{aligned}
     \pr(u, v\in \calH_\infty\mid u,v \in \calV_M) \ge 1-\pr(u\notin \calH_\infty \mid u,v \in \calV_M) - \pr(v \notin \calH_\infty \mid u,v \in \calV_M)\\
     \ge 1-2\sigma\ge 1/2,
     \end{aligned}\]
     whenever $\sigma\le 1/4$.
   Hence 
    \[
     \pr(u,v \in \calH_\infty \mid u,v\in \calV) = \pr(u,v \in \calH_\infty \mid u,v\in\calV_M) \cdot  \pr(u,v \in \calV_M \mid u,v\in \calV)\ge 1/(2M^{2\tau}).
    \]
    Since $\calH_\infty \subseteq V(\calC_\infty)$ by Corollary~\ref{cor:dense-subgraph}\eqref{item:cor-2}, the result follows with $\rho = 1/(2M^{2\tau})$.
\end{proof}
\begin{proof}[Proof of Claim \ref{claim:linear-start-vertices}]
    For $r' >0$ let $\calB^\star(r',M,u):= \calH_\infty\cap B_{r'}(u) $. 
    By Corollary~\ref{cor:dense-subgraph}\eqref{item:cor-2}--\eqref{item:cor-3}, applied with $\delta/2$ instead of $\delta$, $\calH_\infty$ has a positive density in $\calV_M$ around $u$. Hence,     for $r'\ggs M,\delta$, 
    \begin{align}\label{eq:start-neighbourhood-nonempty}
        \pr(\calB^\star(r',M,u) = \emptyset \mid u,v \in \calC_\infty) \le \delta/2.
    \end{align}
    Fix such $r'$.
   Because $\calH_\infty\subseteq \calC_\infty$ connected, paths exists within $\calC_\infty$ to vertices in $\calH_\infty$. In particular, 
 conditioned on $u,v\in\calC_\infty$, fix any procedure to uniquely select a path $\pi_{uu^\star}^G$ from $u$ to each $u^\star\in \calB^\star(r',M,u)$ (for example, take the paths with lowest deviation, breaking ties by random coin-flips). Then the random variables
    \[
    \begin{aligned}
    X &:= \inf\{r>0 \mid \forall u^\star \in \calB^\star(r',M,u): V(\pi_{uu^\star}^G) \subseteq B_r(u)\},\\ 
    Y &:= \inf\{C >0 \mid \forall u^\star \in \calB^\star(r',M,u): \cost{\pi_{uu^\star}^G} \le C\},
    \end{aligned}
    \]
    are a.s.\ finite conditioned on $u,v\in \calC_\infty$.
    On choosing $r$ and $C$ suitably large, we have 
    \begin{align}\label{eq:start-neighbourhood-maxcost}
    \pr(X > r \text{ or } Y > C \mid u,v\in\calC_\infty) \le \delta/2,
    \end{align}
    and the result follows from a union bound over~\eqref{eq:start-neighbourhood-nonempty} and~\eqref{eq:start-neighbourhood-maxcost}.
 \end{proof}

The next corollary combines the previous claims and lemmas and constructs a linear-cost low deviation path between two vertices in the unique infinite component of SFP/IGIRG. 
\begin{corollary}[Linear costs in 1-FPP on IGIRG/SFP]\label{lem:linear-regime-join-two}
    Consider $1$-FPP on IGIRG or SFP of Definition \ref{def:girg} satisfying the assumptions given in \eqref{eq:power_law}--\eqref{eq:F_L-condition} with $0\in\calV$. Assume that the dimension $d\ge 2$, and $\alpha>2, \tau\in(2,3)$, $\mu>\mu_{\mathrm{pol}}$. Let $\kappa,r > 0$, and let $\delta,\zeta \in (0,1)$. Suppose $r\ggs \delta, \zeta, \mpar$, and that $\kappa\ggs \mpar$. Let $\calC_\infty$ be the unique infinite component guaranteed by Corollary \ref{cor:dense-subgraph}\eqref{item:cor-2}. Let $u,v \in \R^d$ for IGIRG or $\Z^d$ for SFP with $|u-v| \ge r$. Let $\calA_{\mathrm{linearcost}}(u,v, \kappa, \zeta)$ be the event that there is a path joining $u$ and $v$ in $G$ with cost at most $\kappa|u-v|$ and deviation at most $\zeta|u-v|$. Then
    \[
        \pr(\calA_{\mathrm{linearcost}}(u,v,\kappa,\zeta)\mid u,v\in\calC_\infty) \ge 1-\delta.
    \]
\end{corollary}
\begin{proof}
    Let $\rho\in (0,1)$ be as in Claim~\ref{claim:two-in-Cinfty}. Let $r_{\ref{cor:dense-subgraph}}$,  $C_{\ref{cor:dense-subgraph}}>0$ be as in Corollary~\ref{cor:dense-subgraph}\eqref{item:cor-4} applied with $\kappa_{\ref{cor:dense-subgraph}} = \kappa/2$, $\delta_{\ref{cor:dense-subgraph}} = \rho\delta/3$, $\zeta_{\ref{cor:dense-subgraph}} = \zeta/2$ and $t_{\ref{cor:dense-subgraph}} = 2$ (and any suitable values of $\eps$, $M$, $\sigma$). Let $r_{\ref{claim:linear-start-vertices}}$ and $C_{\ref{claim:linear-start-vertices}}$ be as in Claim~\ref{claim:linear-start-vertices} applied with $\delta_{\ref{claim:linear-start-vertices}} = \delta/3$.
In Corollary \ref{cor:dense-subgraph}\eqref{item:cor-4}, the requirement $r_{\ref{cor:dense-subgraph}}\ggs M, \zeta, \delta$ is assumed and  in Claim \ref{claim:linear-start-vertices} $r_{\ref{claim:linear-start-vertices}}\ggs M$,  thus we may increase the first value to assume $r_{\ref{cor:dense-subgraph}}\ge r_{\ref{claim:linear-start-vertices}}$.
    We may assume that $r\ggs r_{\ref{cor:dense-subgraph}}, r_{\ref{claim:linear-start-vertices}}, C_{\ref{cor:dense-subgraph}}, C_{\ref{claim:linear-start-vertices}}$.
    We define events as follows.
    \begin{enumerate}[(a)]
        \setlength\itemsep{-0.3em}
        \item Let $\calA_1$ be the event that there is a path $\pi^1$ from $u$ to some vertex $u^\star \in \calH_\infty \cap B_{r_{\ref{claim:linear-start-vertices}}}(u)$ of cost at most $C_{\ref{claim:linear-start-vertices}}$ and with $\calV(\pi^1) \subseteq B_{r_{\ref{claim:linear-start-vertices}}}(u)$.
        \item Let $\calA_2$ be the event that \emph{every} vertex $u_1 \in \calH_\infty\cap B_{r_{\ref{cor:dense-subgraph}}}(u) $ is joined to every vertex $u_2' \in \calH_\infty$ by a path $\pi^2_{u_1,u_2}$ of cost at most $(\kappa/2)|u_1-u_2| +   C_{\ref{cor:dense-subgraph}}$ and deviation at most $(\zeta/2)|u_1-u_2|+C_{\ref{cor:dense-subgraph}}$.
        \item Let $\calA_3$ be the event that there is a path $\pi^3$ from some vertex $v^\star$ in $\calH_\infty \cap B_{r_{\ref{claim:linear-start-vertices}}}(v)$ to $v$ of cost at most $C_{\ref{claim:linear-start-vertices}}$ and with $\calV(\pi^3) \subseteq B_{r_{\ref{claim:linear-start-vertices}}}(v)$.
    \end{enumerate}
    Observe that if $\calA_1\cap\calA_2\cap\calA_3$ occurs, then since $|u-v|\ge r$ is large,    the path $\pi^1\pi^2_{u^\star,v^\star}\pi^3$ has cost at most $\kappa|u-v|/2 + C_{\ref{cor:dense-subgraph}} + 2C_{\ref{claim:linear-start-vertices}} \le \kappa |u-v|$ and deviation at most $\zeta|u-v|/2 + C_{\ref{cor:dense-subgraph}} + 2r_{\ref{claim:linear-start-vertices}} \le \zeta |u-v|$. We must therefore prove
    \begin{equation}\label{eq:linear-regime-proof-goal}
        \pr(\calA_1\cap\calA_2\cap\calA_3\mid u,v\in \calC_\infty) \ge 1-\delta.
    \end{equation}
    By our choice of $r_{\ref{claim:linear-start-vertices}}$ and $C_{\ref{claim:linear-start-vertices}}$, Claim~\ref{claim:linear-start-vertices} implies that
    \[
        \pr(\neg\calA_1 \mid u,v \in \calC_\infty) \le \delta/3,\qquad \pr(\neg\calA_3 \mid u,v \in \calC_\infty) \le \delta/3.
    \]
    Similarly, by Corollary~\ref{cor:dense-subgraph}\eqref{item:cor-4}  we have $\pr(\neg\calA_2 \mid u,v\in\calV) \le \rho\delta/3$. Thus by Lemma~\ref{claim:two-in-Cinfty},
    \[
        \pr(\neg\calA_2 \mid u,v \in \calC_\infty) \le \frac{\pr(\neg\calA_{2} \mid u,v\in\calV)}{\pr(u,v\in\calC_\infty\mid u,v\in \calV)} \le \frac{\rho\delta/3}{\rho} = \delta/3.
    \]
    Applying a union bound, \eqref{eq:linear-regime-proof-goal} follows as required.
\end{proof}

\begin{proof}[Proof of Theorems~\ref{thm:linear_regime} and \ref{thm:threshold_regimes}]
We have already proven the lower bounds: namely the lower bound of Theorem~\ref{thm:linear_regime} follows directly from~\eqref{eq:linear-lower} in Theorem~\ref{thm:linear_polynomial_lower_bound}, with proof on page \pageref{proof:linear_polynomial_lower_bound}, and the proof of the lower bound in Theorem~\ref{thm:threshold_regimes} can be found on page~\pageref{proof:threshhold_regimes_lower}. The upper bounds of both theorems follow from Corollary~\ref{lem:linear-regime-join-two}. 
\end{proof}

It remains to prove Theorem~\ref{thm:finite_graph}, which states that all results also hold in the finite GIRG model if $u_n,v_n$ are chosen uniformly at random from the largest component $\calC_{\max}^{(n)}$.

\begin{proof}[Proof of Theorem~\ref{thm:finite_graph}]
    Let $\kappa_1 > 0$ be small enough to take the role of $\kappa$ in \eqref{eq:linear-lower} in Theorem~\ref{thm:linear_polynomial_lower_bound}, and let $\kappa_2>0$ be large enough to take the role of $\kappa$ in Corollary~\ref{lem:linear-regime-join-two}.
    Let $G_n = (\calV_n,\calE_n)$ be a GIRG satisfying the assumptions of the theorem statement, let $\calC_{\max}^{(n)}$ be the largest component of $G_n$, and let $u_n$ and $v_n$ be vertices chosen independently and uniformly at random from $\calC_{\max}^{(n)}$. For all $u,v\in \R^d$ and all graphs $H$, let $\calA(H,u,v)$ be the event that $u$ and $v$ have cost-distance between $\kappa_1 |u-v|$ and $\kappa_2 |u-v|$ in the (sub)graph $H$. Then it suffices to prove that for all $\delta \in (0,1)$, whenever $n\ggs \delta$,  $\pr(\neg\calA(G_n,u_n,v_n)) \le \delta$ holds.
    
    Let $x_n,y_n \in \calV_n$ be chosen independently and uniformly at random from $\calV_n$. Whenever $n\ggs \delta$;  it is known~\cite[Theorem~3.11]{komjathy2020stopping} that $\calC_{\max}^{(n)}$ has constant density whp, so with probability at least $1/2$, we have $|\calC_{\max}^{(n)}| \ge \delta^{1/4} |\calV_n|$. Thus $\pr(x_n,y_n \in \calC_{\max}^{(n)}) \ge \sqrt{\delta}/2$, and so 
    \begin{align*}
        \pr(\neg\calA(G_n,u_n,v_n)) 
        &= \pr(\neg\calA(G_n,x_n,y_n) \mid x_n,y_n \in \calC_{\max}^{(n)}) \\
        &\le 2\pr(\neg\calA(G_n,x_n,y_n) \mbox{ and }x_n,y_n\in\calC_{\max}^{(n)})/\sqrt{\delta}.
    \end{align*}
    Recall that for all $n>0$, $Q_n := [-n^{1/d}/2, n^{1/d}/2]^d$. Let $x_n',y_n' \in Q_n$ be random points chosen independently and uniformly at random from the Lebesgue measure in $Q_n$. Let $\calV_n'$ be a Poisson point process of unit intensity conditioned on $x_n',y_n'\in\calV_n'$, i.e.\ a Palm process. It is known that the total variation distance of $(\calV', x_n', y_n')$ from $(\calV, x_n, y_n)$ converges to zero as $n\to\infty$, and in particular is at most 
 $\delta^{3/2}/12$  when $n$ is sufficiently large. Thus on taking $G_n'$ to be a GIRG with vertex set $\calV_n'$, we have
    \begin{align*}
        \pr(\neg\calA(G_n,u_n,v_n)) \le (2/\sqrt{\delta})(\delta^{3/2}/12) + 2\pr(\neg\calA(G_n',x_n',y_n') \mbox{ and }x_n',y_n'\in\calC_{\max}^{(n)})/\sqrt{\delta}.
    \end{align*}
    We may couple $G_n'$ to an IGIRG $G^+$ in such a way that $G_n = G^+[Q_n]$. Let $\calC_\infty$ be the infinite component of $G^+$. Further, the giant component of $G_n'$ is part of $\calC_\infty$ with probability tending to $1$ as $n$ tends to infinity, see \cite{komjathy2020explosion}. So for $n$ large enough, $\pr(\calC_{\max}^{(n)} \nsubseteq \calC_\infty) \le \delta^{3/2}/12$, and hence
    \[
        \pr(\neg\calA(G_n,u_n,v_n)) \le \delta/3 + 2\pr(\neg\calA(G_n',x_n',y_n') \mbox{ and }x_n',y_n'\in\calC_\infty)/\sqrt{\delta}.
    \]
    Let $r$ be large enough for Corollary~\ref{lem:linear-regime-join-two} and Theorem~\ref{thm:linear_polynomial_lower_bound} to apply to $\delta^{3/2}/12$, taking $\zeta = \delta^2$ in Corollary~\ref{lem:linear-regime-join-two}. Let $X$ be the set of pairs $(x,y) \in Q_{(1-d\delta^2)n}$ such that $|x-y| \ge r$, and let $\calX_n$ be the event that $(x_n',y_n') \in X$. Observe that $\pr(\neg\calX_n) \le \delta^{3/2}/12$ if $\delta\lls d$ and $n\ggs r$.
    Thus
    \begin{align*}
        \pr(\neg\calA(G_n,u_n,v_n)) &\le \delta/2 + 2\pr(\neg\calA(G_n',x_n',y_n') \mbox{ and }x_n',y_n'\in\calC_\infty\mid \calX_n\mbox{ and }x_n',y_n' \in \calV_n)/\sqrt\delta\\
        &\le \delta/2 + 2\max_{x,y\in X}\pr(\neg\calA(G_n',x,y) \mid x_n=x,y_n=y,x,y\in\calC_\infty)/\sqrt\delta.
    \end{align*}
    By Theorem~\ref{thm:linear_polynomial_lower_bound}, the lower cost bound in the event $\calA(G_n',x,y)$ fails with probability at most $\delta^{3/2}/8$. Moreover, when $x,y \in Q_{(1-d\delta^2)n}$, the upper cost bound in $\calA(G_n',x,y)$ occurs whenever $x$ and $y$ are joined by a path in $\calC_\infty$ with cost at most $\kappa |x-y|$ and deviation at most $\delta^2 |x-y| \le \delta^2 nd$; thus by Corollary~\ref{lem:linear-regime-join-two}, the upper bound fails with probability at most $\delta^{3/2}/8$. By a union bound, it follows that $\pr(\neg\calA(G_n,u_n,v_n) \le \delta$ as required.
\end{proof}

%% file: appendix.tex
\appendix
\section{Appendix}

Here we prove Lemmas~\ref{lem:bond-percolation-density},~\ref{lem:bond-percolation-linear} and~\ref{lem:dense-geometric-locally-dense}. For Lemma~\ref{lem:bond-percolation-density} we use the following theorem, which is a simplified version of~\cite[Theorem~1.1]{DP-percolation}.

\begin{theorem}\label{thm:DP-percolation}
    For all $d \in \N$ with $d \ge 2$ and all $\sigma \in (0,1/2)$ there exists $p_0 = p_0(\sigma,d) \in (0,1)$ such that the following holds. Let $\omega^\star$ be an iid Bernoulli bond percolation on $\mathbb{Z}^d$ with edge-retention probability $p \ge p_0$. For $S \subseteq \mathbb{R}^d$, let 
    \begin{equation}\label{eq:A-dense-perc}
    \calA_\mathrm{dense}(S,\sigma):=\{ \mbox{$S$ contains a unique cluster $\calC$ with $|\calC| \ge (1-\sigma)\vol(S)$}\}.
    \end{equation} 
    Let $S_n = [-n/2, n/2]^d$. Then
    \begin{align*}
        \limsup_{n\to\infty} \frac{1}{n^{d-1}}\log\pr(\neg\calA_\mathrm{dense}(S_n,\sigma)) < 0.
    \end{align*}
\end{theorem}

The corresponding~\cite[Theorem~1.1]{DP-percolation} is stated for site percolation. Since the event $\calA_\mathrm{dense}(S,\delta)$ is monotone; the same result for bond percolation also holds via the standard domination in which we retain a vertex in a coupled site percolation process if and only if we retain all its edges in bond percolation. We now use Theorem~\ref{thm:DP-percolation} to prove Lemma~\ref{lem:bond-percolation-density}, showing that the infinite cluster $\omega$ of a high-density Bernoulli bond-percolation has locally high density, see also \eqref{eq:bond-percolation-density}.
\begin{proof}[Proof of Lemma \ref{lem:bond-percolation-density}]\label{proof:lemma3.2}
The statement that $\pr(0\in \calC_\infty^\star)\ge 1-\sigma$ as the edge-retention probability $1-\eps$ tends to $1$ follows directly from the continuity of the the theta-function in the supercritical regime (in particular, near $1$), see \cite{grimmett1999percolation}. We move on to showing \eqref{eq:bond-percolation-density}. By translation invariance, it is enough to show the statement for $x=0\in \Z^d$. Recall $\calA_\mathrm{dense}(S_n, \sigma)$ from \eqref{eq:A-dense-perc} and that $S_n = [-n/2,n/2]^d$. By Theorem~\ref{thm:DP-percolation}, taking $p_0:=1-\eps$ and $\eta\lls \sigma$, since the dimension is $d\ge2$,
    \begin{equation} \label{eq:percolation-density-0}
        \forall n \ge n_0\colon \pr(\calA_\mathrm{dense}(S_n,\sigma/2)) \ge 1 - \exp(-\eta n^{d-1}) \ge 1 - \exp(-\eta n).
    \end{equation}
    By the Borel-Cantelli lemma, almost surely $\calA_{\mathrm{dense}}(S_n,\sigma/2)$ occurs for all but finitely many values of $n$. Moreover, if $\calA_{\mathrm{dense}}(S_n,\sigma/2)$ and $\calA_{\mathrm{dense}}(S_{n+1},\sigma/2)$ both occur for suitably large $n$, their respective clusters must intersect by the pigeonhole principle and hence must be equal. Thus almost surely there exists a single cluster $\calC_\infty^\star$ such that for all but finitely many $n$, $|\calC_\infty^\star \cap S_n| \ge (1-\sigma/2)\vol(S_n)$. Such a cluster is necessarily both infinite and unique, as required. Hence, let  
    \begin{equation}\label{eq:a-infty}
    \calA_\infty(S,\sigma/2):=\{ |S \cap \calC_\infty^\star| \ge (1-\sigma/2)\vol(S)\}.
    \end{equation}
     Then using \eqref{eq:percolation-density-0}, and the above intersection property when one moves from $S_i$ to $S_{i+1}$, there exists $n_1 \ge n_0$ such that 
    \begin{equation}\label{eq:percolation-density-1}
        \forall n \ge n_1\colon\pr(\calA_\infty(S_n,\sigma/2)) \ge 1 - \sum_{i\ge n}\pr(\neg\calA_{\mathrm{dense}}(S_i,\sigma/2)) \ge 1 - \exp(-\eta n/2).
    \end{equation}
We now develop a boxing scheme. By definition, $S_{4r}=[-2r, 2r]^d$ which fully contains $B_r(0)$, and further, for each $y\in B_r(0)$, also 
$B_{(\log r)^{3/2}}(y)$ is fully contained in $S_{4r}$.
    For all $r>0$, let $\ell(r) := (\log r)^{4/3}$, so that $r = \exp(\ell(r)^{3/4})$. Let $\calS_{4r}$ be a partition of $S_{4r}$ into at most $(4r/\ell(r))^d$ boxes of (equal) side length $\ell \in [\ell(r),2\ell(r)]$. We may assume that $r\ggs \delta, \sigma,\eta$ such that $\ell(r) \ge n_1$ and furthermore:
    \begin{enumerate}[(i)]
        \item\label{item:box-1} $(4r/\ell(r))^d\cdot \exp(-\eta \ell/2) = 4^d\exp(d\ell(r)^{3/4}-\eta \ell/2)/\ell(r)^{d} \le \delta$ for all $\ell\in[\ell(r), 2\ell(r)]$ and
        \item\label{item:box-2} for all $y \in B_r(x)$, there exist disjoint boxes $S_1^{\scriptscriptstyle{(y)}},\dots,S_t^{\scriptscriptstyle{(y)}} \in \calS_{4r}$ such that $S_i^{\scriptscriptstyle{(y)}} \subseteq B_{(\log r)^{3/2}}(y)$ for all $i\in[t]$ and such that $\vol(S_1^{\scriptscriptstyle{(y)}} \cup \dots \cup S_t^{\scriptscriptstyle{(y)}}) \ge (1-\sigma/2)|B_{(\log r)^{3/2}}(y)\cap\Z^d|$.
    \end{enumerate}
    Here, \eqref{item:box-1} implies that the probability that there is a box $S$ in the boxing scheme $\calS_{4r}$ for which $\calA_{\infty}(S, \sigma/2)$ does not holds is at most $\delta$, by combining \eqref{eq:percolation-density-1} with a union bound over the at most $(4r/\ell(r))^d$ boxes.
 So, introducing $\calA_\mathrm{good}(4r,\sigma/2)$ as the event that for all $S \in \calS_{4r}$, the event $\calA_\infty(S,\sigma/2)$ in \eqref{eq:a-infty} occurs, then
    \begin{equation}\label{eq:bond-percolation-density-prob}
        \pr(\calA_\mathrm{good}(4r,\sigma/2)) \ge 1 - (4r/\ell(r))^d\exp(-\eta \ell/2) \ge 1 - \delta.
    \end{equation}
    Suppose that $\calA_\mathrm{good}(4r,\sigma/2)$ occurs, fix some $y\in B_r(0)$ and let $S_1^{\scriptscriptstyle{(y)}},\dots,S_t^{\scriptscriptstyle{(y)}}$ be the boxes in \eqref{item:box-2} contained in $B_{(\log r)^{3/2}}(y)$. Then by the uniqueness of the cluster $\calC_\infty^\star$,
    \begin{align*}
        \frac{|B_{(\log r)^{3/2}}(y) \cap \calC_\infty^\star|}{|B_{(\log r)^{3/2}}(y) \cap \Z^d|} 
        &\ge \frac{1}{|B_{(\log r)^{3/2}}(y)\cap \Z^d|}\sum_{i \in [t]}|S_i^{\scriptscriptstyle{(y)}} \cap \calC_\infty^\star|\\
        &\ge \frac{1}{|B_{(\log r)^{3/2}}(y)\cap \Z^d|}\sum_{i \in [t]}(1-\tfrac{\sigma}{2})\vol(S_i^{\scriptscriptstyle{(y)}})
        \ge (1-\tfrac{\sigma}{2})^2 > 1-\sigma,
    \end{align*}
    where in the last step we used that the union of the boxes $S_1^{\scriptscriptstyle{(y)}}, \dots, S_t^{\scriptscriptstyle{(y)}}$ covers at least $1-\sigma$ proportion of the vertices in $|B_{(\log r)^{3/2}}(y)\cap \Z^d|$ by \eqref{item:box-2}.
    Hence,~\eqref{eq:bond-percolation-density} follows from~\eqref{eq:bond-percolation-density-prob}.
\end{proof}

We next prove Lemma~\ref{lem:bond-percolation-linear}, which states that the infinite component of highly supercritical Bernoulli bond percolation has linear graph distances realised by paths with sublinear deviation.
\begin{proof}[Proof of Lemma \ref{lem:bond-percolation-linear}]\label{proof:lemma3.3}
    We first recall a result on linear scaling of distances in $\calC^\star_\infty$. Let $\{x\leftrightarrow y\}$ be the event that $x,y\in \Z^d$ lie in the same component of $\omega^\star$, and let $d_\star(x,y)$ denote their graph distance in $\omega^\star$. By~\cite[Theorem~1.1]{antal1996chemical}, there exists $\kappa',\eps_0>0$ such that for all $\eps \le \eps_0$ and for all $a \in \Z^d$,\footnote{The formulation of \cite[Theorem~1.1]{antal1996chemical} allows $\kappa'$ to depend on the edge-retention probability $p=1-\eps$ because it allows $p$ to be arbitrarily close to the criticality threshold. So it is not quite clear from the formulation that $\kappa'$ is independent of $\eps$. However, the condition $\eps\le\eps_0$ means that our $p$ is bounded away from the criticality threshold, and the proof in \cite{antal1996chemical} relies on a domination argument, so we may use the same $\kappa'$ for all $\eps \le \eps_0$.}
    \begin{equation}\label{eq:dense-geometric-perc-0}
        \limsup_{|y|\to\infty} \frac{1}{|x-y|} \log\pr\big(\{x\leftrightarrow y\} \cap \{d_\star(x,y) > \kappa'|x-y|\}\big) < 0.
    \end{equation}
    We say that a site $x\in\Z^d$ has \emph{$r$-linear scaling} if $x \in \calC^\star_\infty$ and
    \begin{equation}\label{eq:r-linear-scaling}
    d_\star(x,y) \le \kappa' |x-y| \quad \mbox{ holds for all  } y \in \calC_\infty^\star \setminus B_r(x). 
    \end{equation}
    By a union bound over all $y\in \Z^d\setminus B_r(x)$, it follows from~\eqref{eq:dense-geometric-perc-0} that there exist $c_1,r_1>0$ (depending on $d$) such that for all $x \in \Z^d$ and $r\ge r_1$,
    \begin{equation}\label{eq:dense-geometric-perc}
        \pr\big(x \mbox{ has $r$-linear scaling or }x \notin \calC^\star_\infty\big) \ge 1 - e^{-c_1r}.
    \end{equation}
    We now use~\eqref{eq:dense-geometric-perc} to show that there exist constants $c_2,r_2>0$ (depending on $\zeta$ and $d$) such that if $x,y \in \calC_\infty^\star$ and $|x-y|\ge r_2$, then with probability at least $1-e^{-c_2|x-y|}$ there exists a linear-length low-deviation path from $x$ to $y$. 
    Let $\kappa := 2\sqrt{d} \kappa'$ and $K := (\kappa'+1)\sqrt{d}/\zeta$, and let $x,y \in \calC_\infty^\star$. Cover the straight line segment $S_{x,y}$ by a sequence of $k \le 2K$ cubes $Q^{\scriptscriptstyle{(1)}},\ldots,Q^{\scriptscriptstyle{(k)}}$ in such a way that (see also Figure~\ref{fig:small-deviation}):
    \begin{enumerate}[(a)]
    \setlength\itemsep{-0.2em}
        \item each cube $Q^{\scriptscriptstyle{(i)}}$ has side length $|x-y|/K$;
        \item $Q^{\scriptscriptstyle{(1)}}$ is the cube covering $x$, $Q^{\scriptscriptstyle{(k)}}$ is the cube covering $y$, and $Q^{\scriptscriptstyle{(i)}} \cap Q^{\scriptscriptstyle{(i+1)}}\cap S_{x,y} \ne \emptyset$ for all $i \in [k-1]$;
        \item for all $i \in [k-1]$, the intersection $Q^{\scriptscriptstyle{(i)}}  \cap Q^{\scriptscriptstyle{(i+1)}}$ contains a cube of side-length $|x-y|/(8K)$;
        \item all pair of sites $a,b$ that fall into distinct sets in the list $\{x\}, Q^{\scriptscriptstyle{(1)}}  \cap Q^{\scriptscriptstyle{(2)}} , Q^{\scriptscriptstyle{(2)}} \cap Q^{\scriptscriptstyle{(3)}} ,\ldots, Q^{\scriptscriptstyle{(k-1)}}  \cap Q^{\scriptscriptstyle{(k)}} , \{y\}$ satisfy $|a-b| \ge |x-y|/(2K)$.
    \end{enumerate}

\begin{figure}[t]
    \centering
    \includegraphics[trim=3.0cm 0.0cm 0.0cm 0.0cm,clip,width=0.95\textwidth]{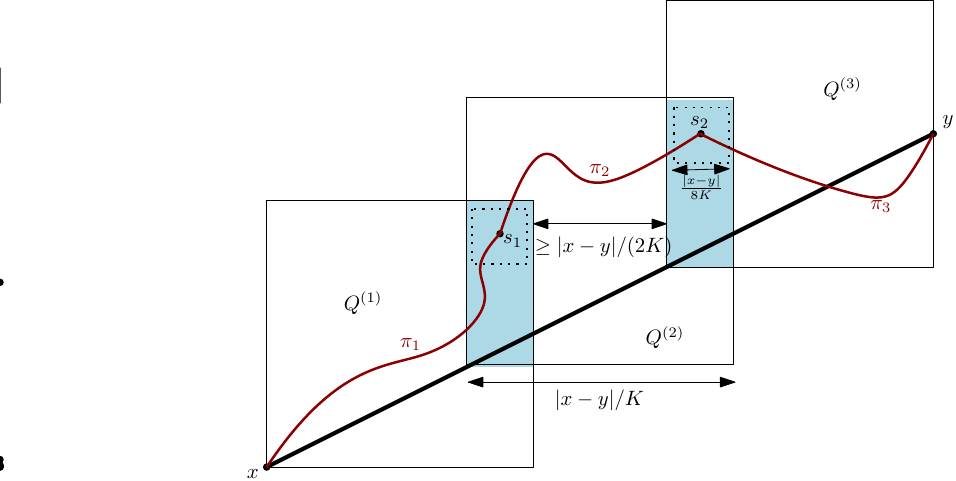}
    \caption[]{\small 
    Example for $k=3$ cubes covering the line from $x$ to $y$. Any two adjacent cubes overlap by at least $4^{-d}\mathrm{Vol}(Q^{\scriptscriptstyle{(i)}})$ and contain a cube of side-length $|x-y|/(8K)$ and any two intersections have distance at least $r/2 = |x-y|/(2K)$, where $r = |x-y|/K$ is the side-length of the cubes. 
    }
    \label{fig:small-deviation}
\end{figure}
    Then the following events occur with probability at least $1-e^{-c_2|x-y|}$ if $|x-y|\ge r_2$ (we specify $c_2$ and $r_2$ below):
    \begin{enumerate}[(i)]
    \setlength\itemsep{-0.2em}
        \item\label{item:event-1} For all $i \in [k-1]$, $Q^{\scriptscriptstyle{(i)}}  \cap Q^{\scriptscriptstyle{(i+1)}} $ contains at least one site $s_i \in \calC^\star_\infty$;
        \item\label{item:event-2} Both $x$ and $y$ have $r$-linear scaling with $r=|x-y|/(2K)$ conditioned on $x,y\in \calC_\infty^\star$;
        \item\label{item:event-3} For all $i \in [k-1]$, all sites $z \in Q^{\scriptscriptstyle{(i)}}  \cap Q^{\scriptscriptstyle{(i+1)}}  \cap \calC^\star_\infty$ have $r$-linear scaling with $r=|x-y|/(2K)$.
    \end{enumerate}
    Indeed, since $\eps\lls d$, $\calC^\star_\infty$ has density at least $1/2$ and since the intersection $Q^{\scriptscriptstyle{(i)}} \cap Q^{\scriptscriptstyle{(i)}} $ contains a cube $S$ of side-length $|x-y|/(8K)$, \eqref{eq:percolation-density-1} applies, i.e., the event in \eqref{eq:a-infty} that $S\cap \calC_\infty$ contains linearly vertices in the volume of $S$ holds with probability $1-\exp(-\eta|x-y|/(8K))$. This event implies that there is at least one vertex in $\calC_\infty^\star\cap S$ and with $c_3:=\eta/8K$, there exists $r_3>0$ depending on $\zeta,d$ such that \eqref{item:event-1} holds with probability at least $1-e^{-c_3|x-y|}$ whenever $|x-y|\ge r_3$. Moreover, using \eqref{eq:dense-geometric-perc} and a union bound over $i\in[k-1]$ and over all sites $z\in Q^{\scriptscriptstyle{(i)}}  \cap Q^{\scriptscriptstyle{(i+1)}}$ (polynomially many in $|x-y|$) and over $x,y$, we deduce that there exists $r_4>0$ depending on $d$ such that \eqref{item:event-2} and \eqref{item:event-3} also hold with probability $1-e^{-c_1|x-y|/(4K)}$ whenever $|x-y|\ge r_4$. Taking $r_2 = \max\{r_3,r_4\}$ and $c_2$ small enough so that $e^{-c_3 r}+e^{-c_1r/(4K)} \le e^{-c_2r}$ for all $r\ge r_2$ guarantees that the above events occur with probability at least $1-e^{-c_2|x-y|}$ whenever $|x-y|\ge r_2$.
    
    Suppose the events in \eqref{item:event-1}--\eqref{item:event-2} all occur. Choose now for all $i\in[k-1]$ sites $s_i\in \calC_\infty^\star \cap Q^{\scriptscriptstyle{(i)}}  \cap Q^{\scriptscriptstyle{(i+1)}}$. Writing $s_0=x, s_k=y$, it holds for all $i\in[k]$ that
     \begin{equation}\label{eq:distance-sisi}
        |s_{i-1}-s_{i}|>|x-y|/(2K) \quad \mbox{and} \quad |s_{i-1}-s_{i}|\le \sqrt{d}|x-y|/K.
    \end{equation}
    Then the $r$-linear-scaling property of $s_i$ and $s_{i-1}$ with $r=|x-y|/(2K)$ in \eqref{eq:r-linear-scaling}, combined with \eqref{eq:distance-sisi}, implies deterministically that for all $i \in [k]$ there exists a path $\pi_i$ from $s_{i-1}$ to $s_i$ in $\calC_\infty^\star$ of length at most $\kappa'\sqrt{d} |x-y|/K$. Since $s_{i-1}$ has distance at most $\sqrt{d}|x-y|/K$ from the segment $S_{x,y}$ and $\pi_i$ contains nearest neighbour edges,  by the definition of $K=(\kappa'+1)\sqrt{d}/\zeta$ at the beginning of the proof and Definition \ref{def:deviation},
    \begin{equation}
    \mathrm{dev}_{xy}(\pi_i) \le \kappa' \sqrt{d} |x-y|/K+\sqrt{d}|x-y|/K = \zeta |x-y|.
    \end{equation} Let $\pi = \pi_1\ldots \pi_k$; then since $k\le 2K$, and $\kappa = 2\sqrt{d} \kappa'$,  $\pi$ has length at most $k\kappa'\sqrt{d} |x-y|/K \le \kappa|x-y|$ and deviation at most $\max(\mathrm{dev}_{xy}(\pi_i))\le \zeta|x-y|$ (see Definition \ref{def:deviation}). Therefore, we have shown that if $x,y\in\calC_\infty^\star$ and $|x-y|\ge r_2$, then with probability at least $1-e^{-c_2|x-y|}$ there exists a path $\pi$ from $x$ to $y$ with length at most $|\pi|\le\kappa |x-y|$ and deviation at most $\mathrm{dev}(\pi)\le \zeta |x-y|$. Taking a union bound over all $y\in\calC_\infty^\star \setminus B_r(x)$ and taking $|x-y|\ge r \ggs \zeta$ yields~\eqref{eq:bond-percolation-distances} with $c=c_2/2$.
\end{proof}

Lastly, we prove Lemma \ref{lem:dense-geometric-locally-dense}, which uses Lemma~\ref{lem:bond-percolation-density} to obtain that the infinite subgraph $\calH_\infty$ of a $(R, \eps)$-dense random geometric graph also has overall high density.

\begin{proof}[Proof of Lemma  \ref{lem:dense-geometric-locally-dense}]\label{proof:lemma3.7}
    Let $\omega^\star$ be a bond percolation process with retention probability $1-20d\eps$.
    We take $\eps$ small enough that Lemma~\ref{lem:bond-percolation-density} applies with $\sigma_{\ref{lem:bond-percolation-density}} = \sigma^3$ and $\delta_{\ref{lem:bond-percolation-density}} = \delta/3$; thus $\omega^\star$ has a unique infinite component $\calC_\infty^\star$ and $\calH_\infty$ is well-defined 
    in \eqref{eq:h-infty}. Then, since $\calH_\infty$ consists of those boxes $S_z$ for which $z\in \Z^d$ belong to $\calC_\infty^\star$ in $\omega^\star$, if $z\in \calC_\infty^\star$, $z$ has at least one open adjacent bond. By the defining coupling in Lemma \ref{lem:bond-percolation-coupling}, such a bond is only open in $\omega^\star$ if $G[S_z]$ is non-empty and connected. This implies that if $z\in \calC_\infty$ then all the vertices in $G[S_z]$ are in $\calH_\infty$. Since 
    $\pr(0\in \calC_\infty^\star)\ge 1-\sigma_{\ref{lem:bond-percolation-density}}$ in Lemma \ref{lem:bond-percolation-density}, the translation invariance of $\Z^d$ readily implies that for any box $S$, $\pr(\calV\cap S\subseteq \calH_\infty) \ge 1-\sigma_{\ref{lem:bond-percolation-density}}\ge 1-\sigma$. 
     
     We turn to prove \eqref{eq:dense-geometric-density}. Throughout, let $\rho = (\log r)^2$ for brevity.
    We set out some preliminary notation and observations. Recall that for all $z\in \Z^d$, $S_z$ is the unique box in $\calS$ such that $Rz \in S_z$ and that boxes have side-length $R$. For all sets $A \subseteq \R^d$ we write $A^\star = \{z \in \Z^d\colon S_z \subseteq A\}$; for the renormalisation of $A$, thus $\mathrm{Boxes}(A):=\bigcup_{z \in A^\star} S_z$ is the union of all boxes fully contained in $A$. Finally, for all $y \in \R^d$, by $\lfloor y/R\rfloor$ we mean taking lower-integer part of each coordinate, here $\lfloor y/R\rfloor$ is the renormalised site, i.e., in $\Z^d$, corresponding to the (box containing) $z\in \R^d$. Clearly then $y\in S_{\lfloor y/R\rfloor}$. Since $r\ggs R$, and the diameter of each box is $\sqrt{d}R$, it is not hard to see that for all $y \in \R^d$, and with $\rho=(\log r)^2$,
    \begin{align*}
  B_{\rho-R\sqrt{d}}(y)& \subseteq  \mathrm{Boxes}(B_{\rho}(y)) \subseteq B_{\rho}(y),\\ 
    B_{\rho/R - 2\sqrt{d}}(\lfloor y/R\rfloor)\cap\Z^d  &\subseteq B_{\rho}(y)^\star \subseteq B_{{\rho}/R + \sqrt{d}}(\lfloor y/R\rfloor),\\
  B_{\rho-R\sqrt{d}}(y)^\star&\subseteq  B_{\rho}(R\lfloor y/R\rfloor)^\star \cap B_{\rho}(y)^\star.
    \end{align*}
Thus there exists $C>0$ depending on $R$ and $d$ (but not on $r$ and thus neither on $\rho$) such that for all $y \in \R^d$,
    \begin{align}\label{eq:dense-geometric-boxes-1}
       \big|\mathrm{Vol}(B_{\rho}(y)) - R^d|B_{\rho}(y)^\star|\big| &\le C{\rho}^{d-1} \le \sigma^4 \rho^d,\\ \label{eq:dense-geometric-boxes-2}
       \big||B_{\rho}(y)^\star| - |B_{{\rho}/R}(\lfloor y/R\rfloor)\cap \Z^d| \big| &\le C{\rho}^{d-1} \le \sigma^4\rho^d,\\
       \label{eq:dense-geometric-boxes-3}
       |B_{\rho}(R\lfloor y/R\rfloor)^\star \setminus B_{\rho}(y)^\star| &\le C \rho^{d-1} \le \sigma^4 \rho^d.
    \end{align}
These all intuitively recover the isoperimetric properties of $\R^d$: for sufficiently large $\rho$, the volume of a ball of radius $\rho$ can be well approximated by a boxing scheme of side-length $R$ \eqref{eq:dense-geometric-boxes-1}, and also the number of boxes we need both have error of order $\rho^{(d-1)/d}$ \eqref{eq:dense-geometric-boxes-2}, and switching from a renormalised set around $y$ to the renormalised set around the renormalised site corresponding to $y$ also causes small error \eqref{eq:dense-geometric-boxes-3}.
  
  Now we relate \eqref{eq:bond-percolation-density} to \eqref{eq:dense-geometric-density}. Note that \eqref{eq:bond-percolation-density} has radius $(\log r)^{3/2}$ inside the probability sign concerning overall high density in $\omega^\star$, while \eqref{eq:dense-geometric-density} has radius $(\log r)^2$ concerning overall high density of $\calH_\infty$ in $G$. Hence, given $r, R$, let $r^\star$ be such that $(\log r^\star)^{3/2} = (\log r)^2/R$, i.e.\ $r^\star := \exp(((\log r)^2/R)^{2/3})$, and clearly for all sufficiently large $r\ggs R$,  $r^\star > 2r/R$. We will now apply Lemma \ref{lem:bond-percolation-density} with $r=r^\star$. Then, the radius inside the fraction in \eqref{eq:bond-percolation-density} is exactly $(\log r^\star)^{3/2}=(\log r)^2/R$. So, for $x\in\R^d$, when we write the event in \eqref{eq:bond-percolation-density} for the site $\lfloor x/R\rfloor\in \Z^d$ and using radius $r^\star$ and $\sigma_{\ref{lem:bond-percolation-density}}=\sigma^3$, we obtain
      \begin{equation}\label{eq:dense-geometric-adense-star}
        \calA_\mathrm{dense}^\star(x, r^\star) = \Big\{\forall y \in B_{r^\star}(\lfloor x/R\rfloor)\colon \frac{|B_{(\log r)^2/R}(y) \cap \calC_\infty^\star|}{|B_{(\log r)^2/R}(y) \cap \Z^d|} \ge 1-\sigma^3\Big\}.
    \end{equation}
    Now applying Lemma~\ref{lem:bond-percolation-density} yields directly that $\pr(\mathcal{A}_{\mathrm{dense}}^\star) \ge 1-\delta_{\ref{lem:bond-percolation-density}} = 1-\delta/3$. When we consider any $y \in B_{r}(x)$ in $\R^d$, clearly $\lfloor y/R\rfloor \in B_{2r/R}(x) \subseteq B_{r^\star}(\lfloor x/R\rfloor)$ because $r^\star\ge 2r/R$; thus if~\eqref{eq:dense-geometric-adense-star} holds, then also
    \[
        \forall y \in B_{r}(x)\colon \frac{|B_{(\log r)^2/R}(\lfloor y/R\rfloor ) \cap \calC_\infty^\star|}{|B_{(\log r)^2/R}(\lfloor y/R\rfloor) \cap \Z^d|} \ge 1-\sigma^3.
    \]
    Note that $B_{(\log r)^2/R}(\lfloor y/R\rfloor )$, when we move from sites in $\Z^d$ to boxes in $\calS$, corresponds roughly to the set of boxes contained in $B_{(\log r)^2}(y)$. In fact \eqref{eq:dense-geometric-boxes-2} exactly quantifies the error being small. 
  Hence, since $r$ is large,  it follows from~\eqref{eq:dense-geometric-boxes-2} that $\calA_\mathrm{dense}^\star(x,r^\star)$ also implies the following event:
    \begin{equation}\label{eq:dense-geometric-adense}
        \calA_{\mathrm{dense}} (x,r):= \Big\{\forall y \in B_{r}(x)\colon \frac{|B_{(\log r)^2}(y)^\star \cap \calC_\infty^\star|}{|B_{(\log r)^2}(y)^\star|} \ge 1-2\sigma^3\Big\}.
    \end{equation}
    We have shown that
    \begin{equation}\label{eq:dense-geometric-adense-bound}
        \pr(\calA_{\mathrm{dense}}(x,r)) \ge \pr(\calA_{\mathrm{dense}}^\star(x,r^\star)) \ge  1-\delta/3.
        \end{equation}
    Interpreting the event on the lhs, $\calA_{\mathrm{dense}}$ says that sites in $\calC_\infty^\star$ form high density in $B_{(\log r)^2}(y)^\star$, or equivalently that boxes containing vertices of $\calH_\infty$ are dense in $B_{(\log r)^2}(y)$; meanwhile, the event in~\eqref{eq:dense-geometric-density} says that the vertices inside these boxes are dense in $B_{(\log r)^2}(y)\cap \calV$. Thus we can achieve the high density in \eqref{eq:dense-geometric-density} if we can control the number of vertices per box.
    
    We now prove a concentration bound for the number of vertices of $G$ in a given collection of boxes; since $\calC_\infty^\star$ is random, we basically will sum the errors over all realisations of $\calC_\infty^\star$ satisfying $\calA_{\mathrm{dense}}(x,r)$.     
We again abbreviate $\rho:=(\log r)^2$. Here below, we denote by $A$ any set in $B_{\rho}(y)^\star$ with number of vertices $|A| \le 2\sigma^3 |B_{\rho}(y)^\star|$ so that $A$ can essentially serve as a possible realisation of the \emph{complement} of $\calC_\infty^\star$ inside the ball $B_{\rho}(y)^\star$ when the event $\calA_{\mathrm{dense}}(x,r)$ holds. Then, define $\calA_{\mathrm{box}}$ as
    \begin{equation}\label{eq:dense-geometric-local-0a}
      \calA_{\mathrm{box}}:= \bigcap_{y \in B_r(x)} \bigcap_{\substack{A\subseteq B_{\rho}(y)^\star\\|A| \le 2\sigma^3 |B_{\rho}(y)^\star|}}\bigg\{  \Big|\bigcup_{z \in B_{\rho}(y)^\star \setminus A} \calV[S_z]\Big| \ge (1-\sigma/5)\mathrm{Vol}(B_{\rho}(y))\Big\}\bigg\},
    \end{equation}
    i.e., that leaving out the boxes in any not-too-large set $A$ from $B_{\rho}(y)^\star$, the remaining boxes still contain enough vertices proportional to the volume. When $\calV = \Z^d$, $\calA_{\mathrm{box}}$ always occurs. For $\calV$ a Poisson point process, we dominate $\calA_{\mathrm{box}}$ by an intersection of events as follows. For all $y \in \R^d$ and $A \subseteq \Z^d$, let 
    \begin{equation}\label{eq:dense-geometric-local-0b}
        \calA_{\mathrm{box}}(y,A) := \bigg\{ \Big|\bigcup_{z \in B_{\rho}(y)^\star \setminus A} \calV[S_z]\Big| \ge (1-\sigma/10)R^d|B_{\rho}(y)^\star|\bigg\}.
    \end{equation}
    Applying~\eqref{eq:dense-geometric-boxes-1} yields that $(1-\sigma/10)R^d|B_{\rho}(y)^\star| \ge (1-\sigma/5)\vol(B_{\rho}(y))$. Moreover, by~\eqref{eq:dense-geometric-boxes-3}, we can lower-bound $B_{\rho}(y)^\star$ in~\eqref{eq:dense-geometric-local-0a} by $B_{\rho}(R\lfloor y/r\rfloor)^\star \setminus A_y$ for some set $A_y$ with $|A_y| \le \sigma^3 |B_{\rho}(y)^\star|$, effectively discretising the continuous intersection over $y \in B_r(x)$. Thus
    \begin{align}
        \calA_{\mathrm{box}} 
        &\supseteq \bigcap_{y \in B_r(x)} \bigcap_{\substack{A\subseteq B_{\rho}(y)^\star\\|A| \le 2\sigma^3 |B_{\rho}(y)^\star|}} \calA_{\mathrm{box}}(y,A)
        \supseteq \bigcap_{y \in B_r(x)} \bigcap_{\substack{A\subseteq B_{\rho}(R\lfloor y/R\rfloor)^\star\\|A| \le 4\sigma^3 |B_{\rho}(R\lfloor y/R\rfloor)^\star|}} \calA_{\mathrm{box}}(R\lfloor y/R\rfloor,A)\nonumber\\ \label{eq:dense-geometric-local-2}
        &\supseteq \bigcap_{z \in B_{2r}(x)^\star} \bigcap_{\substack{A\subseteq B_{\rho}(Rz)^\star\\|A| \le 4\sigma^3 |B_{\rho}(Rz)^\star|}} \calA_{\mathrm{box}}(Rz,A),
    \end{align}
We obtained the last row by noting that the rhs of the first row is the same event for all $y$ with the same renormalised site $\lfloor y/R\rfloor$. 
       When $\calV$ is a PPP, the number of vertices in boxes in $B_\rho(y)^\star\setminus A$ in~\eqref{eq:dense-geometric-local-0b} follows a Poisson distribution with mean $R^d|B_{\rho}(y)^\star \setminus A|$; thus when $|A| \le 4\sigma^3 |B_{\rho}(y)^\star|$, this mean is at least $(1-\sigma/20)R^d|B_{\rho}(y)^\star|$, and so by a Chernoff bound we arrive to
    \[
        \pr(\neg \calA_{\mathrm{box}}(y,A)) \le \exp(-\sigma^2\rho^d/300).
    \]
    Now some combinatorics to deal with the union when taking the complement event in \eqref{eq:dense-geometric-local-2}: since $\sigma$ is small, the number of sets $A \subseteq B_{\rho}(Rz)^\star$ with $|A| \le 4\sigma^3|B_{\rho}(Rz)^\star|$ is at most $\exp(\sigma^2\rho^d/600)$. Moreover, there are at most $4(r/R)^d$ choices of $z \in B_{2r}(x)^\star$ in \eqref{eq:dense-geometric-local-2}.
    Hence using a union bound, $\rho = (\log r)^2$, and $r\ggs \delta, \sigma$, for Poisson $\calV$,
    \begin{align}\label{eq:dense-geometric-avert-bound}
        \pr(\calA_{\mathrm{box}}) \ge 1 - 4(r/R)^d\exp(\sigma^2\rho^d/600) \cdot \exp(-\sigma^2\rho^d/300) \ge 1 - \delta/3.
    \end{align}
 Our final event, $\calA_{\mathrm{ball}}$, says that all radius-$\rho=(\log r)^2$ balls near $x$ contain at most roughly the expected number of vertices, so that 
    \begin{equation}\label{eq:a-ball}
    \begin{aligned}
        \calA_{\mathrm{ball}} &= \bigcap_{y \in B_r(x)} \Big\{|B_{\rho}(y) \cap \calV| \le (1+\sigma/5)\vol(B_{\rho}(y))\Big\} \\
        &\supseteq \bigcap_{y \in B_r(x) \cap \Z^d} \Big\{ |B_{\rho}(y) \cap \calV| \le (1+\sigma/10)\vol(B_{\rho}(y))\Big\}.
        \end{aligned}
    \end{equation}
    where in the second row we discretised the space to obtain a finite intersection at the cost of reducing the error to $\sigma/10$.
    It is then immediate from Chernoff bounds and a union bound over $y$ that
    \begin{equation}\label{eq:dense-geometric-aball-bound}
        \pr(\calA_{\mathrm{ball}}) \ge 1 - \delta/3.
    \end{equation}
    Now by a union bound on their complements in~\eqref{eq:dense-geometric-adense-bound}, \eqref{eq:dense-geometric-avert-bound}, \eqref{eq:dense-geometric-aball-bound}, the intersection $\calA_{\mathrm{dense}}\cap\calA_{\mathrm{box}}\cap\calA_{\mathrm{ball}}$ occurs with probability at least $1-\delta$. Assume the intersection of the three events occur and let $y \in B_r(x)$. Since $\calA_{\mathrm{dense}}$ occurs, $|B_{\rho}(y)^\star \setminus \calC_\infty^\star| \le 2\sigma^3|B_{\rho}(y)^\star|$. Since $\calA_{\mathrm{box}}$ also occurs in \eqref{eq:dense-geometric-adense}, taking $A = B_{\rho}(y)^\star \setminus \calC_\infty^\star$ yields
    \begin{align*}
        |B_{\rho}(y) \cap \calH_\infty| 
        &= \Big|B_{\rho}(y)\cap \bigcup_{z \in \calC_\infty^\star} \calV[S_z] \Big| \ge \Big|\bigcup_{z \in B_{\rho}(y)^\star \cap \calC_\infty^\star}\calV[S_z]\Big|
        \ge (1-\sigma/5)\vol(B_{\rho}(y)).
    \end{align*}
    Finally, since $\calA_{\mathrm{ball}}$ occurs in \eqref{eq:a-ball}, it follows that
    \[
        \frac{|B_{\rho}(y)\cap \calH_\infty|}{|B_{\rho}(y) \cap \calV|} \ge \frac{(1-\sigma/5)\vol(B_{\rho}(y))}{(1+\sigma/5)\vol(B_{\rho}(y))} \ge 1-\sigma.
    \]
    Hence, the event in \eqref{eq:dense-geometric-density} is implied by the intersection of $\calA_{\mathrm{dense}}\cap\calA_{\mathrm{box}}\cap\calA_{\mathrm{ball}}$, which finishes the proof of~\eqref{eq:dense-geometric-density}.
    \end{proof}